%% file: goodformal2.tex
\documentclass[12pt]{article}
\usepackage{amsmath, amsthm, amssymb, stmaryrd, fullpage}
\usepackage{color}
\usepackage{graphicx}

\newtheorem{theorem}{Theorem}[subsection]
\numberwithin{equation}{theorem}
\newtheorem{lemma}[theorem]{Lemma}
\newtheorem{cor}[theorem]{Corollary}
\newtheorem{prop}[theorem]{Proposition}
\theoremstyle{definition}
\newtheorem{defn}[theorem]{Definition}
\newtheorem{example}[theorem]{Example}
\newtheorem{remark}[theorem]{Remark}
\newtheorem{hypothesis}[theorem]{Hypothesis}
\newtheorem{notation}[theorem]{Notation}

\newcommand{\calD}{\mathcal{D}}
\newcommand{\calE}{\mathcal{E}}
\newcommand{\calF}{\mathcal{F}}
\newcommand{\calI}{\mathcal{I}}
\newcommand{\calO}{\mathcal{O}}
\newcommand{\calR}{\mathcal{R}}

\newcommand{\gothm}{\mathfrak{m}}
\newcommand{\gotho}{\mathfrak{o}}
\newcommand{\gothp}{\mathfrak{p}}
\newcommand{\gothq}{\mathfrak{q}}

\newcommand{\CC}{\mathbb{C}}
\newcommand{\DD}{\mathbb{D}}
\newcommand{\PP}{\mathbb{P}}
\newcommand{\QQ}{\mathbb{Q}}
\newcommand{\RR}{\mathbb{R}}
\newcommand{\ZZ}{\mathbb{Z}}

\newcommand{\del}{\partial}

\newcommand{\be}{\mathbf{e}}
\newcommand{\bv}{\mathbf{v}}

\newcommand{\dual}{\vee}

\DeclareMathOperator{\Aut}{Aut}
\DeclareMathOperator{\codim}{codim}
\DeclareMathOperator{\divis}{div}
\DeclareMathOperator{\End}{End}
\DeclareMathOperator{\Frac}{Frac}
\DeclareMathOperator{\height}{height}
\DeclareMathOperator{\Irr}{Irr}
\DeclareMathOperator{\rank}{rank}
\DeclareMathOperator{\ratrank}{ratrank}
\DeclareMathOperator{\Real}{Re}
\DeclareMathOperator{\Reg}{Reg}
\DeclareMathOperator{\RZ}{RZ}
\DeclareMathOperator{\semis}{ss}
\DeclareMathOperator{\Spec}{Spec}
\DeclareMathOperator{\Supp}{Supp}
\DeclareMathOperator{\trdefect}{trdefect}
\DeclareMathOperator{\trdeg}{trdeg}

\begin{document}

\title{Good formal structures for flat meromorphic connections, II:
Excellent schemes}
\author{Kiran S. Kedlaya}
\date{July 30, 2010}
\maketitle

\begin{abstract}
Given a flat meromorphic connection 
on an excellent scheme over a field of characteristic zero,
we prove existence of good formal structures after blowing up;
this extends a theorem of Mochizuki for algebraic varieties.
The argument combines a numerical criterion for good formal structures
from a previous paper, with an analysis based on the geometry of
an associated valuation space (Riemann-Zariski space).
We obtain a similar result
over the formal completion of an excellent scheme along a closed subscheme.
If we replace the excellent scheme by a complex analytic variety,
we obtain a similar but weaker result in which the blowup can only
be constructed in a suitably small neighborhood of a prescribed point.
\end{abstract}

\section*{Introduction}

The Hukuhara-Levelt-Turrittin decomposition theorem
gives a classification of differential modules over the field
$\CC((z))$ of formal Laurent series resembling the decomposition of a
finite-dimensional vector space equipped with a linear endomorphism
into generalized eigenspaces. It implies that after adjoining
a suitable root of $z$, one can express any differential module as a successive
extension of one-dimensional modules.
This classification serves as the basis for
the asymptotic analysis of meromorphic connections around a 
(not necessarily regular) singular point.
In particular, it leads to a coherent description of the \emph{Stokes phenomenon},
i.e., the fact that the asymptotic growth of horizontal sections 
near a singularity must be described using different asymptotic series depending
on the direction along which one approaches the singularity.
(See \cite{varadarajan} for a beautiful exposition of this material.)

In our previous paper \cite{kedlaya-goodformal1}, we gave an
analogue of the Hukuhara-Levelt-Turrittin decomposition
for \emph{irregular} flat formal meromorphic connections on complex
analytic or algebraic surfaces. 
(The regular case is already well understood in all dimensions,
by work of Deligne \cite{deligne}.)
The result \cite[Theorem~6.4.1]{kedlaya-goodformal1} states that 
given a connection, one can find a blowup of its underlying space
and a cover of that blowup ramified along the pole locus of the connection,
such that after passing to 
the formal completion at any point of the cover, the connection
admits a \emph{good decomposition} in the sense of Malgrange
\cite[\S 3.2]{malgrange-reseau}.
This implies that one gets
(formally at each point) a successive extension of connections of rank 1;
one also has some control over the pole loci of these connections.
The precise statement had been conjectured by Sabbah 
\cite[Conjecture~2.5.1]{sabbah}, and was proved in the algebraic case by
Mochizuki \cite[Theorem~1.1]{mochizuki}. The methods of
\cite{mochizuki} and \cite{kedlaya-goodformal1} are quite different;
Mochizuki uses reduction to positive characteristic and some study of
$p$-curvatures, whereas we use properties of differential modules
over one-dimensional nonarchimedean analytic spaces.

The purpose of this paper
is to extend our previous theorem from surfaces to
complex analytic or algebraic varieties of arbitrary dimension.
Most of the hard work concerning differential modules over
nonarchimedean analytic spaces was already carried out in
\cite{kedlaya-goodformal1}; consequently, this paper consists largely
of arguments of a more traditional algebro-geometric nature.
As in \cite{kedlaya-goodformal1},
we do not discuss asymptotic analysis or the Stokes
phenomenon; these have been treated in the two-dimensional case by Sabbah
\cite{sabbah} (building on work of Majima \cite{majima}), 
and one expects the higher-dimensional case to behave similarly.

The paper divides roughly into three parts.
In the remainder of this introduction, we describe the contents of these
parts in more detail, then conclude with some remarks about what remains
for a subsequent paper.

\subsection{Birational geometry}

In the first part of the paper (\S~\ref{sec:prelim}--\ref{sec:valuation}),
we gather some standard tools from the birational geometry of schemes.
One of these is Grothendieck's notion of an \emph{excellent ring}, which
encompasses rings of finite type over a field, local rings of 
complex analytic varieties, and their formal completions.
Using excellent rings and schemes, we can give a unified treatment of
differential modules in both the algebraic and analytic categories,
without having to keep track of formal completions.

Another key tool we introduce is the theory of Krull valuations
and Riemann-Zariski spaces. The compactness of the latter will be the key
to translating a local decomposition theorem for flat meromorphic connections
into a global result.

\subsection{Local structure theory}

In the second part of the paper
(\S~\ref{sec:excellent}--\ref{sec:good formal define}), we 
continue the local study of differential modules
from \cite{kedlaya-goodformal1}.
(Note that this part of the paper can be read almost entirely independently from the first
part, except for one reference to the definition of an excellent ring.)
We first define the notion of a \emph{nondegenerate differential ring},
which includes global coordinate rings of smooth algebraic varieties,
local rings of smooth complex analytic varieties, and 
formal completions of these.
We prove an equivalence between different notions of good
formal structures, which is needed to ensure that our results really do
address a generalization of Sabbah's conjecture.
We then collect some descent arguments to transfer
good formal structures between a power series ring over
a domain and the corresponding series ring over the fraction field of that
domain. 
We finally translate the local algebraic calculations into geometric
consequences for differential modules on nondegenerate differential schemes
and complex analytic varieties. 

It simplifies matters greatly that the numerical criterion
for good formal structures established in the first part of
\cite{kedlaya-goodformal1} is not limited to surfaces, but rather
applies in any dimension.
We incorporate that result
\cite[Theorem~4.4.2]{kedlaya-goodformal1} in a more geometric formulation
(see Theorem~\ref{T:criterion} and Proposition~\ref{P:geom numerical}):
a connection on a nondegenerate differential scheme admits a good formal
structure precisely at points where the irregularity is measured by
a suitable Cartier divisor. In other words, there exist a closed subscheme
(the \emph{turning locus}) and a Cartier divisor defined away
from the turning locus (the \emph{irregularity divisor}) such that
for any divisorial valuation not supported entirely in the turning locus,
the irregularity of the connection along that valuation equals the
multiplicity of the irregularity divisor along that valuation.

\subsection{Valuation-theoretic analysis and global results}

In the third part of the paper 
(\S~\ref{sec:berkovich}--\ref{sec:good formal}),
we attack the higher-dimensional analogue of the
aforementioned 
conjecture of Sabbah  \cite[Conjecture~2.5.1]{sabbah}
concerning good formal structures for connections on surfaces.
Before discussing the techniques used, let us recall briefly
how Sabbah's original conjecture was resolved in \cite{kedlaya-goodformal1},
and why the method used there is not suitable for the higher-dimensional case.

As noted earlier, the first part of \cite{kedlaya-goodformal1} provides a
numerical criterion for the existence of good formal structures.
In the second part of \cite{kedlaya-goodformal1}, it is verified that the
numerical criterion can be satisfied on surfaces after suitable blowing up.
This verification involves a combinatorial analysis of the variation of
irregularity on a certain space of valuations; that space is essentially
an infinitely ramified tree. (More precisely, it is a
one-dimensional nonarchimedean analytic space in the sense of
Berkovich \cite{berkovich}.) Copying this analysis directly 
in a higher-dimensional setting involves replacing the tree by a
higher-dimensional polyhedral complex whose geometry is extremely difficult
to describe; it seems difficult to simulate on such spaces
the elementary arguments concerning convex functions which appear in
\cite[\S 5]{kedlaya-goodformal1}.

We instead take an approach more in the spirit of
birational geometry (after Zariski). 
Given a connection on a nondegenerate differential scheme,
we seek to construct a blowup on which the turning locus is empty.
To do this, it suffices to check that 
for each centered valuation on the scheme, there is a blowup on which the
turning locus misses the center of the valuation. The same blowup then satisfies
the same condition for all valuations in some neighborhood of the given
valuation in the Riemann-Zariski space of the base scheme. 
Since the Riemann-Zariski space is quasicompact, there are finitely many
blowups which together eliminate the turning locus; taking a single
blowup which dominates them all achieves the desired result.

The obstruction in executing this approach is a standard bugbear
in birational geometry: it is very difficult to classify valuations
on schemes of dimension greater than 2. We overcome this difficulty
using a new idea, drawn
from our work on semistable reduction for overconvergent $F$-isocrystals
\cite{kedlaya-part1, kedlaya-part2, kedlaya-part3, kedlaya-part4},
and from Temkin's proof of inseparable local uniformization
for function fields in positive characteristic \cite{temkin-unif}.
The idea is to quantify the difficulty of describing a valuation in local 
coordinates using a numerical invariant called the
\emph{transcendence defect}.
A valuation of transcendence defect zero (i.e., an
\emph{Abhyankar valuation}) can be described completely in local coordinates.
A valuation of positive transcendence defect cannot be so described, but it
can be given a good \emph{relative} description
in terms of the Berkovich open unit disc over
a complete field of lower transcendence defect. This 
constitutes a valuation-theoretic formulation of the standard
algebro-geometric technique of fibering a variety in curves.

By returning the argument to the study of Berkovich discs,
we forge a much closer link with the combinatorial
analysis in \cite[\S 5]{kedlaya-goodformal1} than may have been evident
at the start of the discussion. One apparent difference from 
\cite{kedlaya-goodformal1} is that we are now forced
to consider discs over complete fields which are not discretely valued,
so we need some more detailed analysis of differential modules on Berkovich
discs than was used in \cite{kedlaya-goodformal1}.
However, this difference is ultimately illusory: the analysis
in question (from our book on $p$-adic differential
equations \cite{kedlaya-course}) is already used heavily in our
joint paper with Xiao on differential modules on nonarchimedean
polyannuli \cite{kedlaya-xiao}, on which the first part of
\cite{kedlaya-goodformal1} is heavily dependent.

In any case, using this fibration technique, we obtain
an analogue of the Hukuhara-Levelt-Turrittin 
decomposition for a flat meromorphic connection on an integral 
nondegenerate differential
scheme,
after blowing up in a manner dictated by an initial choice of a valuation 
on the scheme (Theorem~\ref{T:higher HLT}).
As noted above, thanks to the quasicompactness of Riemann-Zariski spaces,
this resolves a form of Sabbah's conjecture applicable to flat meromorphic
connections on any nondegenerate differential scheme (Theorem~\ref{T:global}).
When restricted to the case of an algebraic variety, this result
reproduces a theorem of Mochizuki \cite[Theorem~19.5]{mochizuki2},
which was proved using a sophisticated combination of
algebraic and analytic methods.

\subsection{Further remarks}

Using the aforementioned theorem, we also resolve the
higher-dimensional analogue of 
Sabbah's conjecture for \emph{formal} flat meromorphic
connections on excellent schemes (Theorem~\ref{T:global formal1});
this case is not covered by Mochizuki's results even in the case of
a formal completion of an algebraic variety.
We obtain a similar result for complex analytic varieties
(Theorem~\ref{T:global formal2}), but it is somewhat weaker:
in the analytic case,
we only obtain blowups producing good formal
structures which are \emph{locally} defined. That is,
the local blowups need not patch together to give a global blowup.
To eliminate this defect, 
one needs a more quantitative form of Theorem~\ref{T:global},
in which one produces a blowup which is in some sense \emph{functorial}.
This functoriality is meant in the sense of functorial resolution of
singularities for algebraic varieties, where the functoriality is defined
with
respect to smooth morphisms; when working with excellent schemes, one should
instead allow morphisms which are
\emph{regular} (flat with geometrically regular fibres).
We plan to address this point in a subsequent paper.

We mention in passing that while the valuation-theoretic fibration
argument described above is not original to this paper,
its prior use has been somewhat limited.
We suspect that there are additional problems susceptible to this technique,
e.g., in the valuation-theoretic study of plurisubharmonic singularities
\cite{boucksom-favre-jonsson}.

\subsection*{Acknowledgments}
Thanks to Bernard Teissier, Michael Temkin, and Liang Xiao
for helpful discussions.
Financial support was provided by NSF CAREER grant DMS-0545904,
DARPA grant HR0011-09-1-0048, MIT (NEC Fund,
Cecil and Ida Green Career Development Professorship),
and the Institute for Advanced Study (NSF grant DMS-0635607, James D. Wolfensohn Fund).

\section{Preliminaries from birational geometry}
\label{sec:prelim}

We begin by introducing some notions from birational geometry,
notably including Grothendieck's definition of excellent schemes.

\setcounter{theorem}{0}
\begin{notation}
For $X$ an integral separated scheme, let $K(X)$ denote the function field of $X$.
\end{notation}

\subsection{Flatification}

\begin{defn}
Let $f: Y \to X$ be a morphism of integral separated schemes. We say 
$f$ is \emph{dominant}
if the image of $f$ is dense in $X$; it is equivalent to require that
the generic
point of $Y$ must map onto the generic point of $X$.
We say $f$ is \emph{birational} if there exists an open dense subscheme
$U$ of $X$ such that the base change of $f$ to $U$ is an isomorphism
$Y \times_X U \to U$; this implies that $f$ is dominant.
We say $f$ is a \emph{modification} (of $X$) if it is proper and birational.
\end{defn}

\begin{defn}
Let $f: Y \to X$ be a modification of an integral separated scheme $X$, 
and let $g: Z \to X$ be 
a dominant morphism.
Let $U$ be an open dense subscheme of $X$ over which $f$ is an isomorphism.
The \emph{proper transform} of $g$ under $f$
is defined as the morphism $W \to Y$, where $W$ is the Zariski
closure of $Z \times_X U$ in $Z \times_X Y$;
this does not depend on the choice of $U$.
\end{defn}

We will use the following special case of 
Raynaud-Gruson flatification \cite[premi\`ere partie,
\S 5.2]{raynaud-gruson}.

\begin{theorem} \label{T:flatification}
Let $g: X \to S$ be a dominant
morphism of finite presentation of finite-dimensional
noetherian integral separated schemes. 
Then there exists a modification $f: T \to S$ such that
the proper transform of $g$ under $f$
is a flat morphism.
\end{theorem}

\subsection{Excellent rings and schemes}

The class of excellent schemes was introduced by Grothendieck
\cite[\S 7.8]{ega4-2} in order to capture 
the sort of algebro-geometric objects that occur most commonly
in practice, while excluding some pathological examples that appear
in the category of locally noetherian schemes.
The exact definition is less important than the stability of excellence
under some natural operations; the impatient reader may wish to
skip immediately to Proposition~\ref{P:excellent}.
On the other hand, the reader interested in more details
may consult either \cite[\S 7.8]{ega4-2} or
\cite[\S 34]{matsumura-alg}.

\begin{defn}
A morphism of schemes is \emph{regular} if it is flat with geometrically
regular fibres. A ring $A$ is a \emph{G-ring} if for
any prime ideal $\gothp$ of $A$, the morphism
$\Spec(\widehat{A}_\gothp) \to \Spec(A_\gothp)$ is regular.
(Here $\widehat{A}_\gothp$ denotes the completion of the local ring 
$A_\gothp$ with respect to its maximal ideal $\gothp A_\gothp$.)
\end{defn}

\begin{defn}
The \emph{regular locus} of a locally noetherian 
scheme $X$, denoted $\Reg(X)$, is the set of points $x \in X$
for which the local ring $\calO_{X,x}$ of $X$ at $x$ is regular.
A noetherian 
ring $A$ is 
\emph{J-1} if $\Reg(\Spec(A))$ is open in $A$.
We say $A$ is \emph{J-2} if every finitely generated $A$-algebra 
is J-1;
it suffices to check this condition for finite $A$-algebras
\cite[Theorem~73]{matsumura-alg}.
\end{defn}

\begin{defn}
A ring $A$ is \emph{catenary} if 
for any prime ideals $\gothp \subseteq \gothq$ in $A$,
all maximal chains of prime ideals from $\gothp$ to $\gothq$
have the same finite length. (The finiteness of the length of each maximal
chain is automatic if $A$ is noetherian.)
A ring $A$ is \emph{universally catenary} if any finitely generated $A$-algebra
is catenary.
\end{defn}

\begin{defn} \label{D:quasi-excellent}
A ring $A$ is \emph{quasi-excellent} if it is noetherian, a G-ring,
and J-2. A quasi-excellent ring is \emph{excellent} if it is also
universally catenary.
A scheme is \emph{(quasi-)excellent} if it is locally noetherian and
covered by open subsets isomorphic to the spectra of (quasi-)excellent rings. 
Note that an affine scheme
is (quasi-)excellent if and only if its coordinate ring is.
\end{defn}

As suggested earlier, the class of excellent rings is broad enough to
cover most typical cases of interest in algebraic geometry.

\begin{prop} \label{P:excellent}
The class of (quasi-)excellent rings is stable under formations of
localizations and finitely generated algebras (including quotients).
Moreover, a noetherian ring is (quasi-)ex\-cellent if and only if its maximal reduced
quotient is.
\end{prop}
\begin{proof}
See \cite[Definition~34.A]{matsumura-alg}.
\end{proof}
\begin{cor}
Any scheme locally of finite type over a field is excellent.
\end{cor}
\begin{proof}
A field $k$ is evidently noetherian, a G-ring, and J-2. It is also universally
catenary because for any finitely generated integral $k$-algebra $A$,
the dimension of $A$ equals the transcendence degree of $\Frac(A)$ over $k$,
by Noether normalization (see \cite[\S 8.2.1]{eisenbud}).
Hence $k$ is excellent. By Proposition~\ref{P:excellent}, any finitely
generated $k$-algebra is also excellent. This proves the claim.
\end{proof}
\begin{cor}
Any modification of a (quasi-)excellent scheme
is again (quasi-)excellent.
\end{cor}

\begin{remark}  \label{R:qe props}
The classes of excellent and quasi-excellent rings enjoy many additional
properties which we will not be using. For completeness, we mention
a few of these.
See \cite[Definition~34.A]{matsumura-alg} for omitted references.
\begin{itemize}
\item
If a local ring is noetherian and a G-ring, then it is J-2 and hence
quasi-excellent.

\item
Any Dedekind domain of characteristic $0$, such as $\ZZ$, is excellent.
However, this  fails in positive characteristic \cite[34.B]{matsumura-alg}.

\item
Any quasi-excellent ring is a \emph{Nagata ring}, i.e., a noetherian
ring which is universally Japanese.
(A ring $A$ is \emph{universally Japanese} if 
for any finitely generated integral $A$-algebra $B$
and any finite extension $L$ of $\Frac(B)$, the integral closure of
$B$ in $L$ is a finite $B$-module.)

\end{itemize}
\end{remark}

\begin{remark}
It has been recently shown by Gabber using
a weak form of local uniformization (unpublished)
that excellence is preserved under completion with respect to an ideal.
This answers an old question of Grothendieck
\cite[Remarque~7.4.8]{ega4-2}; the special case for 
excellent $\QQ$-algebras of finite dimension
had been established previously by Rotthaus \cite{rotthaus}.
However, we will not need Gabber's result because we will establish excellence
of the rings we consider using derivations; see Lemma~\ref{L:nondegenerate}.
\end{remark}

\subsection{Resolution of singularities for quasi-excellent schemes}

Upon introducing the class of quasi-excellent schemes,
Grothendieck showed that it is in some sense the maximal class of
schemes for which resolution of singularities is possible.
\begin{prop}[Grothendieck]
Let $X$ be a locally noetherian scheme. Suppose that for any integral separated
scheme $Y$ finite over $X$, there exists a modification
$f: Z \to Y$ with $Z$ regular. Then $X$ is quasi-excellent.
\end{prop}
\begin{proof}
See \cite[Proposition~7.9.5]{ega4-2}.
\end{proof}

Grothendieck then suggested that Hironaka's proof of resolution of 
singularities for varieties over a field of characteristic zero
could be adapted to check that any quasi-excellent scheme
over a field of characteristic zero admits a resolution of 
singularities. To the best of our knowledge, this claim was never verified.
However, an analogous statement has
been established more recently by Temkin,
using an alternative proof of Hironaka's theorem due to Bierstone and Milman.

\begin{defn}
A \emph{regular pair} is a pair $(X,Z)$, in which $X$ is a regular scheme,
and $Z$ is a closed subscheme of $X$ which is a \emph{normal crossings
divisor}. The latter means that \'etale locally, $Z$ is the zero
locus on $X$ of a regular function of the form $t_1^{e_1} \cdots t_n^{e_n}$, for
$t_1,\dots,t_n$ a regular sequence of parameters and $e_1,\dots,e_n$
some nonnegative integers.
\end{defn}

\begin{theorem} \label{T:desing1}
For every noetherian quasi-excellent integral scheme $X$ over $\Spec(\QQ)$, 
and every closed proper subscheme $Z$ of $X$, there
exists a modification $f: Y \to X$ such that $(Y, f^{-1}(Z))$
is a regular pair.
\end{theorem}
\begin{proof}
See \cite[Theorem~1.1]{temkin-excellent}.
\end{proof}

\begin{remark}
One can further ask for a desingularization procedure which is
\emph{functorial} for regular morphisms.
This question has been addressed by
Bierstone, Milman, and Temkin \cite{bmt, temkin1, temkin2}.
We will need this in a subsequent paper; see Remark~\ref{R:functorial}.
\end{remark}

\subsection{Alterations}

It will be convenient to use a slightly larger class of morphisms
than just modifications.

\begin{defn}
An \emph{alteration} of an integral separated scheme $X$
is a proper, dominant, generically finite 
morphism $f: Y \to X$ with $Y$ integral.
If $X$ is a scheme over $\Spec(\QQ)$,
this implies that there is an open dense subscheme
$U$ of $X$ such that $Y \times_X U \to U$ is finite \'etale.
If $X$ is excellent, then so is $Y$ by Proposition~\ref{P:excellent}.
\end{defn}

\begin{remark}
Alterations were introduced by de Jong to give a weak form of resolution
of singularities in positive characteristic and for arithmetic schemes;
see \cite[Theorem~4.1]{dejong}. Since here we only consider schemes
over a field of characteristic 0, this benefit is not relevant for us;
the reason we consider alterations is because the valuation-theoretic
arguments of \S~\ref{sec:local analysis} are easier to state in terms
of alterations than modifications. Otherwise, one must work not just
with the Berkovich unit disc, but also with more general one-dimensional
analytic spaces, as in Temkin's proof of inseparable local uniformization
\cite{temkin-unif}.
\end{remark}

\begin{lemma} \label{L:flatify}
Let $g: Z \to X$ be an alteration of a finite-dimensional noetherian
integral separated scheme $X$. Then there exists a modification
$f: Y \to X$ such that the proper transform of $g$ under $f$ is finite flat.
\end{lemma}
\begin{proof}
By Theorem~\ref{T:flatification},
we can choose $f$ so that the proper transform $g'$ 
of $g$ under $f$ is flat.
Since $g'$ is flat, it has equidimensional fibres by
\cite[Theorem~19]{matsumura-alg}.
Hence $g'$ is locally of finite type with finite fibres, i.e., $g'$
is quasifinite.
Since any proper quasifinite morphism is finite by Zariski's main theorem
\cite[Th\'eor\`eme~8.11.1]{ega4-3},
we conclude that $g'$ is finite.
\end{proof}

\subsection{Complex analytic spaces}
\label{subsec:complex}

We formally introduce the category of complex analytic spaces,
and the notion of a modification in that category. One can also define
alterations of complex analytic spaces, but we will not need them here.

\begin{defn}
For $X$ a locally ringed space, a \emph{closed subspace} of $X$ is a subset of the
form $\Supp(\calO_X/\calI)$ for some ideal sheaf $\calI$ on $X$, equipped with
the restriction of the sheaf $\calO_X/\calI$.
\end{defn}

\begin{defn}
A \emph{complex analytic space} is a locally ringed space $X$ which is locally isomorphic
to a closed subspace of an affine space carrying the sheaf of holomorphic functions.
We define morphisms of complex analytic spaces, and closed subspaces of a complex analytic space,
using the corresponding definitions
of the underlying locally ringed spaces.
\end{defn}

\begin{defn}
A \emph{modification} of irreducible reduced separated complex analytic spaces
is a morphism $f: Y \to X$ which is proper (as a map of topological spaces)
and surjective, and which
restricts to an isomorphism on the complement of 
a closed subspace of $X$. For instance,
the analytification of a modification of complex algebraic varieties
is again a modification; the hard part of this statement is the fact
that algebraic properness implies topological properness, for which 
see \cite[Expos\'e~XII, Proposition~3.2]{sga1}.
\end{defn}

\begin{defn}
In the category of complex analytic spaces, a \emph{regular pair} 
will denote a pair $(X,Z)$ in which $X$ is a complex analytic space,
$Z$ is a closed subspace of $X$,
and for each $x \in X$, $\calO_{X,x}$ is regular (i.e., $X$ is smooth at $x$)
and the ideal sheaf defining $Z$ defines 
a normal crossings divisor on $\Spec(\calO_{X,x})$.
\end{defn}

The relevant form of resolution of singularities for complex analytic
spaces is due to Aroca, Hironaka, and Vicente
\cite{ahv1, ahv2}.

\begin{theorem} \label{T:desing complex}
For every irreducible reduced separated complex analytic space $X$ and every closed proper subspace $Z$ of $X$, there
exists a modification $f: Y \to X$ such that $(Y, f^{-1}(Z))$
is a regular pair.
\end{theorem}

\section{Valuation theory}
\label{sec:valuation}

We need some basic notions from the classical theory
of Krull valuations.
Our blanket reference for valuation theory is \cite{vaquie}.

\subsection{Krull valuations}

\begin{defn} \label{D:val}
A \emph{valuation} (or \emph{Krull valuation})
on a field $F$ with values in a totally ordered group $\Gamma$ 
is a function $v: F \to \Gamma \cup \{+\infty\}$ satisfying the
following conditions.
\begin{enumerate}
\item[(a)] For $x,y \in F$, $v(xy) = v(x) + v(y)$.
\item[(b)] For $x,y \in F$, $v(x+y) \geq \min\{v(x), v(y)\}$.
\item[(c)] We have $v(1) = 0$ and $v(0) = +\infty$.
\end{enumerate}
We say $v$ is \emph{trivial} if $v(x) = 0$ for all $x \in F^\times$.
A \emph{real valuation} is a Krull valuation with $\Gamma = \RR$.

We say the valuations $v_1, v_2$ are \emph{equivalent} if 
for all $x,y \in F$,
\[
v_1(x) \geq v_1(y) \qquad \Longleftrightarrow \qquad
v_2(x) \geq v_2(y).
\]
The isomorphism classes of the following objects associated to $v$ are equivalence invariants:
\begin{align*}
\mbox{value group:} & \quad \Gamma_v = v(F^\times) \\
\mbox{valuation ring:} & \quad \gotho_v = \{x \in F: v(x) \geq 0\} \\
\mbox{maximal ideal:} & \quad \gothm_v = \{x \in F: v(x) > 0 \} \\
\mbox{residue field:} & \quad \kappa_v = \gotho_v / \gothm_v.
\end{align*}
Note that the equivalence classes of valuations on $F$ are in bijection
with the \emph{valuation rings} of the field $F$, i.e., the subrings $\gotho$
of $F$ such that for any $x \in F^\times$, at least one of
$x$ or $x^{-1}$ belongs to $\gotho$. (Given a valuation ring $\gotho$,
the natural map $v: F \to (F^\times/\gotho^\times) \cup \{+\infty\}$ 
sending $0$ to $+\infty$ is a valuation with $\gotho_v = \gotho$.)
\end{defn}

\begin{defn}
Let $R$ be an integral domain, and let $v$ be a valuation on $\Frac(R)$.
We say $v$ is \emph{centered on $R$} if $v$ takes nonnegative values on $R$,
i.e., if $R \subseteq \gotho_v$.
In this case, $R \cap \gothm_v$ is a prime ideal of $R$, called the
\emph{center} of $v$ on $R$.

Similarly, let $X$ be an integral separated scheme with function field $K(X)$,
and let $v$ be a valuation on $K(X)$.
The \emph{center} of $v$ on $X$ is the set of $x \in X$ with
$\gotho_{X,x} \subseteq \gotho_v$; it is either empty or an irreducible
closed subset of $X$ (see \cite[Proposition~6.2]{vaquie}, keeping in mind
that the hypothesis of separatedness is needed to reduce to the affine case).
In the latter case, we say $v$ is \emph{centered on $X$},
and we refer to the generic point of the center as the \emph{generic center}
of $v$. In fact, we will refer so often to valuations on $K(X)$ centered on $X$
that we will simply call them \emph{centered valuations on $X$}.

If $X = \Spec R$, then $v$ is centered on $X$ if and only if
$v$ is centered on $R$, in which case the center of $v$ on $R$ is the generic
center of $v$ on $X$.
\end{defn}

\begin{lemma} \label{L:proper}
\begin{enumerate}
\item[(a)]
Let $F$ be a field and let $v$ be a valuation on $F$.
Then for any field $E$ containing $F$, there exists an extension of $v$
to a valuation on $E$.
\item[(b)]
Let $f: Y \to X$ be a proper dominant
morphism of integral separated schemes. For any centered valuation $v$ on $X$,
any extension of $v$ to $K(Y)$ is centered on $Y$.
(Such an extension exists by (a).)
\end{enumerate}
\end{lemma}
\begin{proof}
We may deduce (a) by applying to $\gotho_v$
the fact that every local subring of $E$ is dominated by a valuation
ring \cite[\S 1]{vaquie}.
Part (b) is the valuative criterion for properness; see
\cite[Th\'eor\`eme~7.3.8]{ega2}.
\end{proof}

\subsection{Numerical invariants}
\label{subsec:numerical}

Here are a few basic numerical invariants attached to valuations.

\begin{defn} \label{D:composite}
Let $v$ be a valuation on a field $F$.
An \emph{isolated subgroup} (or \emph{convex subgroup})
of $\Gamma_v$ is a subgroup
$\Gamma'$ 
such that for all $\alpha \in \Gamma', \beta \in \Gamma_v$ with
$-\alpha \leq \beta \leq \alpha$, we have $\beta \in \Gamma'$.
In this case, the quotient group $\Gamma_v/\Gamma'$ inherits a total ordering
from $\Gamma_v$, and we obtain a valuation $v'$ on $F$ with value group
$\Gamma_v/\Gamma'$ by projection. We also obtain a valuation
$\overline{v}$ on $\kappa_{v'}$ with value group $\Gamma'$.
The valuation $v$ is said to be a \emph{composite} of
$v'$ and $\overline{v}$.

Let $\ratrank(v) = \dim_{\QQ} (\Gamma_v \otimes_{\ZZ} \QQ)$ denote
the \emph{rational rank} of $v$.
For $k$ a subfield of $\kappa_v$, 
let $\trdeg(\kappa_v/k)$ denote the \emph{transcendence degree} of 
$\kappa_v$ over $k$.

The \emph{height} (or \emph{real rank}) of $v$, denoted $\height(v)$, is the
maximum length of a chain of proper isolated subgroups of $\Gamma_v$;
note that $\height(v) \leq \ratrank(v)$
\cite[Proposition~3.5]{vaquie}.
By definition, $\height(v) > 1$ if and only if $\Gamma_v$ admits a nonzero 
proper isolated subgroup, in which case $v$ can be described as a 
composite valuation as above.
On the other hand, 
$\height(v) \leq 1$ if and only if $v$ is equivalent to a real valuation
\cite[Proposition~3.3, Exemple~3]{vaquie}.
\end{defn}

There is a fundamental inequality due to Abhyankar (generalizing
a result of Zariski), which gives rise to an additional numerical invariant
for valuations centered on noetherian schemes. This invariant
quantifies the difficulty
of describing the valuation in local coordinates.

\begin{defn} \label{D:transcendence defect}
Let $X$ be a noetherian integral separated scheme.
Let $v$ be a centered valuation on $X$, with center $Z$.
Define the \emph{transcendence defect} of $v$ as
\[
\trdefect(v) = \codim(Z,X) - \ratrank(v) - \trdeg(\kappa_v/K(Z)).
\]
(The term \emph{defect rank} is used in 
\cite[Remark~2.1.1]{temkin-unif} for the
same quantity.)
We say $v$ is an \emph{Abhyankar valuation} if $\trdefect(v) = 0$.
\end{defn}

\begin{theorem}[Zariski-Abhyankar inequality]  \label{T:Abhyankar}
Let $X$ be a noetherian integral separated scheme.
Then for any centered valuation $v$ on $X$ with center $Z$,
$\trdefect(v) \geq 0$.
Moreover, if equality occurs, then $\Gamma_v$ is a finitely generated
abelian group, and $\kappa_v$ is a finitely generated extension
of $K(Z)$.
\end{theorem}
\begin{proof}
We may reduce immediately to the case where $X = \Spec(R)$ for $R$
a local ring, and $Z$ is the closed point of $X$. 
In this case,
see \cite[Appendix~2, Corollary, p.\ 334]{zariski-samuel}
or \cite[Th\'eor\`eme~9.2]{vaquie}.
\end{proof}

These invariants behave nicely under alterations,
in the following sense.
\begin{lemma}  \label{L:extension}
Let $f: Y \to X$ be an alteration of a noetherian integral separated scheme $X$.
Let $v$ be a centered valuation on $X$.
\begin{enumerate}
\item[(a)]
There exists at least one extension $w$ of $v$ to
a centered valuation on $Y$.
\item[(b)]
For any $w$ as in (a), we have $\height(w) = \height(v)$,
$\ratrank(w) = \ratrank(v)$, and $\trdefect(w) \leq \trdefect(v)$
with equality if $X$ is excellent.
\end{enumerate}
\end{lemma}
\begin{proof}
For (a), apply Lemma~\ref{L:proper}.
For (b), we may check the equality of heights and rational ranks at the level
of the function fields.
To check the inequality of transcendence defects,
let $Z$ be the center of $v$ on $X$, and let $W$ be the center of $w$
on $Y$; then
\begin{align*}
\trdefect(v) - \trdefect(w) &=
\codim(Z,X) - \codim(W,Y) + \trdeg(\kappa_w/K(W)) - 
\trdeg(\kappa_v/K(Z)) \\
&= \codim(Z,X) - \codim(W,Y) + \trdeg(\kappa_w/\kappa_v)
- \trdeg(K(W)/K(Z)).
\end{align*}
We may check that $\trdeg(\kappa_w/\kappa_v) = 0$ 
at the level of function fields;
see \cite[\S 5]{vaquie} for this verification.
After replacing $X$ by the spectrum of a local ring, we may apply
\cite[Th\'eor\`eme~5.5.8]{ega4-2} 
or
\cite[Theorem~23]{matsumura-alg}
to obtain the inequality
\[
\codim(Z,X) + \trdeg(K(Y)/K(X))
\geq \codim(W,Y) + \trdeg(K(W)/K(Z)),
\]
with equality in case $X$ is excellent. Since $f$ is an alteration,
$K(Y)$ is algebraic over $K(X)$ and so $\trdeg(K(Y)/K(X)) = 0$.
This yields the desired comparison.
\end{proof}

Here is another useful property of Abhyankar valuations.
\begin{remark} \label{R:extend valuation}
Let $R$ be a noetherian local ring with completion
$\widehat{R}$. Let $v$ be a centered real valuation
on $R$. Then $v$ extends by continuity to a function $\hat{v}:
\widehat{R} \to \Gamma_v \cup \{+\infty\}$.
This function is in general a \emph{semivaluation}, in that it satisfies
all of the conditions defining a valuation except that 
$\gothp = \hat{v}^{-1}(+\infty)$ is only a prime ideal, not necessarily
the zero ideal.

For instance,
choose $a = a_1 x + a_2x^2 + \cdots \in k \llbracket x \rrbracket$
transcendental over $k(x)$.
Put $R = k[x,y]_{(x,y)}$,
and let $v$ be the restriction of the $x$-adic valuation of
$k \llbracket x \rrbracket$ along the map
$R \to k \llbracket x \rrbracket$ defined by $x \mapsto x,
y \mapsto a$.
Then $v$ is a valuation on $R$, but $\gothp = (y-a) \neq 0$.

On the other hand, suppose $v$ is a real Abhyankar valuation.
Then $\hat{v}$ induces a valuation on $\widehat{R}/\gothp$ 
with the same value group and residue field as $v$. By
the Zariski-Abhyankar inequality, this forces $\dim (\widehat{R}/\gothp)
\geq \dim R$; since $\dim R = \dim \widehat{R}$
\cite[Corollary~10.12]{eisenbud},
this is only possible for $\gothp = 0$.
That is, if $v$ is a real Abhyankar valuation, it extends to
a valuation on $\widehat{R}$.
\end{remark}

\subsection{Riemann-Zariski spaces}

We need a mild generalization of Zariski's original
compactness theorem for spaces of valuations, which we prefer to state
in scheme-theoretic language. See
\cite[\S 2]{temkin-rz} for a much broader generalization.

\begin{defn}
For $R$ a ring, the \emph{patch topology} on 
$\Spec(R)$ is the topology generated by the sets $D(f)
= \{\gothp \in \Spec(R): f \notin \gothp\}$
and their complements. 
Note that for any $f \in R$, the open sets for the patch topology
on $D(f)$ are also open for the patch topology on $R$.
It follows that we can define the patch topology on a scheme $X$
to be the topology generated by the open subsets for the patch topologies
on all open affine subschemes of $X$, and this will agree with the previous
definition for affine schemes. The resulting topology
is evidently finer than the Zariski topology.
\end{defn}

\begin{lemma} \label{L:patch}
Any noetherian scheme is compact for the patch topology.
\end{lemma}
\begin{proof}
By noetherian induction, it suffices to check that if $X$ is a noetherian
scheme such that each closed proper subset of $X$ is compact for the patch topology,
then $X$ is also compact for the patch topology.
Since a noetherian scheme is covered by finitely many affine 
noetherian schemes, each of which has finitely many irreducible
components, 
we may reduce to the case of an irreducible affine noetherian scheme
$\Spec(R)$. 

Given an open cover of $\Spec(R)$ for the patch topology, there must be
an open set covering the generic point. This open set must contain a basic
open set of the form $D(f) \setminus D(g)$ for some $f,g \in R$,
but this only covers the generic point if $D(g)$ is empty. Hence
our open cover includes an open set containing $D(f)$ for some $f \in R$
which is not nilpotent. By hypothesis, the closed set $X \setminus D(f)$
is covered by finitely many opens from the cover, as then is $X$.
\end{proof}

\begin{defn}
Let $X$ be a noetherian integral separated scheme.
The \emph{Riemann-Zariski space} $\RZ(X)$ consists of the
equivalence classes of centered valuations on $X$.
If $X = \Spec(R)$, we also write $\RZ(R)$ instead of $\RZ(X)$.

We may identify $\RZ(X)$ with the inverse limit over modifications
$f: Y \to X$, as follows.
Given $v \in \RZ(X)$, we take the element of the inverse limit
whose component on a modification $f: Y \to X$ 
is the generic center of $v$ on $Y$ (which exists by
Lemma~\ref{L:proper}(a)).
Conversely, given an element of the inverse limit with value $x_Y$
on $Y$, form the direct limit of the local rings $\calO_{Y,x_Y}$; 
this gives a valuation
ring because any $g \in K(X)$ defines a rational map $X \dashrightarrow
\PP^1_\ZZ$, and the Zariski closure of the graph of this rational
map is a modification $W$ of $X$ such that 
one of $g$ or $g^{-1}$ belongs to $\calO_{W,x_W}$.

We equip $\RZ(X)$ with the \emph{Zariski topology}, defined as
the inverse limit of the Zariski topologies on the modifications of $X$.
For any dominant morphism $X \to W$ of noetherian integral schemes, 
we obtain an induced
continuous morphism $\RZ(X) \to \RZ(W)$ using proper transforms.
\end{defn}

\begin{theorem} \label{T:Zariski}
For any noetherian integral separated scheme $X$, the space
$\RZ(X)$ is quasi\-compact.
\end{theorem}
\begin{proof}
For each modification $f: Y \to X$, $Y$ is compact for the patch topology
by Lemma~\ref{L:patch}. If we topologize $\RZ(X)$ with the inverse limit of the
patch topologies, the result is compact by Tikhonov's theorem. The Zariski topology
is coarser than this, so $\RZ(X)$ is quasicompact for the Zariski topology.
(Note that we cannot check this directly 
because an inverse limit of quasicompact topological spaces need not be
quasicompact.)
\end{proof}

\begin{prop} \label{P:open map}
Let $f: Y \to X$ be a morphism of finite-dimensional
noetherian integral separated schemes,
which is dominant and of finite presentation.
Then the map $\RZ(f): \RZ(Y) \to \RZ(X)$ is open.
\end{prop}
\begin{proof}
By Theorem~\ref{T:flatification},
there exists a modification $g: Z \to X$ such that the proper transform
$h: W \to Z$ of $f$ under $g$ is flat. The claim now follows from the fact
that a morphism which is flat and locally of finite presentation is open
\cite[Th\'eor\`eme~2.4.6]{ega4-2}.
\end{proof}

\section{Nondegenerate differential schemes}
\label{sec:excellent}

We now explain how the notion of an excellent scheme 
interacts with derivations,
and with our discussion of good formal structures 
in \cite{kedlaya-goodformal1}.

\setcounter{theorem}{0}
\begin{defn}
Let $R \hookrightarrow S$ be an inclusion of domains,
and let $M$ be a finite $S$-module.
By an \emph{$R$-lattice} in $M$, we mean a finite $R$-submodule $L$ of $M$
such that the induced map $L \otimes_R S \to M$ is surjective.
\end{defn}

\subsection{Nondegenerate differential local rings}

We first introduce a special class of differential local rings.

\begin{defn}
A \emph{differential (local) ring}
is a (local) ring $R$ equipped with an $R$-module $\Delta_R$ acting on
$R$ via derivations, together with a Lie algebra structure on $\Delta_R$
compatible with the Lie bracket on derivations.
\end{defn}

\begin{defn} \label{D:nondegenerate}
Let $(R, \Delta_R)$ be a differential local ring
with maximal ideal $\gothm$ and residue field $\kappa = R/\gothm$.
We say $R$ is \emph{nondegenerate} if it satisfies the following conditions.
\begin{enumerate}
\item[(a)]
The ring $R$ is a regular (hence noetherian) local $\QQ$-algebra.
\item[(b)]
The $R$-module $\Delta_R$ is coherent.
\item[(c)]
For some regular sequence of parameters $x_1,\dots,x_n$ of $R$,
there exists a sequence $\del_1,\dots,\del_n \in \Delta_R$ of derivations of \emph{rational type}
with respect to $x_1,\dots,x_n$. That is, $\del_1,\dots,\del_n$ must commute pairwise, and must
satisfy
\begin{equation} \label{eq:rational type}
\del_i(x_j) = \begin{cases} 1 & (i=j) \\ 0 & (i \neq j).
\end{cases}
\end{equation}
The existence of such derivations for a single regular sequence of parameters implies the same 
for any other regular sequence of parameters; see
Corollary~\ref{C:regular sequence}.
\end{enumerate}
\end{defn}

\begin{remark}
Note that \eqref{eq:rational type} by itself does not force $\del_1,\dots,\del_n$ to
commute pairwise. For instance, if $R = \CC(t)\llbracket x_1,x_2 \rrbracket$,
we can satisfy \eqref{eq:rational type} by taking
\[
\del_1 = \frac{\del}{\del x_1} + \frac{\del}{\del t}, \qquad
\del_2 = \frac{\del}{\del x_2} + t \frac{\del}{\del t}.
\]
\end{remark}

\begin{example} \label{exa:nondegenerate}
The following are all examples of nondegenerate local rings.
\begin{enumerate}
\item[(a)]
Any local ring of a smooth scheme over a field of characteristic 0.
\item[(b)]
Any local ring of a smooth complex analytic space.
\item[(c)]
Any completion of a nondegenerate local ring with respect to a prime ideal. (By 
Lemma~\ref{L:localize complete} below, we may reduce to the case of completion with respect
to the maximal ideal, for which the claim is trivial.)
\end{enumerate}
\end{example}

We insert some frequently invoked remarks concerning the regularity condition.
\begin{remark} \label{R:regular local}
Let $R$ be a regular local ring. Let $\widehat{R}$ be the completion of $R$ with 
respect to its maximal ideal; then the morphism $R \to \widehat{R}$ is faithfully flat
\cite[Theorem~8.14]{matsumura}. 
Moreover, since \'etaleness descends down
faithfully flat quasicompact morphisms of schemes
\cite[Expos\'e~IX, Proposition~4.1]{sga1}, any element of $\widehat{R}$ which is algebraic
over $R$ generates a finite \'etale extension of $R$ within $\widehat{R}$
with the same residue field.
\end{remark}

\begin{lemma} \label{L:localize complete}
Let $(R, \Delta_R)$ be a nondegenerate differential local ring with maximal ideal $\gothm$.
For any prime ideal $\gothq$ contained in $\gothm$, the differential ring
$(R_{\gothq}, \Delta_R \otimes_R R_{\gothq})$ is again nondegenerate.
\end{lemma}
\begin{proof}
By \cite[Theorem~45, Corollary]{matsumura-alg}, $R_{\gothq}$ is regular,
and any regular sequence of parameters of $R$ contains a regular sequence of parameters
for $R_{\gothq}$. From these assertions, the claim is evident.
\end{proof}
The following partly generalizes \cite[Lemma~2.1.3]{kedlaya-goodformal1}.
As the proof is identical, we omit details.
\begin{lemma} \label{L:complete nondegenerate}
Let $(R, \Delta_R)$ be a nondegenerate differential local ring.
Suppose further that for some $i \in \{1,\dots,n\}$,
$R$ is $x_i$-adically complete.
\begin{enumerate}
\item[(a)]
For $e \in \ZZ$,
$x_i \del_i$ acts on $x_i^e R/x_i^{e+1} R$
via multiplication by $e$.
\item[(b)]
The action of $x_i \del_i$ on $x_i R$ is bijective.
\item[(c)]
The kernel $R_i$ of $\del_i$ on $R[x_i^{-1}]$
is contained in $R$ and projects bijectively onto $R/x_i R$. There thus
exists an isomorphism $R \cong R_i \llbracket x_i \rrbracket$
under which $\del_i$ corresponds to $\frac{\del}{\del x_i}$.
\end{enumerate}
\end{lemma}

\begin{cor} \label{C:complete nondegenerate}
Let $(R, \Delta_R)$ be a nondegenerate differential local ring
which is complete with respect to its maximal ideal $\gothm$.
Then there exists an isomorphism
$R \cong k \llbracket x_1,\dots,x_n \rrbracket$ for $k = R/\gothm$,
under which $\del_i$ corresponds to $\frac{\del}{\del x_i}$ for
$i=1,\dots,n$.
\end{cor}

\begin{cor} \label{C:regular sequence}
Let $(R, \Delta_R)$ be a (not necessarily complete)
nondegenerate differential local ring. Then condition (c) 
of Definition~\ref{D:nondegenerate} holds
for any regular sequence of parameters.
\end{cor}
\begin{proof}
Let $y_1,\dots,y_n$ be a second regular sequence of parameters.
Define the matrix $A$ by putting $A_{ij} = \del_i(y_j)$; then
$\det(A)$ is not in the maximal ideal $\gothm$ of $R$, and so is a unit in $R$.
We may then define the derivations
\[
\del'_j = \sum_i (A^{-1})_{ij} \del_i \qquad (j=1,\dots,n),
\]
and these will satisfy
\[
\del'_i(y_j) = \begin{cases} 1 & (i=j) \\ 0 & (i \neq j).
\end{cases}
\]
It remains to check that the $\del'_i$ commute pairwise; for this, it is harmless to pass
to the case where $R$ is complete with respect to $\gothm$.
In this case, by Corollary~\ref{C:complete nondegenerate}, we have an isomorphism
$R \cong k\llbracket x_1,\dots,x_n \rrbracket$ under which each $\del_i$ corresponds to
the formal partial derivative in $x_i$. 
That isomorphism induces an embedding of $k$ into $R$ whose image is killed by 
$\del_1,\dots,\del_n$ and hence also by $\del'_1,\dots,\del'_n$. We thus obtain a second isomorphism
$R \cong k \llbracket y_1,\dots,y_n \rrbracket$ under which each $\del'_i$
corresponds to the formal partial derivative in $y_i$. In particular, these 
commute pairwise, as desired.
\end{proof}

\subsection{Nondegenerate differential schemes}

We now consider more general differential rings and schemes,
following  \cite[\S 1]{kedlaya-goodformal1}.
We introduce the nondegeneracy condition for these and show that it implies
excellence.

\begin{defn}
A \emph{differential scheme}
is a scheme $X$ equipped with a quasicoherent
$\calO_X$-module $\calD_X$ acting on
$\calO_X$ via derivations, together with a Lie algebra structure on
$\calD_X$ compatible
with the Lie bracket on derivations. 
Note that the category of differential affine schemes is equivalent
to the category of differential rings in the obvious fashion.
\end{defn}

\begin{defn} \label{D:nondegenerate scheme}
We say a differential scheme $(X, \calD_X)$ is \emph{nondegenerate}
if $X$ is separated and noetherian of finite Krull dimension,
$\calD_X$ is coherent over $\calO_X$,
and each local ring of $X$ is nondegenerate. We say a differential ring $R$ is 
nondegenerate if $\Spec(R)$ is nondegenerate; this agrees with the previous definition
in the local case.
\end{defn}

\begin{remark} \label{R:extend nondegenerate}
Let $(R, \Delta_R)$ be a differential domain. For several types of ring 
homomorphisms
$f: R \to S$ with $S$ also a domain, 
there is a canonical way to extend the differential structure
on $R$ to a differential structure on $S$, provided we insist that the
differential structure be \emph{saturated}. That is, we equip $S$ with
the subset of $\Delta_{\Frac(R)} \otimes_R S$ consisting of elements which
act as derivations on $\Frac(S)$ preserving $S$.

To be specific, we may perform such a canonical extension for $f$
of the following types.
\begin{enumerate}
\item[(a)]
A generically finite morphism of finite type.
\item[(b)]
A localization.
\item[(c)]
A morphism from $R$ to its completion with respect to some ideal.
\end{enumerate}
If $R$ is nondegenerate, it is clear in cases (b) and (c) that $S$ is also
nondegenerate. In case (a), one can only expect this if $S$ is regular,
in which case it is true but not immediate; this is the content of the
following lemma.
\end{remark}

\begin{lemma} \label{L:extend nondegenerate}
Let $X$ be a nondegenerate differential scheme, and let $f: Y \to X$
be an alteration with $Y$ regular. Then the canonical differential
scheme structure on $Y$ (see Remark~\ref{R:extend nondegenerate})
is again nondegenerate.
\end{lemma}
\begin{proof}
What is needed is to check that every local ring of $Y$ is nondegenerate,
so we may fix
$y \in Y$ and $x = f(y) \in X$. Since the nondegenerate locus is stable
under generization by Lemma~\ref{L:localize complete}, it suffices
to consider cases where $y$ and $x$ have the same codimension, as 
such $y$ are dense in each fibre of $f$.

Choose a regular system of parameters $x_1,\dots,x_n$ for $X$ at $x$.
Since $\calO_{X,x}$ is nondegenerate, we can choose derivations
$\del_1,\dots,\del_n$ acting on some neighborhood $U$ of $x$
which are of rational type with respect to $x_1, \dots, x_n$.
By Lemma~\ref{L:complete nondegenerate},
we can write $\widehat{\calO_{X,x}} \cong k \llbracket x_1,\dots,x_n \rrbracket$
in such a way that $\del_1,\dots,\del_n$ correspond to $\frac{\del}{\del x_1},
\dots, \frac{\del}{\del x_n}$.

Choose a regular system of parameters $y_1,\dots,y_n$ for $Y$ at $y$.
By our choice of $y$, the residue field $\ell$ of $y$
is finite over $k$. We may thus identify
$\widehat{\calO_{Y,y}}$ with $\ell \llbracket y_1,\dots, y_n \rrbracket$
\emph{for the same $n$} in such a way that the morphism
$\widehat{\calO_{X,x}} \to \widehat{\calO_{Y,y}}$ carries $k$ into $\ell$.
Using such an identification, we see (as in the proof of Corollary~\ref{C:regular sequence})
that if we define the matrix $A$ over $\calO_{Y,y}$ by $A_{ij} = \del_i(y_j)$, 
then
put $\del'_j = \sum_i (A^{-1})_{ij} \del_i$, we obtain derivations of rational type
with respect to $y_1,\dots,y_n$. Hence $\calO_{Y,y}$ is also nondegenerate, as desired.
\end{proof}

\begin{lemma} \label{L:nondegenerate}
Let $(X, \calD_X)$ be a nondegenerate differential scheme.
\begin{enumerate}
\item[(a)]
The scheme $X$ is excellent.
\item[(b)]
For each $x \in X$, the differential local ring $\calO_{X,x}$ is simple.
\item[(c)]
If $X$ is integral, then 
the subring of $\Gamma(X, \calO_X)$ killed by the action of $\calD_X$
is a field.
\end{enumerate}
\end{lemma}
\begin{proof}
For (a), see \cite[Theorem~101]{matsumura}.
For (b), note that for $\gothm_{X,x}$ the maximal ideal of $\calO_{X,x}$
and $\kappa_x$ the residue field, the nondegeneracy condition 
forces the pairing
\begin{equation} \label{eq:pairing}
\Delta_{\calO_{X,x}} \times \gothm_{X,x}/\gothm_{X,x}^2 \to \kappa_x
\end{equation}
to be nondegenerate on the right. The differential ring $\calO_{X,x}$
is then simple by \cite[Proposition~1.2.3]{kedlaya-goodformal1}.
For (c), note that the nondegeneracy condition prevents any
$r \in \Gamma(X, \calO_X)$ killed by the action of $\calD_X$ from
belonging to the maximal ideal of any local ring of $X$ unless it vanishes in
the local ring.
\end{proof}
\begin{cor}
Every local ring of an algebraic variety over a field of characteristic
$0$, or of a complex analytic variety, is excellent. Moreover, the completion
of any such ring with respect to any ideal is again excellent.
\end{cor}
\begin{proof}
By Proposition~\ref{P:excellent}, it is enough to check both claims
when the variety in question is the (algebraic or analytic) affine space
of some dimension. In particular, such a space has regular local rings,
so Lemma~\ref{L:nondegenerate}(a) applies to yield the conclusion.
\end{proof}
\begin{cor} \label{C:Stein}
Let $X$ be a smooth complex analytic space which is Stein. Let $K$ be a compact
subset of $X$. Then the localization of $\Gamma(X, \calO_X)$ at $K$
(i.e., the localization by the multiplicative set of functions which do not
vanish on $K$) is excellent, as is any completion thereof.
\end{cor}
\begin{proof}
The localization is noetherian by \cite[Theorem~1.1]{hironaka-flattening};
the claim then follows from Lemma~\ref{L:nondegenerate}.
\end{proof}

\begin{remark}
The terminology \emph{nondegenerate} for differential rings arises from the fact we had
originally intended condition (c) of Definition~\ref{D:nondegenerate} to state that
the pairing \eqref{eq:pairing} must be nondegenerate on the right. However, it is unclear whether this suffices to ensure the existence
of derivations of rational type with respect to a regular sequence of parameters.
Without such derivations, it is more difficult to work in local coordinates.
To handle this situation,
one would need to rework significant sections of both 
\cite{kedlaya-goodformal1} and \cite{kedlaya-xiao}; we opted against this approach 
because it is not necessary in the
applications of greatest interest, to algebraic and analytic varieties.
\end{remark}

\subsection{$\nabla$-modules}

We next consider differential modules.

\begin{defn}
A \emph{$\nabla$-module} over a differential scheme $X$ is a 
coherent $\calO_X$-module $\calF$ equipped with an action of
$\calD_X$ compatible with the action on $\calO_X$.
Over an affine differential scheme with noetherian underlying scheme, 
this is the same as a finite
differential module over the coordinate ring.
\end{defn}

\begin{defn}
For $(X, \calD_X)$ a differential scheme and $\phi \in \Gamma(X, \calO_X)$,
let $E(\phi)$ be the $\nabla$-module free on one generator $\bv$
satisfying $\del(\bv) = \del(\phi)\bv$ for any open subscheme $U$
of $X$ and any $\del \in \Gamma(U, \calD_X)$.
\end{defn}
\begin{lemma}
Let $X$ be a nondegenerate differential scheme. Then 
every $\nabla$-module over $X$ is locally free over $\calO_X$
(and hence projective).
\end{lemma}
\begin{proof}
This follows from Lemma~\ref{L:nondegenerate}(b) via 
\cite[Proposition~1.2.6]{kedlaya-goodformal1}.
\end{proof}

\begin{remark} \label{R:locally free}
Let $R$ be a nondegenerate differential domain,
and let $M$ be a finite differential module over $R$.
By Lemma~\ref{L:nondegenerate},
$M$ is projective over $R$, and so is a direct summand of a free $R$-module.
Hence 
for any (not necessarily differential)
domains $S,T,U$ with $R \subseteq S,T \subseteq U$,
within $M \otimes_R U$ we have
\[
(M \otimes_{R} S) \cap
(M \otimes_{R} T) =
M \otimes_{R} (S \cap T).
\]
\end{remark}

\subsection{Admissible and good decompositions}
\label{subsec:good decomp}

We next reintroduce the notions of good decompositions
and good formal structures
from \cite{kedlaya-goodformal1}, in the language of nondegenerate
differential rings. Remember that these notions do not quite match the
ones used by Mochizuki; see 
\cite[Remark~4.3.3, Remark~6.4.3]{kedlaya-goodformal1}.

\begin{hypothesis} \label{H:local nondegenerate}
Throughout \S~\ref{subsec:good decomp}--\ref{subsec:criterion},
let $R$ be a nondegenerate differential local ring.
Let $x_1,\dots,x_n$ be a regular sequence of parameters for $R$,
and put $S = R[x_1^{-1},\dots,x_m^{-1}]$ for some
$m \in \{0,\dots,n\}$.
Let $\widehat{R}$ be the completion of $R$ with respect to its maximal
ideal, and put $\widehat{S} = \widehat{R}[x_1^{-1},\dots,x_m^{-1}]$.
Let $M$ be a finite differential module over $S$.
\end{hypothesis}

\begin{defn}
Let $\Delta_R^{\log}$ be the subset of $\Delta_R$ consisting of derivations
under which the ideals $(x_1),\dots,(x_m)$ are stable.
We say $M$ is \emph{regular} if there exists a free $R$-lattice
$M_0$ in $M$ stable under the action of $\Delta_R^{\log}$.
We say $M$ is \emph{twist-regular} if $\End(M) = M^\dual \otimes_R M$
is regular.
\end{defn}

\begin{example}
For any $\phi \in R$, $E(\phi)$
is regular. For any $\phi \in S$, $E(\phi)$ is twist-regular.
\end{example}

\begin{remark} \label{R:reconcile regular}
In case $R = \widehat{R}$,
when checking regularity of $M$,
we may choose an isomorphism $R \cong k \llbracket x_1,\dots,x_n \rrbracket$
as in Corollary~\ref{C:complete nondegenerate} and check stability of the
lattice just under $x_1\del_1,\dots,x_m \del_m$, $\del_{m+1},\dots,\del_n$,
as then \cite[Proposition~2.2.8]{kedlaya-goodformal1} implies stability
under all of $\Delta_R^{\log}$. This means that in case $R = \widehat{R}$,
our definitions of regularity
and twist-regularity match the definitions from \cite{kedlaya-goodformal1},
so we may invoke results from \cite{kedlaya-goodformal1} referring to these
notions without having to worry about our extra level of generality.
\end{remark}

\begin{prop} \label{P:twist-regular}
Suppose $R = \widehat{R}$.
Then the differential module $M$ is twist-regular if and only if $M = E(\phi) \otimes_S N$
for some $\phi \in S$ and some regular differential module $N$ over $S$.
\end{prop}
\begin{proof}
This is \cite[Theorem~4.2.3]{kedlaya-goodformal1}
(plus Remark~\ref{R:reconcile regular}).
\end{proof}

\begin{defn} \label{D:local model1}
An \emph{admissible decomposition} of $M$ is an isomorphism
\begin{equation} \label{eq:local model1}
M \cong \bigoplus_{\alpha \in I} E(\phi_\alpha) \otimes_{S} 
\calR_\alpha
\end{equation}
for some $\phi_\alpha \in S$
(indexed by an arbitrary set $I$)
and some regular differential modules $\calR_\alpha$ over $S$.
An admissible decomposition is \emph{good} if it satisfies the following
two additional conditions.
\begin{enumerate}
\item[(a)]
For $\alpha \in I$, if $\phi_\alpha \notin R$, then
$\phi_\alpha$ has the form $u x_1^{-i_1} \cdots x_m^{-i_m}$
for some unit $u$ in $R$ 
and some nonnegative integers $i_1,\dots, i_m$.
\item[(b)]
For $\alpha, \beta \in I$, 
if $\phi_\alpha - \phi_\beta \notin R$,
then $\phi_\alpha - \phi_\beta$ has the form $u x_1^{-i_1} \cdots x_m^{-i_m}$
for some unit $u$ in $R$ and some nonnegative integers $i_1,\dots, i_m$.
\end{enumerate}
A \emph{ramified good decomposition} of $M$ is a good decomposition
of $M \otimes_R R'$ for some connected finite 
integral extension $R'$ of $R$
such that $R' \otimes_R S$ is \'etale over $S$. By Abhyankar's lemma
\cite[Expos\'e XIII, Proposition~5.2]{sga1},
any such extension is contained in $R''[x_1^{1/h},\dots,x_m^{1/h}]$
for some connected finite \'etale extension $R''$ of $R$ and 
some positive integer $h$.
A \emph{good formal structure} of $M$ is a ramified good decomposition
of $M \otimes_R R'$ for $R'$ the completion of $R$ with respect
to $(x_1,\dots,x_m)$. 
This is not the same as a ramified good decomposition
of $M \otimes_R \widehat{R}$ (since $\widehat{R}$ is the completion
with respect to the larger ideal $\gothm$), but any such decomposition
does in fact induce a good formal structure
(see Proposition~\ref{P:good formal}).
\end{defn}

\begin{remark} \label{R:minimal}
In Definition~\ref{D:local model1}, an admissible decomposition need
not be unique if it exists. However, there is a unique \emph{minimal}
admissible decomposition, obtained by combining the terms indexed by
$\alpha$ and $\beta$ whenever $\phi_\alpha - \phi_\beta \in R$.
The resulting minimal admissible decomposition is good if and only if
the original admissible decomposition is good.
\end{remark}

The following limited descent argument will crop up several times.
\begin{prop} \label{P:descend good ring}
Suppose that $R$ is henselian, and that
$M \otimes_S \widehat{S}$ admits a filtration 
$0 = M_0 \subset \cdots \subset M_\ell = M \otimes_S \widehat{S}$
by differential submodules,
in which each successive quotient $M_{j+1}/M_{j}$
admits an admissible decomposition
$\oplus_{\alpha \in I_j} E(\phi_\alpha) \otimes_{\widehat{S}} \calR_\alpha$.
Then the
$\phi_\alpha$ always belong to $S + \widehat{R}$; in particular, they can be
chosen in $S$ if desired.
\end{prop}
\begin{proof}
For $i=0,\dots,m$, put $\widehat{S}_i = \widehat{R}[x_1^{-1},\dots,x_i^{-1}]$.
We show that 
\begin{equation} \label{eq:descend}
\phi_\alpha \in S + \widehat{S}_i \qquad (i=m,\dots,0; \, \alpha \in \bigcup_j I_j)
\end{equation}
by descending induction on $i$, the case $i=m$ being evident
and the case $i=0$ yielding the desired result.
Given \eqref{eq:descend} for some $i>0$, 
we can choose a nonnegative integer $h$
such that
\begin{equation} \label{eq:descend2}
x_i^h \phi_\alpha \in S + \widehat{S}_{i-1} \qquad (\alpha \in \bigcup_j I_j).
\end{equation}
We wish to achieve this for $h=0$, which we accomplish using a second
descending induction on $h$. If $h>0$, write each
$x_i^{h} \phi_\alpha$ as 
$f_\alpha + g_\alpha$ with $f_\alpha \in S$ and $g_\alpha \in 
\widehat{S}_{i-1}$.

Choose $\alpha \in I_j$ for some $j$.
Let $T$ and $U$ denote the $x_i$-adic completions of $\Frac(S)$
and $\Frac(\widehat{S})$, respectively.
Put $N = M \otimes_S E(-x_i^{-h} f_\alpha)$. 
We may then apply \cite[Theorem~2.3.3]{kedlaya-goodformal1}
to obtain a Hukuhara-Levelt-Turrittin decomposition of 
$N \otimes_S T'$ for some finite extension $T'$ of $T$.

We claim that $U' = T' \otimes_T U$ is a field extension of $U$, from which it follows
that $T'$ is the integral closure of $T$ in $U'$. It suffices to check this after adjoining
$x_i^{1/m}$ for some positive integer $m$, so we may assume $T'$ is unramified over $T$.
In that case, we must show that the residue fields of $T'$ and $U$ are linearly disjoint
over the residue field of $T$, i.e., that for any finite extension $\ell$ of
$\Frac(R/x_i R)$, $\ell$ and $\Frac(\widehat{R}/x_i \widehat{R})$ have no common
subfield strictly larger than $\Frac(R/x_iR)$. This holds because by
Remark~\ref{R:regular local}, such a subfield would induce a finite \'etale extension of $R/x_i R$
with the same residue field; however, any such extension must equal $R/x_i R$ because the latter
ring is henselian (because $R$ is). This proves the claim.

We may extend scalars to obtain a Hukuhara-Levelt-Turrittin decomposition of
$N \otimes_S U'$ in which the factors $E(r)$ all have $r \in T'$. 
However, since $\alpha \in I_j$,
$(N \otimes_S U') \otimes_{U'} E(-x_i^{-h} g_\alpha)$ 
has a nonzero regular subquotient. 
This is only possible
if one of the factors $E(r)$ in the decomposition of $N \otimes_S U'$
satisfies $r \equiv x_i^{-h} g_\alpha \pmod{\gotho_{U'}}$.

In particular, if we choose $e_1,\dots,e_{i-1}$ so that
$g'_\alpha = x_1^{e_1}\dots x_{i-1}^{e_{i-1}} g_\alpha$ belongs to
$\widehat{R}$, then the image of $g'_\alpha$ in $\widehat{R}/x_i \widehat{R}$
is algebraic over $\Frac (R/x_i R)$. We again use the henselian property
of $R/x_i R$ to deduce that the image of $g'_\alpha$
in $\widehat{R}/x_i \widehat{R}$ must in fact belong to $R/x_i R$.
 This allows 
us to replace $h$ by $h-1$ in \eqref{eq:descend2}, completing both
inductions.
\end{proof}

\begin{remark} \label{R:descend good ring}
Proposition~\ref{P:descend good ring} implies that if $R$ is henselian and
$M \otimes_S \widehat{S}$ admits a good decomposition, then the terms
$\phi_\alpha$ appearing in \eqref{eq:local model1} can be defined over $S$.
Using Theorem~\ref{T:regular} below, we can also realize the regular modules
$\calR_\alpha$ over $S$ provided that we can identify 
$\widehat{R}$ with $k \llbracket x_1,\dots,x_n \rrbracket$
in such a way that $k$ embeds into $R$.
In such cases, $M \otimes_S \widehat{S}$ admits
a good decomposition
in the sense of Sabbah \cite[I.2.1.5]{sabbah}. This observation generalizes
an argument of Sabbah for surfaces \cite[Proposition~I.2.4.1]{sabbah}
and fulfills a promise made in \cite[Remark~6.2.5]{kedlaya-goodformal1}.

On the other hand, Proposition~\ref{P:descend good ring} \emph{does not} 
imply that any good decomposition of $M \otimes_S \widehat{S}$
descends to a good decomposition of $M$ itself.
This requires two additional steps which cannot always be carried out. One must descend
the projectors cutting out the summands $E(\phi_\alpha) \otimes_{\widehat{S}} \calR_\alpha$
of the minimal good decomposition.
If this can be achieved, then by virtue of Proposition~\ref{P:descend good ring},
each summand can be twisted to give a differential module $N_\alpha$ over $S$ such that
$N_\alpha \otimes_S \widehat{S}$ is regular.
One must then check that each $N_\alpha$ itself is regular.
For a typical situation where both steps can be executed, see 
Theorem~\ref{T:descend}.
\end{remark}

\subsection{Good decompositions over complete rings}
\label{subsec:criterion}

We now examine more closely the case of a nondegenerate
differential \emph{complete} local ring, recalling some of the
key results from \cite{kedlaya-goodformal1}.
Throughout \S~\ref{subsec:criterion}, continue to retain
Hypothesis~\ref{H:local nondegenerate}.

\begin{defn}
Use Corollary~\ref{C:complete nondegenerate}
to identify $\widehat{R}$ with $k \llbracket x_1,\dots,x_n \rrbracket$
for $k$ the residue field of $R$, using some derivations
$\del_1,\dots,\del_n$ of rational type with respect to $x_1,\dots,x_n$.
For $r = (r_1,\dots,r_n) \in [0, +\infty)^n$, let $|\cdot|_r$ be the 
$(e^{-r_1},\dots,e^{-r_n})$-Gauss norm on $\widehat{R}$; note that this does 
not depend on the choice of the isomorphism 
$\widehat{R} \cong k \llbracket x_1,\dots,x_n \rrbracket$.
Let $F_r$ be the completion of $\Frac(\widehat{R})$ with respect to
$|\cdot|_r$. Let $F(M,r)$ be the irregularity of $M \otimes_S F_r$,
as defined in \cite[Definition~1.4.8]{kedlaya-goodformal1}.
We say $M$ is \emph{numerical} if $F(M,r)$ is a linear function of $r$.
\end{defn}

The following is a consequence of \cite[Theorem~3.2.2]{kedlaya-goodformal1}.
\begin{theorem} \label{T:convex}
The function $F(M,r)$ is
continuous, convex, and piecewise linear.
Moreover, for $j \in \{m+1,\dots,n\}$, if we fix $r_i$ for $i \neq j$,
then $F(M,r)$ is nonincreasing as a function of $r_j$ alone.
\end{theorem}

We have the following numerical criterion for regularity
by \cite[Theorem~4.1.4]{kedlaya-goodformal1} (plus Remark~\ref{R:reconcile
regular}).
\begin{theorem} \label{T:regular}
Assume that $R = \widehat{R}$. 
Then the following conditions are equivalent.
\begin{enumerate}
\item[(a)]
$M$ is regular.
\item[(b)]
There exists a basis of $M$ on which $x_1 \del_1,\dots,x_m\del_m$
act via commuting matrices over $k$ with prepared eigenvalues
(i.e., no eigenvalue or difference between two eigenvalues
equals a nonzero integer),
and $\del_{m+1},\dots,\del_n$ act via the zero matrix.
\item[(c)]
We have $F(M,r) = 0$ for all $r$.
\end{enumerate}
\end{theorem}

We also have the following numerical criterion for existence of a 
ramified good decomposition in the complete case.
\begin{theorem} \label{T:criterion}
Assume that $R = \widehat{R}$. The following conditions are equivalent.
\begin{enumerate}
\item[(a)]
The module $M$ admits a ramified good decomposition.
\item[(b)]
Both $M$ and $\End(M)$ are numerical.
\end{enumerate}
\end{theorem}
\begin{proof}
This holds by \cite[Theorem~4.4.2]{kedlaya-goodformal1} modulo
one minor point: 
since \cite[Theorem~4.4.2]{kedlaya-goodformal1} does not allow for
derivations on the residue field of $R$,
it only gives a good decomposition with respect to the action of $\del_1,
\dots,\del_n$. However, if we form the minimal good decomposition
as in Remark~\ref{R:minimal}, this decomposition must be 
preserved by the actions of the other derivations.
\end{proof}

We will find useful 
the following consequence of the existence of a good decomposition.
\begin{prop} \label{P:filtration}
Assume that $R = \widehat{R}$ and that $M$ admits a good decomposition.
Then for some finite \'etale extension $R'$ of $R$,
$M \otimes_R R'$ admits a filtration 
$0 = M_0 \subset \cdots \subset M_d = M \otimes_R R'$ by differential
submodules, with the following properties.
\begin{enumerate}
\item[(a)]
We have $\rank(M_i) = i$ for $i=0,\dots,d$.
\item[(b)]
For $i=1,\dots,d-1$, there exists an endomorphism of 
$\wedge^i M$ as a differential module with image $\wedge^i M_i$.
\end{enumerate}
\end{prop}
\begin{proof}
We reduce first to the case where $M$ is twist-regular, then to the case
where $M$ is regular. 
By Theorem~\ref{T:regular},
there exists a basis of $M$ on which $x_1 \del_1,\dots,x_m\del_m$
acts via commuting matrices over $k$ with prepared eigenvalues,
and $\del_{m+1},\dots,\del_n$ act via the zero matrix.
Let $V$ be the $k$-span of this basis. Choose a finite extension $k'$
of $k$ containing all of the eigenvalues of these matrices, and put
$R' = k' \llbracket x_1,\dots,x_n \rrbracket$.

We can now split $V \otimes_k k'$ as a direct sum such that
on each summand, each $x_i \del_i$ acts with a single eigenvalue;
this splitting induces a splitting of $M$ itself.
After replacing $k$ with $k'$,
we may now reduce to the case where each $x_i \del_i$ acts on $V$
with a single eigenvalue. By twisting, we can force that eigenvalue to be zero.

By Engel's theorem 
\cite[Theorem~9.9]{fulton-harris},
$x_1 \del_1,\dots,x_m \del_m$ act on some
complete flag $0 = V_0 \subset \cdots \subset V_d = V$ in $V$.
The corresponding submodules $0 = M_0 \subset \cdots \subset M_d = M$
of $M$ have the desired property: for instance, $M_1$ occurs as the image
of the composition of the projection $M \to M/M_{d-1}$,
an isomorphism $M/M_{d-1} \to M_1$, and the inclusion $M_1 \to M$.
\end{proof}

\section{Descent arguments}
\label{sec:descent}

In this section, we make some crucial descent arguments for differential
modules over a localized power series ring with coefficients in
a base ring.

\subsection{Hensel's lemma in noncommutative rings}

We need to recall a technical tool in the study of differential modules
over nonarchimedean rings, a form of Hensel's lemma for noncommutative
rings introduced by Robba.
Our presentation follows Christol \cite{christol}.
\begin{theorem}[Robba, Christol] \label{T:master factor}
Let $R$ be a nonarchimedean, not necessarily commutative ring.
Suppose the nonzero elements $a,b,c \in R$ and the additive subgroups $U,V,W 
\subseteq R$ satisfy the following conditions.
\begin{enumerate}
\item[(a)]
The spaces $U,V$ are complete under the norm, and $UV \subseteq W$.
\item[(b)]
The map $f(u,v) = av + ub$ is a surjection of $U \times V$ onto $W$.
\item[(c)]
There exists $\lambda > 0$ such that 
\[
|f(u,v)| \geq \lambda \max\{|a||v|, |b||u|\} \qquad (u \in U, v \in V).
\]
(Note that this forces $\lambda \leq 1$.)
\item[(d)]
We have $ab - c \in W$ and
\[
|ab-c| < \lambda^2 |c|.
\]
\end{enumerate}
Then there exists a unique pair $(x,y) \in U \times V$ such that
\[
c = (a+x)(b+y), \qquad |x| < \lambda |a|, \qquad |y| < \lambda |b|.
\]
For this $x,y$, we also have
\[
|x| \leq \lambda^{-1} |ab-c| |b|^{-1}, \qquad
|y| \leq \lambda^{-1} |ab-c| |a|^{-1}.
\]
\end{theorem}
\begin{proof}
See \cite[Proposition~1.5.1]{christol} or \cite[Theorem~2.2.2]{kedlaya-course}.
\end{proof}

\subsection{Descent for iterated power series rings}
\label{subsec:iterated}

Using Christol's factorization theorem, we make a decompletion
argument analogous to \cite[\S 2.6]{kedlaya-goodformal1}.
We use the language of Newton polygons and slopes for twisted polynomials,
as presented in \cite[Definition~1.6.1]{kedlaya-goodformal1}.

\begin{hypothesis} \label{H:iterated}
Throughout \S~\ref{subsec:iterated},
let $h$ be a positive integer.
Let $A$ be a differential domain of characteristic $0$,
such that the module of derivations on $K = \Frac(A)$
is finite-dimensional over $K$, and the constant subring $k$ of $A$
is also the constant subring of $K$. (In particular, $k$ must 
be a field.)
Define the ring $R_h(K)$ as $K((x_1))\cdots((x_h))$.
Define the ring $R_h^\dagger(A)$ as
the union of $A((x_1/f))\cdots((x_h/f))[f^{-1}]$
over all nonzero $f \in A$.
Equip $R_h(K)$ and $R_h^\dagger(A)$ 
with the componentwise derivations coming from
$A$, plus the derivations
$\del_1,\dots,\del_h = \frac{\del}{\del x_1},
\dots, \frac{\del}{\del x_h}$.
\end{hypothesis}

\begin{lemma} \label{L:unit}
The rings $R_h(K)$ and $R_h^\dagger(A)$ are fields.
\end{lemma}
\begin{proof}
This is clear for $R_h(K)$, so we concentrate on
$R_h^\dagger(A)$.
We proceed by induction on $h$, the case $h=0$ being clear because
$R_0^\dagger(A) = K$.

Given $r \in R^\dagger_h(A)$, write
$r = \sum_i r_i x_h^i$ with $r_i \in R^\dagger_{h-1}(A)$.
Let $m$ be the smallest index such that $r_m \neq 0$.
By the induction hypothesis, $r_m$ is a unit in $R^\dagger_{h-1}(A)$,
so we may reduce to the case where $m=0$ and $r_m = 1$.
In this case, for some nonzero $f \in A$ we have
\[
1-r \in (x_h/f) A((x_1/f))\cdots((x_{h-1}/f))\llbracket x_h/f \rrbracket
[f^{-1}];
\]
by replacing $f$ by a large power, we may force
\[
1-r \in (x_h/f) A((x_1/f))\cdots((x_{h-1}/f))\llbracket x_h/f \rrbracket.
\]
In that case, the formula $r^{-1} = \sum_{j=0}^\infty (1-r)^j$
shows that $r^{-1} \in R^\dagger_h(A)$.
\end{proof}

Using Christol's factorization theorem, we obtain the following.
\begin{prop} \label{P:dagger factor}
Equip $R_h^\dagger(A)$ with the $x_h$-adic norm 
(of arbitrary normalization) and the derivation
$x_h \del_h$. Then any twisted polynomial
$P \in R_h^\dagger(A)\{T\}$ factors uniquely as a product
$Q_1 \cdots Q_m$ with each $Q_i$ having only one slope $r_i$ in its
Newton polygon, and $r_1 < \cdots < r_m$.
\end{prop}
\begin{proof}
Over $R_h(K)$, such a factorization exists and is unique 
by \cite[Lemma~1.6.2]{kedlaya-part1}. It thus suffices to check
existence over $R_h^\dagger(A)$. 
For this, we induct on the number of slopes in the Newton
polygon of $P(T) = \sum_i P_i T^i$.
Suppose there is more than one slope; we can then
choose an index $i$ corresponding to
an internal vertex of the Newton polygon. By Lemma~\ref{L:unit},
$P_i$ is a unit in $R_h^\dagger(A)$, so we may reduce to the case $P_i = 1$.

Since $i$ corresponds to an internal vertex of the Newton polygon,
there exists a positive rational number $r/s$ such that
$-\log |P_{i-j}| - j(r/s) \log |x_h| > 0$ for $j \neq 0$.
That is, the norm of $x_h^{jr/s} P_{i-j}$ is less than 1
for each $j \neq 0$.
Since $P_{j} \in R_h^\dagger(A)$ for all $j \neq i$,
we can choose $f \in A$ nonzero so that
\[
P_{j} \in A((x_1/f))\cdots((x_h/f))[f^{-1}] \qquad (j \neq i).
\]
We then have
\[
x_h^{(i-j)r/s} P_{j} \in
(x_h/f)^{1/s} A((x_1/f))\cdots((x_{h-1}/f)) \llbracket (x_h/f)^{1/s}
\rrbracket [f^{-1}] \qquad (j \neq i).
\]
As in the proof of Lemma~\ref{L:unit},
by replacing $f$ by a large power, we may force
\begin{equation} \label{eq:factor}
x_h^{(i-j)r/s} P_{j} \in
(x_h/f)^{1/s} A((x_1/f))\cdots((x_{h-1}/f)) \llbracket (x_h/f)^{1/s}
\rrbracket \qquad (j \neq i).
\end{equation}
Let $U$ be the set of twisted polynomials $Q(T) = \sum_j Q_j T^j \in R_h^\dagger(A)\{T\}$ 
of degree at most $\deg(P) - i$ such that
\[
x_h^{-jr/s} Q_{j} \in
A((x_1/f))\cdots((x_{h-1}/f)) \llbracket (x_h/f)^{1/s}
\rrbracket
\qquad (j=0,\dots,\deg(P)-i).
\]
Let $V$ be the set of twisted polynomials $Q(T) = \sum_j Q_j T^j \in R_h^\dagger(A)\{T\}$ 
of degree at most $i-1$ such that
\[
x_h^{(i-j)r/s} Q_{j} \in
A((x_1/f))\cdots((x_{h-1}/f)) \llbracket (x_h/f)^{1/s}
\rrbracket
\qquad (j=0,\dots,i-1).
\]
Let $W$ be the set of twisted polynomials $Q(T) = \sum_j Q_j T^j \in R_h^\dagger(A)\{T\}$ 
of degree at most $\deg(P)$ such that
\[
x_h^{(i-j)r/s} Q_{j} \in
A((x_1/f))\cdots((x_{h-1}/f)) \llbracket (x_h/f)^{1/s}
\rrbracket
\qquad (j=0,\dots,\deg(P)).
\]
Then $U,V,W$ are complete for the $|x_h|^{-r/s}$-Gauss norm and $UV \subseteq W$.
Put $a = 1, b = T^i, c = P$, so that $(u,v) \mapsto av + bu$ is a surjection
of $U \times V$ onto $W$, and $|ab-c| < |c|$.

We now invoke Theorem~\ref{T:master factor}
to obtain a nontrivial factorization $Q_1 Q_2$ of $P$ in which
all slopes of $Q_1$ are less than $-r/s$ while all slopes
of $Q_2$ are greater than $-r/s$.
(Condition (c) of Theorem~\ref{T:master factor}
may be verified exactly as in \cite[Theorem~2.2.1]{kedlaya-course}.)
We may then invoke the induction hypothesis to conclude.
\end{proof}

\begin{prop} \label{P:dagger decomp}
Equip $R_h^\dagger(A)$ with the $x_h$-adic norm 
(of arbitrary normalization).
Let $M$ be a finite differential module over $R_h^\dagger(A)$. Then
$M$ admits a unique decomposition $M = \oplus_{s \geq 1} M_s$
as a direct sum of differential submodules, such that for each $s \geq 1$,
the scale multiset of $\del_h$ on 
$M_s \otimes_{R_h^\dagger(A)} R_h(K)$ consists entirely of $s$.
\end{prop}
\begin{proof}
As in \cite[Proposition~1.6.3]{kedlaya-goodformal1}, except using
Proposition~\ref{P:dagger factor} in place of \cite[Lemma~1.6.2]{kedlaya-goodformal1}.
\end{proof}

The following argument is reminiscent of 
\cite[Lemma~2.6.3]{kedlaya-goodformal1}.
\begin{lemma} \label{L:dagger descend}
Equip $R_h^\dagger(A)$ with the $x_h$-adic norm 
(of arbitrary normalization).
Let $M$ be a finite differential module over $R_h^\dagger(A)$ 
such that the scale of
$\del_h$ on $M \otimes_{R_h^\dagger(A)} R_h(K)$ is equal to $1$.
Then $M$ admits an $R_{h-1}^\dagger(A)$-lattice stable under 
all of the given derivations on $R_h^\dagger(A)$.
\end{lemma}
\begin{proof}
Put $R = R_h^\dagger(A)$ and $R' = R_{h-1}^\dagger(A)((x_h))$.
By \cite[Proposition~2.2.10]{kedlaya-goodformal1}, we can find 
a regulating lattice $W$ in $M \otimes_R R'$.
By \cite[Proposition~2.2.11]{kedlaya-goodformal1}, the characteristic polynomial
of $x_h \del_h$ on $W/x_h W$ has coefficients in the constant subring of 
$R$, which is $k$.
(Note that in order to satisfy the running 
hypothesis \cite[Hypothesis~2.1.1]{kedlaya-goodformal1}, we need
that the module of derivations on $K$ is finite-dimensional over 
$K$.)

Choose a basis of $W/x_h W$ on which $x_h \del_h$ acts via a matrix over $k$.
Since $R$ is a dense subfield of $R'$, we can lift this basis to 
a basis $\be_1,\dots,\be_d$ of $W$ consisting of elements of $M$.
Let $N$ be the matrix of action of $x_h \del_h$ on this basis.
Then $N$ has entries in 
\[
R_h^\dagger(A) \cap (k + x_h R_{h-1}^\dagger(A)\llbracket x_h \rrbracket).
\]
We can thus choose $f \in A$ nonzero
so that $N$ has entries in 
\[
k + (x_h/f) A((x_1/f))\cdots((x_{h-1}/f))\llbracket x_h/f \rrbracket.
\]
Write $N = \sum_{i=0}^\infty N_i (x_h/f)^i$ with $N_i$ having
entries in $R_{h-1}^\dagger(A)$.
As in \cite[Lemma~2.2.12]{kedlaya-goodformal1},
there exists a unique matrix $U = \sum_{i=0}^\infty U_i (x_h/f)^i$
over $R_{h-1}^\dagger(A) \llbracket x_h/f \rrbracket$
with $U_0$ equal to the identity matrix, such that 
\[
NU + x_h \del_h(U) = U N_0.
\]
Namely, given $U_0,\dots,U_{i-1}$, there is a unique choice of $U_i$
satisfying
\[
i U_i = U_i N_0 - N_0 U_i - \sum_{j=1}^{i} N_j U_{i-j}
\]
because $N_0$ has prepared eigenvalues. More explicitly,
each entry of $U_i$ is a certain $k$-linear combination of entries of 
$\sum_{j=1}^{i} N_j U_{i-j}$.
By induction on $i$, it follows that
\[
U_i \in A((x_1/f))\cdots((x_{h-1}/f)) \qquad (i=1,2,\dots).
\]
That is, $U$ has entries in $A((x_1/f)) \cdots ((x_h/f))$
and thus in $R_h^\dagger(A)$.
The desired result follows.
\end{proof}

\begin{prop} \label{P:descend power}
For any finite differential module $M$ over $R_h^\dagger(A)$, we have
\[
H^0(M) = H^0(M \otimes_{R_h^\dagger(A)} R_h(K)).
\]
\end{prop}
\begin{proof}
We induct on $h$ with trivial base case $h=0$.
Suppose that $h > 0$. 
Put $R = R_h^\dagger(A)$.
Pick any $\bv \in M \otimes_{R} R_h(K)$.
Equip $R$ with the $x_h$-adic norm (of arbitrary normalization).
By Proposition~\ref{P:dagger decomp},
we can split $M$ as a direct sum $M_1 \oplus M_2$, in which 
the scale multiset of $\del_h$ on $M_1 \otimes_{R} R_h(K)$
has all elements equal to 1, while the scale multiset of
$\del_h$ on $M_2 \otimes_R R_h(K)$ has all elements greater than 1.
We must then have $\bv \in M_1 \otimes_R R_h(K)$.

By Lemma~\ref{L:dagger descend}, $M_1$ admits an $R_{h-1}^\dagger(A)$-lattice
$N$ stable under all of the given derivations on $R$.
Write $\bv$ formally as $\sum_{i \in \ZZ} \bv_i x_h^i$ with
$\bv_i \in N \otimes_{R_{h-1}^\dagger(A)} R_{h-1}(K)$. 
Since the action of $x_h \del_h$ on $N$ has prepared
eigenvalues, the equality $x_h \del_h(\bv) = 0$ implies that
$\bv_i = 0$ for $i \neq 0$.
Hence $\bv \in N \otimes_{R_{h-1}^\dagger(A)} R_{h-1}(K)$,
so the induction hypothesis implies $\bv \in N$, and the desired
result follows.
\end{proof}

\subsection{Descent for localized power series rings}
\label{subsec:diff mod relative}

Using the iterated power series rings we have just considered, 
we obtain some descent
results for localized power series rings.

\setcounter{theorem}{0}
\begin{hypothesis}
Throughout \S~\ref{subsec:diff mod relative}, let 
$A$ be a differential domain of characteristic $0$,
such that the module of derivations on $K = \Frac(A)$
is finite-dimensional over $K$, and the constant subring $k$ of $A$
is also the constant subring of $K$. 
For integers $n \geq m \geq 0$,
let $R_{n,m}^\dagger(A)$ be the union
of
\[
A \llbracket x_1/f, \dots, x_n/f \rrbracket
[(x_1/f)^{-1},\dots,(x_m/f)^{-1}][f^{-1}]
\]
over all nonzero $f \in A$.
Put $R_{n,m}(K) = K \llbracket x_1,\dots,x_n \rrbracket[x_1^{-1},\dots,x_m^{-1}]$.
Equip $R_{n,m}^\dagger(A)$ and $R_{n,m}(K)$
with the componentwise derivations on $A$
plus the derivations $\del_1,\dots,\del_n = \frac{\del}{\del x_1},
\dots, \frac{\del}{\del x_n}$.
Note that $R_{n,0}^\dagger(A)$ is a henselian local ring
which is nondegenerate as a differential ring.
\end{hypothesis}

\begin{prop} \label{P:descend1}
For any finite differential module $M$ over $R^\dagger_{n,m}(A)$,
\[
H^0(M) = H^0(M \otimes_{R_{n,m}^\dagger(A)} R_{n,m}(K)).
\]
\end{prop}
\begin{proof}
Embed $R^\dagger_{n,m}(A)$ into the ring $R^\dagger_n(A)$ defined in
Hypothesis~\ref{H:iterated}. Then within $R_n(K)$,
$R^\dagger_{n,m}(A)$ is the intersection of $R_{n,m}(K)$
and $R_n^\dagger(A)$.
By Remark~\ref{R:locally free},
within $M \otimes_{R_{n,m}^\dagger(A)} R_n(K)$ we have
\begin{equation} \label{eq:intersect}
(M \otimes_{R^\dagger_{n,m}(A)} R_{n,m}(K))
\cap
(M \otimes_{R^\dagger_{n,m}(A)} R_n^\dagger(A))
=
M.
\end{equation}
Given any $\bv \in H^0(M \otimes_{R^\dagger_{n,m}(A)} R_{n,m}(K))$,
we have $\bv \in H^0(M \otimes_{R^\dagger_{n,m}(A)} R_n^\dagger(A))$
by Lemma~\ref{L:dagger descend}. By \eqref{eq:intersect},
this implies $\bv \in H^0(M)$, proving the claim.
\end{proof}

We also need the following related argument in the regular case.
\begin{prop} \label{P:descend2}
Let $M$ be a finite differential module over $R^\dagger_{n,m}(A)$,
such that $M \otimes_{R^\dagger_{n,m}(A)} R_{n,m}(K)$ is regular.
Then $M$ is regular, in the sense that there exists a basis of
$M$ on which $x_1 \del_1,\dots,x_n \del_n$ act via matrices
over $K$.
\end{prop}
\begin{proof}
Again, embed $R^\dagger_{n,m}(A)$ into $R^\dagger_n(A)$.
We can construct bases of $M \otimes_{R^\dagger_{n,m}(A)} R_n(K)$ 
having the desired
property in two different fashions. One is to apply
\cite[Theorem~4.1.4]{kedlaya-goodformal1} to construct a suitable basis
of $M \otimes_{R^\dagger_{n,m}(A)} R_{n,m}(K)$, 
and then extend scalars to $R_n(K)$. The other is to
construct a suitable basis of $M \otimes_{R^\dagger_{n,m}(A)} R_n^\dagger(A)$
by repeated application of Lemma~\ref{L:dagger descend}, and then extend
scalars to $R_n(K)$.

The resulting bases must have the same 
$K[x_1,\dots,x_n][x_1^{-1},\dots,x_m^{-1}]$-span (as in the proof
of \cite[Proposition~2.2.13]{kedlaya-goodformal1}). 
They thus both consist of elements of the intersection
of $M \otimes_{R^\dagger_{n,m}(A)} R_{n,m}(K)$ with 
$M \otimes_{R^\dagger_{n,m}(A)} R_n^\dagger(A)$. By 
\eqref{eq:intersect},
this intersection equals $M$;
this yields the desired result.
\end{proof}

Putting these arguments together yields the following.
\begin{theorem} \label{T:descend}
Let $M$ be a finite differential module over $R^\dagger_{n,m}(A)$.
Then any admissible (resp.\ good) decomposition of 
$M \otimes_{R^\dagger_{n,m}(A)} R_{n,m}(K)$
descends to an admissible (resp.\ good) decomposition of $M$.
\end{theorem}
\begin{proof}
Given an admissible decomposition of 
$M \otimes_{R^\dagger_{n,m}(A)} R_{n,m}(K)$,
the projectors onto the summands are horizontal sections
of $\End(M) \otimes_{R^\dagger_{n,m}(A)} R_{n,m}(K)$.
These descend to $\End(M)$ by Proposition~\ref{P:descend1}.
With notation as in \eqref{eq:local model1}, 
the $\phi_\alpha$ can be chosen in $R^\dagger_{n,m}(A)$
by Proposition~\ref{P:descend good ring}.
Hence the $\calR_\alpha$ can be defined over $R^\dagger_{n,m}(A)$;
they are regular by Proposition~\ref{P:descend2}.
\end{proof}
\begin{cor} \label{C:descend}
Let $M$ be a finite differential module over 
$A\llbracket x_1,\dots,x_n \rrbracket 
[x_1^{-1},\dots,x_m^{-1}]$.
Then any admissible (resp.\ good) decomposition of 
$K\llbracket x_1,\dots,x_n \rrbracket 
[x_1^{-1},\dots,x_m^{-1}]$
descends to an admissible (resp.\ good) decomposition of 
$A_f\llbracket x_1,\dots,x_n \rrbracket 
[x_1^{-1},\dots,x_m^{-1}]$
for some nonzero $f \in A$.
\end{cor}

\subsection{Good formal structures}

One application of Theorem~\ref{T:descend} is to relate 
ramified good decompositions over complete rings to good formal structures
over noncomplete rings.

\begin{prop} \label{P:good formal}
Let $R$ be a nondegenerate differential local ring
with completion $\widehat{R}$.
Let $x_1,\dots,x_n$ be a regular sequence of parameters for $R$,
and put $S = R[x_1^{-1},\dots,x_m^{-1}]$ for some $m$.
Let $M$ be a finite differential module over $S$. Then any
ramified good decomposition of $M \otimes_R \widehat{R}$ 
induces a good formal structure of $M$.
\end{prop}
\begin{proof}
We may assume from the outset that $R$ is complete with respect
to the ideal $(x_1,\dots,x_m)$.
In addition,
by replacing $R$ with a finite integral extension $R'$ such 
that $R' \otimes_R S$
is \'etale over $S$, we may reduce to the case where $M \otimes_R
\widehat{R}$ admits
a good decomposition.

Choose derivations
$\del_1,\dots,\del_n \in \Delta_R$ of rational type with respect to $x_1,\dots,x_n$,
then identify $\widehat{R}$ with
$k \llbracket x_1,\dots,x_n \rrbracket$ as in Corollary~\ref{C:complete
nondegenerate}. Let $R_m$ be the joint kernel of $\del_1,\dots,\del_m$
on $R$; then by Lemma~\ref{L:complete nondegenerate}, 
we have an isomorphism $R \cong R_m \llbracket x_1,\dots,x_m \rrbracket$.
Put $K =\Frac(R_m)$.
By Theorem~\ref{T:criterion}, for some finite extension $K'$ of $K$
and some positive integer $h$,
\[
M_{K'} = M \otimes_S K' \llbracket x_1^{1/h},\dots,x_m^{1/h} \rrbracket [x_1^{-1/h},\dots,x_m^{-1/h}]
\]
admits a ramified good decomposition;
by enlarging $R$, we may reduce to the case $h=1$.

Put $L = \Frac(k\llbracket x_{m+1},\dots,x_n \rrbracket)$,
and let $L'$ be a component of $L \otimes_K K'$.
By Remark~\ref{R:locally free}, combining the minimal good decompositions
of $M \otimes_R \widehat{R}$ and $M_{K'}$ yields a good decomposition of 
\[
M \otimes_S T\llbracket x_1,\dots,x_m \rrbracket [x_1^{-1},\dots,x_m^{-1}]
\]
for $T$ equal to the intersection 
$k \llbracket x_{m+1},\dots,x_n \rrbracket \cap K'$ within $L'$.
By Remark~\ref{R:regular local}, $T$ is finite \'etale over $R$, yielding the desired result.
\end{proof}

\section{Good formal structures}
\label{sec:good formal define}

We collect some basic facts about good formal structures
on nondegenerate differential schemes, complex
analytic varieties,
and formal completions thereof.

\setcounter{theorem}{0}
\begin{hypothesis} \label{H:good formal define}
Throughout \S~\ref{sec:good formal define},
let $X$ be \emph{either} a nondegenerate differential scheme,
or a smooth (separated) complex analytic space (see \S~\ref{subsec:complex}).
For short, we distinguish these two options as the \emph{algebraic case} and 
the \emph{analytic case}.
Let $Z$ be a closed subspace of $X$ containing no irreducible component of
$X$.
Let $\widehat{X|Z}$ be the formal completion of $X$ along $Z$
(in the category of locally ringed spaces).
Let $\calE$ be a $\nabla$-module over $\calO_{\widehat{X|Z}}(*Z)$.
\end{hypothesis}

\subsection{Good formal structures}

\begin{defn}
Let $x \in Z$ be a point in a neighborhood of which 
$(X,Z)$ is a regular pair.
We say that $\calE$ admits an \emph{admissible decomposition}
(resp.\ a \emph{good decomposition}, a \emph{ramified good decomposition})
at $x$ if 
the restriction of $\calE$ to $\widehat{\calO_{X,x}}(*Z)$
admits an admissible decomposition (resp.\ a good decomposition,
a ramified good decomposition).
Let $Y$ be the intersection of the components of $Z$ passing through $x$;
by Proposition~\ref{P:good formal},
the restriction of $\calE$ to $\widehat{\calO_{X,x}}(*Z)$
admits a ramified good decomposition if and only if the restriction of
$\calE$ to $\calO_{\widehat{X|Y},x}(*Z)$ does so. We describe this
condition
by saying that $\calE$ admits a \emph{good formal structure} at $x$.

Suppose $(X,Z)$ is a regular pair.
We define the \emph{turning locus} of $\calE$ to be
the set of $x \in Z$ at which
$\calE$ fails to admit a good formal structure; 
this set may be equipped with the structure of a 
reduced closed subspace of $Z$, by Proposition~\ref{P:open locus}
below.
\end{defn}

\begin{remark}
One might consider the possibility that the restriction of
$\calE$ to $\calO_{\widehat{X|Z},x}(*Z)$ itself admits a ramified good
decomposition, or in Sabbah's language, that $\calE$ admits a 
\emph{very good formal structure} at $x$.
However, an argument of Sabbah \cite[Lemme~I.2.2.3]{sabbah} shows that 
one cannot
in general achieve very good formal structures even after blowing up. 
For this reason, we make no further study of very good formal structures.
\end{remark}

\begin{remark}
If $\calE$ is defined over $\calO_X(*Z)$ itself, one can also speak about
good formal structures at points outside of $Z$, but they trivially always
exist.
\end{remark}

\begin{prop} \label{P:open locus}
Suppose $(X,Z)$ is a regular pair. Then
the turning locus of $\calE$ is the underlying set of a unique reduced
closed subspace of $Z$, containing no irreducible component of $Z$.
\end{prop}
\begin{proof}
We first treat the algebraic case.
Suppose $W$ is an irreducible 
closed subset of $Z$ not contained in the turning locus.
By the numerical criterion from Theorem~\ref{T:criterion},
the generic point of $W$ also lies outside the turning locus.
Corollary~\ref{C:descend} then implies that the intersection of $W$ with the
turning locus is contained in some closed proper subset of $W$.
By noetherian induction, it follows that the turning locus is closed in 
$Z$; it thus carries a unique reduced subscheme structure.
Moreover, the turning locus cannot contain the generic point of any component
of $Z$, because at any such point we may apply the usual
Turrittin-Levelt-Hukuhara decomposition theorem (or equivalently,
because the numerical criterion of Theorem~\ref{T:criterion} is always
satisfied when the base ring is one-dimensional). Hence the turning locus
cannot contain
any whole irreducible component of $Z$.

We next reduce the analytic case to the algebraic case.
Recall that $X$ admits a neighborhood basis consisting of compact subsets
of Stein subspaces of $X$. Let $K$ be an element of this basis,
let $U$ be a Stein subspace of $X$ containing $K$,
and let $V$ be an open set contained in $K$.
By Corollary~\ref{C:Stein}, the localization $R$ of $\Gamma(U, \calO_U)$
at $K$ is noetherian, hence a nondegenerate differential ring.
Let $I$ be the ideal of $R$ defined by $Z$, and put $M = \Gamma(U, \calE) \otimes_{\Gamma(U, \calO_U)} R$ as a differential
module over $R$. By the previous paragraph, the turning
locus of $M$ may be viewed as a reduced closed subscheme of $\Spec(R/I)$
not containing any irreducible component.
Its inverse image under the map
$V \to \Spec(R)$ is then a reduced closed subspace of $V \cap Z$
not containing any irreducible component. Since we can choose 
$V$ to cover a neighborhood of any given point of $X$, we deduce the
desired result.
\end{proof}

\subsection{Irregularity and turning loci}
\label{subsec:irregularity}

It will be helpful to rephrase the
numerical criterion for good formal structures
(Theorem~\ref{T:criterion}) in geometric language.
For this, we must first formalize the notion of irregularity.

\begin{defn}
Let $E$ be an irreducible component of $Z$. We define the \emph{irregularity}
$\Irr_E(\calE)$ of $\calE$ along $E$ as follows.

Suppose first that we are in the algebraic case. Let $\eta$ be the generic point
of $E$. Let $L$ be the completion of
$\Frac(\calO_{X,\eta})$, equipped with its discrete valuation
normalized to have value group $\ZZ$.
We define $\Irr_E(\calE)$ as the irregularity
of the differential module over $L$ induced by $\calE$, in the sense of
\cite[Definition~1.4.9]{kedlaya-goodformal1}.

Suppose next that we are in the analytic case; in this case, we use a
``cut by curves'' definition.
We may assume that $(X,Z)$ is a regular pair by discarding its irregular locus
(which has codimension at least 2 in $X$).
Let $T$ be the turning locus of $\calE$; by Proposition~\ref{P:open locus},
$T \cap E$ is a proper closed subspace of $E$, so in particular
its complement is dense in $E$.
We claim that there exists a nonnegative integer $m$ with the following
property: for any curve $C$ in $X$
and any isolated point $z$ of $C \cap E$ not belonging to $T$,
the irregularity of the restriction of $\calE$ to $C$ at the point $z$ is equal
to $m$ times the intersection multiplicity of $C$ and $E$ at $z$.
Namely, it suffices to check this assertion on each element of a basis for the
topology of $X$, which may be achieved as in the proof of
Proposition~\ref{P:open locus}. We define $\Irr_E(\calE) = m$.
\end{defn}

\begin{defn}
An \emph{irregularity divisor} for $\calE$ is a Cartier divisor
$D$ on $X$ such that for any normal modification $f: Y \to X$
and any prime divisor $E$ on $Y$,  the irregularity
of $f^* \calE$ along $E$
is equal to the multiplicity of $f^* D$ along $E$. Such a divisor
is unique if it exists. Moreover, any $\QQ$-Cartier divisor satisfying
the definition must have all integer multiplicities, and so must be an 
integral Cartier divisor.
\end{defn}

\begin{prop} \label{P:geom numerical}
Suppose that $(X,Z)$ is a regular pair.
Then the following conditions are equivalent.
\begin{enumerate}
\item[(a)]
The turning locus of $\calE$ is empty.
\item[(b)]
Both $\calE$ and $\End(\calE)$ admit irregularity divisors.
\end{enumerate}
\end{prop}
\begin{proof}
Given (b), (a) follows by Theorem~\ref{T:criterion}. Given (a),
we may check (b) locally around a point $x \in Z$.
We make a sequence of reductions to successively more restrictive
situations, culminating in one where we can read off the claim.
Namely, we reduce so as to enforce the following hypotheses.
\begin{enumerate}
\item[(a)]
There exist a regular sequence of parameters $x_1,\dots,x_n \in \calO_{X,x}$
for $X$ at $x$ (by shrinking $X$).
\item[(b)]
We have $Z = V(x_1 \cdots x_m)$ (by shrinking $X$).
\item[(c)]
The module $\calE$ admits a good decomposition at $x$
(by shrinking $X$, then
replacing $X$ by a finite cover ramified along $Z$).
\item[(d)]
With notation as in \eqref{eq:local model1},
the $\phi_\alpha$ belong to $\calO_{X,x}[x_1^{-1},\dots,x_m^{-1}]$
(by applying Proposition~\ref{P:descend good ring}).
\end{enumerate}
In this case, we claim that in some neighborhood of $x$,
the irregularity divisor of $\calE$ is the sum of the
principal divisors $-\rank(\calR_\alpha) \divis(\phi_\alpha)$
over all $\alpha \in I$ with $\phi_\alpha \notin \calO_{X,x}$,
while the irregularity divisor of $\End(\calE)$ is the sum of the
principal divisors $-\rank(\calR_\alpha) \rank(\calR_\beta)
\divis(\phi_\alpha-\phi_\beta)$
over all $\alpha,\beta \in I$ with $\phi_\alpha - \phi_\beta \notin 
\calO_{X,x}$.
This may be checked as in the proof of Proposition~\ref{P:open locus}.
\end{proof}

\subsection{Deligne-Malgrange lattices}
\label{subsec:deligne-malgrange}

In the work of Mochizuki \cite{mochizuki2},
the approach to constructing good formal structures is via the analysis
of Deligne-Malgrange lattices. Since we use a different technique to construct
good formal structures, it is worth indicating how to recover information
about Deligne-Malgrange lattices.

\begin{defn}
Let $\calF$ be a coherent sheaf over 
$\calO_X(*Z)$ (resp.\ 
over $\calO_{\widehat{X|Z}}(*Z)$).
A \emph{lattice} of $\calF$ is a coherent 
$\calO_X$-submodule (resp.\ $\calO_{\widehat{X|Z}}$-submodule)
$\calF_0$ of $\calF$
such that the induced
map 
$\calF_0 \otimes_{\calO_X} \calO_{X}(*Z) \to \calF$ 
(resp.\ 
$\calF_0 \otimes_{\calO_{\widehat{X|Z}}} \calO_{\widehat{X|Z}}(*Z) \to
\calF$) is surjective. We make the following observations.
\begin{enumerate}
\item[(a)]
Let $\calF$ be a coherent sheaf over 
$\calO_X(*Z)$, and put $\widehat{\calF} = \calF \otimes_{\calO_X(*Z)}
\calO_{\widehat{X|Z}}(*Z)$. Then the map
\[
\calF_0 \mapsto \widehat{\calF}_0 = \calF \otimes_{\calO_X}
\calO_{\widehat{X|Z}}
\]
gives a bijection between lattices of $\calF$ and lattices of $\widehat{\calF}$,
as in \cite[Proposition~1.2]{malgrange-reseau}.
Moreover, $\calF_0$ is locally free if and only if $\widehat{\calF}$ is,
because the completion of a noetherian local ring is faithfully flat
(see Remark~\ref{R:regular local}).
\item[(b)]
In the analytic case,
a coherent sheaf over $\calO_X(*Z)$ or $\calO_{\widehat{X|Z}}(*Z)$ need not admit
any lattices at all. See \cite[Exemples~1.5, 1.6]{malgrange-reseau}.
\end{enumerate}
\end{defn}

By contrast, even in the analytic case,
a $\nabla$-module always admits a lattice which is nearly canonical.
It only depends on a certain splitting of the reduction modulo $\ZZ$ map.
\begin{defn}
In the algebraic case, let $K_0$ be a field containing each
connected component of the subring
of $\Gamma(X, \calO_X)$ killed by the action of all derivations.
(Each of those components is a field by Lemma~\ref{L:nondegenerate}(d).)
In the analytic case, put $K_0 = \CC$.
Let $\overline{K_0}$ be an algebraic closure of $K_0$.
Let $\tau: \overline{K_0}/\ZZ \to \overline{K_0}$ be a section of the quotient 
$\overline{K_0} \to \overline{K_0}/\ZZ$.
We say $\tau$ is \emph{admissible} if $\tau(0) = 0$,
$\tau$ is equivariant for the action of the absolute Galois group of $K_0$,
and for any $\lambda \in \overline{K_0}$ and
any positive integer $a$, we have
\begin{equation} \label{eq:admissible}
\tau(\lambda) - \lambda = 
\left\lceil \frac{\tau(a\lambda) - a\lambda}{a} \right\rceil.
\end{equation}
Such a section always exists by \cite[Lemma~2.4.3]{kedlaya-goodformal1}.
For instance, if $K_0 = \CC$, one may take $\tau$ to have image
$\{s \in \CC: \Real(s) \in [0,1)\}$.
\end{defn}
For the remainder of \S~\ref{subsec:deligne-malgrange}, fix a choice
of an admissible section $\tau$.

\begin{defn}
Suppose that $(X,Z)$ is a regular pair.
A \emph{Deligne-Malgrange lattice} of the $\nabla$-module $\calE$ 
over $\calO_{\widehat{X|Z}}(*Z)$ 
is a lattice $\calE_0$ of $\calE$
such that for each point $x \in Z$, the restriction
of $\calE_0$ to $\widehat{\calO_{X,x}}$ is the Deligne-Malgrange
lattice of the restriction of $\calE$ to $\widehat{\calO_{X,x}}(*Z)$,
in the sense of \cite[Definition~4.5.2]{kedlaya-goodformal1}
(for the admissible section $\tau$).
Such a lattice is evidently unique if it exists.
\end{defn}

\begin{theorem} \label{T:DM1}
Suppose that $(X,Z)$ is a regular pair
and that $\calE$ has empty turning locus.
Then the Deligne-Malgrange lattice $\calE_0$ 
of $\calE$ exists and is locally free over $\calO_{\widehat{X|Z}}$.
\end{theorem}
\begin{proof}
It suffices to check both assertions in case $X$ is the spectrum of a local ring $R$,
$Z$ is the zero locus of $x_1 \cdots x_m$ for some regular sequence of parameters
$x_1,\dots,x_n$ of $R$, and $R$ is complete with respect to the $(x_1 \cdots x_m)$-adic 
topology (but not necessarily with respect to the $(x_1, \dots, x_m)$-adic topology).
For $i=1,\dots,m$,
let $F_i$ be the $x_i$-adic completion of $\Frac(R)$,
put $\calE_i = \calE \otimes_i F_i$, and let $\calE_{0,i}$ be the
Deligne-Malgrange lattice of $\calE_i$.
We define $\calE_0$ to be the $R$-submodule of $\calE$ consisting of elements
whose image in $\calE_i$ belongs to $\calE_{0,i}$ for $i=1,\dots,m$. 
As in
\cite[Lemma~4.1.2]{kedlaya-goodformal1}, we see that $\calE_0$ is a lattice in $\calE$
and that $\calE_0 \otimes_R \widehat{R}$ is the Deligne-Malgrange lattice of
$\calE \otimes_{R[x_1^{-1},\dots,x_m^{-1}]} \widehat{R}[x_1^{-1},\dots,x_m^{-1}]$.
In particular, $\calE_0 \otimes_R \widehat{R}$ is a finite free $\widehat{R}$-module
by \cite[Proposition~4.5.4]{kedlaya-goodformal1}, 
so $\calE_0$ is a finite free $R$-module by faithful flatness of completion 
(Remark~\ref{R:regular local} again).
\end{proof}

\begin{remark} \label{R:canonical lattice}
Let $U$ be the open (by Proposition~\ref{P:open locus})
subspace of $X$ on which $(X,Z)$ is a regular pair and $\calE$
has no turning locus.
Malgrange \cite[Th\'eor\`eme~3.2.1]{malgrange-reseau} constructed a Deligne-Malgrange
lattice over $U$; this construction is reproduced by our 
Theorem~\ref{T:DM1}. Malgrange then went on to establish the much deeper fact that
this lattice extends over all of $X$
\cite[Th\'eor\`eme~3.2.2]{malgrange-reseau}. We will only reproduce this result after
establishing existence of good formal structures after blowing up; see
Theorem~\ref{T:deligne-malgrange}.
\end{remark}

The following property of Deligne-Malgrange lattices follows immediately from
Proposition~\ref{P:good formal}; we formulate it to make a link with Mochizuki's
work. See Remark~\ref{R:mochizuki}.

\begin{prop} \label{P:good DM}
Suppose that $(X,Z)$ is a regular pair
and that $\calE$ has empty turning locus.
Choose any $x \in Z$. For $U$ an open neighborhood of $x$ in $X$,
$f: U' \to U$ a finite cover ramified over $Z$,
and $y \in f^{-1}(x)$, put $Z' = f^{-1}(Z)$ and let $Y$ denote the
intersection of the irreducible components of $Z'$ passing through $y$.
Then we can choose $U$ and $f$ so that for any $y \in f^{-1}(x)$,
any admissible decomposition of the restriction of $f^* \calE$ to
$\widehat{\calO_{U',y}}(*Z')$
induces a corresponding decomposition of the
restriction to $\calO_{\widehat{U'|Y},y}$ of the
Deligne-Malgrange lattice $\calE'_0$ of $f^* \calE$.
\end{prop}

\begin{remark} \label{R:mochizuki}
Suppose that $(X,Z)$ is a regular pair, and that both 
$\calE$ and $\End(\calE)$ have empty turning locus.
(The restriction on $\End(\calE)$ is needed to overcome the discrepancy
between our notion of a good decomposition and Mochizuki's definition
of a good set of irregular values; see 
\cite[Remark~4.3.3, Remark~6.4.3]{kedlaya-goodformal1}.)
The conclusion of Proposition~\ref{P:good DM}
asserts that
$\calE'_0$ is an \emph{unramifiedly good Deligne-Malgrange lattice}
in the language of Mochizuki \cite[Definition~5.1.1]{mochizuki2}.

By virtue of the definition of Deligne-Malgrange
lattices in the one-dimensional case 
\cite[Definition~2.4.4]{kedlaya-goodformal1},
it is built into the definition of Deligne-Malgrange lattices in general
that $f_* \calE'_0 = \calE_0$.
Hence $\calE_0$ is a \emph{good Deligne-Malgrange lattice} in the language
of Mochizuki; that is,
Proposition~\ref{P:good DM} fulfills a promise made in
\cite[Remark~4.5.5]{kedlaya-goodformal1}.
\end{remark}

\section{The Berkovich unit discs}
\label{sec:berkovich}

In \cite[\S 5]{kedlaya-goodformal1}, we introduced the Berkovich closed
and open unit discs 
over a complete discretely valued field of equal characteristic $0$,
and used their geometry to make a fundamental finiteness argument 
as part of the proof of Sabbah's conjecture. Here, we need the analogous construction over an arbitrary complete nonarchimedean
field of characteristic $0$. To achieve this level of generality, 
we must recall some results
from \cite[\S 2]{kedlaya-part4}, and make some arguments as in
\cite[\S 4]{kedlaya-part4}.

\setcounter{theorem}{0}
\begin{hypothesis} \label{H:berkovich}
Throughout \S~\ref{sec:berkovich},
let $F$ be a field complete for a nonarchimedean norm $|\cdot|_F$,
of residual characteristic $0$.
Define the real valuation $v_F$ by $v_F(\cdot) = -\log |\cdot|_F$.
Let $\CC_F$ denote a completed algebraic closure of $F$, equipped
with the unique extensions of $|\cdot|_F$ and $v_F$.
\end{hypothesis}

\begin{notation}
For $A$ a subring of $F$
and $\rho > 0$, let $|\cdot|_\rho$ denote the $\rho$-Gauss norm on
$A[x, x^{-1}]$ with respect to $|\cdot|_F$. 
For $\alpha \leq \beta \in (0, +\infty)$,
define the following rings.
\begin{itemize}
\item
Let $A \langle \alpha/x \rangle$ denote the completion
of $A[x^{-1}]$ under $|\cdot|_{\alpha}$.
\item
Let $A \langle x/\beta \rangle$ denote the
completion of $A[x]$ under $|\cdot|_\beta$.
\item
Let $A \langle \alpha/x, x/\beta \rangle$
denote the Fr\'echet completion of $A[x,x^{-1}]$
under $|\cdot|_\alpha$ and $|\cdot|_\beta$
(equivalently, under $|\cdot|_\rho$ for all $\rho \in [\alpha,\beta]$).
\end{itemize}
For $\beta = 1$, we abbreviate $A \langle x/\beta \rangle,
A \langle \alpha/x, x/\beta \rangle$ to $A \langle x \rangle,
A \langle \alpha/x, x \rangle$.
Note that none of these rings changes if we replace $A$ by its
completion under $|\cdot|_F$.
\end{notation}

\subsection{The Berkovich closed unit disc}

We first recall a few facts about the Berkovich closed unit disc
over the field $F$. The case $F = \CC((x))$ was treated in
\cite[\S 5]{kedlaya-goodformal1}, but we need to reference the more
general treatment in \cite[\S 2.2]{kedlaya-part4}. (Note that the treatment
there allows positive residual characteristic, which we exclude here.)

\begin{defn} \label{D:closed disc}
The Berkovich closed unit disc $\DD = \DD_F$ consists of the 
multiplicative seminorms $\alpha$ on $F[x]$ which are compatible with the
given norm on $F$ and bounded above by
the 1-Gauss norm. 
For instance, for 
$z \in \gotho_{\CC_F}$ and $r \in [0,1]$, the function $\alpha_{z,r}: 
F[x] \to [0, +\infty)$ taking $P(x)$ to the $r$-Gauss norm
of $P(x+z)$ is a seminorm; it is in fact the supremum seminorm 
on the disc $D_{z,r} = \{z' \in \CC_F: |z-z'| \leq r\}$.
\end{defn}

\begin{lemma} \label{L:Berkovich surjective}
For any complete extension $F'$ of $F$, the restriction map
$\DD_{F'} \to \DD_F$ is surjective.
\end{lemma}
\begin{proof}
See \cite[Corollary~1.3.6]{berkovich}.
\end{proof}

\begin{defn}
For $\alpha, \beta \in \DD$, we say that $\alpha$ \emph{dominates}
$\beta$, notated $\alpha \geq \beta$, if $\alpha(P) \geq \beta(P)$ for
all $P \in F[x]$. Define the \emph{radius} of $\alpha \in \DD$,
denoted $r(\alpha)$,
to be the infimum of $r \in [0,1]$ for which there exists $z \in \gotho_{\CC_F}$
with $\alpha_{z,r} \geq \alpha$.
\end{defn}

As in \cite[Proposition~5.2.2]{kedlaya-goodformal1}, we use the following
classification of points of $\DD$. See
\cite[1.4.4]{berkovich} for the case where $F$ is algebraically
closed, or \cite[Proposition~2.2.7]{kedlaya-part4} for the general case.
\begin{prop} \label{P:classify points}
Each element of $\DD$ is of exactly one of the following four types.
\begin{enumerate}
\item[(i)]
A point of the form $\alpha_{z,0}$ for some $z \in \gotho_{\CC_F}$.
\item[(ii)]
A point of the form $\alpha_{z,r}$ for some $z \in \gotho_{\CC_F}$
and $r \in (0,1] \cap |\CC_F^\times|_F$.
\item[(iii)]
A point of the form $\alpha_{z,r}$ for some $z \in \gotho_{\CC_F}$
and $r \in (0,1] \setminus
 |\CC_F^\times|_F$.
\item[(iv)]
The infimum of a sequence $\alpha_{z_i,r_i}$ in which the discs $D_{z_i,r_i}$
form a decreasing sequence
with empty intersection and positive limiting radius. 
\end{enumerate}
Moreover, the points which are minimal under domination
are precisely those of type (i) and (iv).
\end{prop}

By \cite[Lemma~2.2.12]{kedlaya-part4}, we have the following.
\begin{lemma} \label{L:unique point}
For each $\alpha \in \DD$ and each $r \in [r(\alpha),1]$,
there is a unique point $\alpha_r \in \DD$ with $r(\alpha_r) = r$
and $\alpha_r \geq \alpha$. (By Proposition~\ref{P:classify points},
if $r \neq r(\alpha)$, we can always write $\alpha_r = \alpha_{z,r}$
for some $z \in \gotho_\CC$.)
\end{lemma}

\begin{cor} \label{C:stabilize}
If $\alpha \in \DD$ is of type (iv) and is the infimum of the sequence
$\alpha_{z_i,r_i}$, then for any $r \in (0,1)$ with $\alpha_{0,r} \geq \alpha$
and any $P \in F\langle x/r \rangle$, there exists
an index $i_0$ such that $\alpha_{z_i,r_i}(P) = \alpha(P)$ for $i \geq i_0$.
\end{cor}
\begin{proof}
The case $P \in F[x]$ follows from
the proof of \cite[Proposition~2.2.7]{kedlaya-part4}. The general
case follows by choosing $Q \in F[x]$ such that $\alpha_{0,r}(P-Q) <
\alpha(P)$ and applying the previous case to $Q$.
\end{proof}

Except for some points of type (i),
every point of $\DD$ induces
a valuation on $F(x)$. These valuations have the following numerical
behavior \cite[Lemma~2.2.16]{kedlaya-part4}.
\begin{lemma} \label{L:same corank}
Let $\alpha$ be a point of $\DD$ of type (ii) or (iii). Let $v(\cdot) = - \log \alpha(\cdot)$
be the corresponding real valuation on $F(x)$.
\begin{enumerate}
\item[(a)]
If $\alpha$ is of type (ii), then 
\[
\trdeg(\kappa_v / \kappa_{v_F}) = 1, \qquad \dim_\QQ ((\Gamma_v/\Gamma_{v_F}) \otimes_{\ZZ} \QQ) = 0.
\]
\item[(b)]
If $\alpha$ is of type (iii), then 
\[
\trdeg(\kappa_v / \kappa_{v_F}) = 0, \qquad \dim_\QQ ((\Gamma_v/\Gamma_{v_F}) \otimes_{\ZZ} \QQ) = 1.
\]
\end{enumerate}
\end{lemma}

\subsection{More on irrational radius}
\label{subsec:berk descent}

Let us take a closer look at the case of Proposition~\ref{P:classify points}
of type (iii), i.e., a disc of an irrational radius.

\begin{hypothesis}
Throughout \S~\ref{subsec:berk descent},
in addition to Hypothesis~\ref{H:berkovich},
choose $r \in (0,1) \setminus |\CC_F^\times|_F$,
so that $\alpha_{0, r} \in \DD$ is a point of type (iii).
(The case where $\alpha_{0,r}$ is of type (ii) is a bit more complicated,
and we will not need it here.)
\end{hypothesis}

\begin{lemma} \label{L:irrational factor}
Suppose $g \in F \langle r/x,x/r \rangle$
is such that $\alpha_{0,r}(g-1) < 1$.
Then $g$ can be factored uniquely as $g_1 g_2$ with
$g_1 \in 1 + x F \langle x/r \rangle^\times$,
$g_2 \in F \langle r/x \rangle^\times$, 
$\alpha_{0,r}(g_1-1) < 1$, and $\alpha_{0,r}(g_2 - 1) < 1$.
In particular, $g$ is a unit in $F \langle r/x, x/r \rangle$.
\end{lemma}
\begin{proof}
Apply Theorem~\ref{T:master factor} with
$U = x F \langle x/r \rangle$,
$V = F \langle r/x \rangle$,
$W = F \langle r/x,x/r \rangle$,
$a = b = 1$, and $c = g$.
(Compare the proof of Proposition~\ref{P:dagger factor}
or \cite[Theorem~2.2.1]{kedlaya-course}.)
\end{proof}

\begin{lemma} \label{L:irrational local}
Any nonzero $g \in F \langle r/x, x/r \rangle$ can be factored (not uniquely) as
$g = x^i g_1 g_2$ for some
$i \in \ZZ$, 
$g_1 \in F\langle x/r \rangle ^\times$,
and $g_2 \in F \langle r/x \rangle ^\times$.
\end{lemma}
\begin{proof}
(Compare \cite[Lemma~2.2.14]{kedlaya-part4}.)
Write $g = \sum_{i \in \ZZ} g_i x^i$ with $|g_i|_F r^i \to 0$ as $i \to \pm \infty$.
Since $r \notin |\CC_F^\times|_F$, there exists a unique index $j$
which maximizes $|g_j|_F r^j$.
Lemma~\ref{L:irrational factor} implies that $g_j^{-1} x^{-j} g$
factors as a unit in $F \langle x/r \rangle$ times a unit in
$F \langle r/x \rangle$,
yielding the claim.
\end{proof}
\begin{cor}
The completion of $F(x)$ under $\alpha_{0,r}$ is equal to $F \langle r/x,x/r
\rangle$.
\end{cor}
\begin{proof}
The complete ring $F \langle r/x,x/r \rangle$ contains $F[x,x^{-1}]$
as a dense subring, and is a field by Lemma~\ref{L:irrational local}.
This proves the claim.
\end{proof}

\begin{lemma} \label{L:irrational normal2}
Suppose $F$ is integrally closed in the complete extension $F'$.
Then $F \langle r/x, x/r
\rangle$ is integrally closed in $F' \langle r/x, x/r \rangle$.
\end{lemma}
\begin{proof}
We may reduce to the case where both $F$ and $F'$ are algebraically closed.
Let $f = \sum_{i \in \ZZ} f_i x^i
\in F' \langle r/x, x/r \rangle$ be an element which is integral over
$F \langle r/x, x/r \rangle$.
Let $P(T)$ be the minimal polynomial of $f$ over $F \langle r/x, x/r \rangle$.
Then for each $\tau \in \Aut(F'/F)$, $\sum_{i \in \ZZ} \tau(f_i) x^i$ is also
a root of $P$. Hence each $f_i$ must have finite orbit under $\tau$, and so
must belong to $F$.
\end{proof}

\begin{prop} \label{P:irrational glueing}
Let $A$ be a subring of $F$.
Let $S$ be a complete subring of $F \langle r/x,x/r \rangle$
which is topologically finitely generated over $A \langle r/x,x/r \rangle$, 
such that $\Frac(S)$ is finite over $\Frac(A \langle r/x, x/r \rangle)$.
Then there exists a subring $A'$ of $F$ 
which is finitely generated over $A$,
such that $\Frac(A')$ is finite over $\Frac(A)$ and
$S \subseteq A' \langle r/x,x/r \rangle$.
\end{prop}
\begin{proof}
It suffices to check the claim in case $S$ is the completion of $A \langle r/x,x/r
\rangle[g]$ for some $g \in F \langle r/x, x/r \rangle$ which is integral over
$\Frac(A \langle r/x,x/r \rangle)$, with minimal polynomial $P(T)$.
By Lemma~\ref{L:irrational normal2}, 
the coefficients of $g$ must belong to the completion of the integral closure of
$\Frac(A)$ within $F$. 
Consequently, for any $\epsilon >0$, we can choose a finitely generated $A$-subalgebra $A'$ of $F$ 
with $\Frac(A')$ finite over $\Frac(A)$, so that
there exists $h \in A' \langle r/x, x/r \rangle$ with $\alpha_{0,r}(g-h) < \epsilon$.
For $\epsilon$ suitably small, we may then perform a Newton iteration to compute
a root of $P(T)$ in $A' \langle r/x, x/r \rangle$ close to $h$, which will be forced to
equal $g$.
\end{proof}

\subsection{Differential modules on the open unit disc}
\label{subsec:diff mod disc}

We now collect some facts about differential modules on Berkovich discs,
particularly concerning their behavior in a neighborhood of a minimal point.
The hypothesis of residual characteristic $0$ will simplify matters greatly;
an analogous but more involved
treatment in the case of positive residual characteristic
is \cite[\S 4]{kedlaya-part4}.

\begin{hypothesis} \label{H:Berkovich nabla}
Throughout \S~\ref{subsec:diff mod disc}--\ref{subsec:extend hor},
let $M$ be a $\nabla$-module of rank $d$ over the \emph{open} 
Berkovich unit disc
\[
\DD_0 = \{\alpha \in \DD: \alpha_{0,r} \geq \alpha \mbox{ for some } r \in [0,1)\}.
\]
That is, for each $r \in [0,1)$, we must specify
a differential module $M_r$ of rank $d$ over $F \langle x/r \rangle$,
plus isomorphisms $M_r \otimes_{F \langle x/r \rangle} F \langle x/s \rangle
\cong M_s$
for $0 < s < r < 1$ satisfying the cocycle condition.
(Note that $\alpha_{0,0} \notin \DD_0$, contrary to the convention
adopted in \cite[Definition~5.3.1]{kedlaya-goodformal1}.)
\end{hypothesis}

\begin{defn} \label{D:berkovich radii}
For $\alpha \in \DD_0$,
put $I_\alpha = (0, +\infty)$ if $\alpha$ is of type (i)
and $I_\alpha = (0, -\log r(\alpha)]$ otherwise.
For $s \in I_\alpha$,
let $\alpha_s$ be the unique point of $\DD_0$
with $\alpha_s \geq \alpha$ and $r(\alpha_s) = e^{-s}$
(given by Lemma~\ref{L:unique point}).
Let $F_{\alpha,s}$ be the completion of $F(x)$ under $\alpha_s$.
Define $f_1(M,\alpha,s) \geq  \dots \geq f_d(M,\alpha,s) \geq s$
so that the scale multiset of $\frac{\del}{\del x}$
on $M \otimes F_{\alpha,s}$ 
consists of
$e^{f_1(M,\alpha,s)-s}, \dots, e^{f_d(M,\alpha,s)-s}$.
Beware that this is a different normalization than in the definition
of irregularity; the new normalization is such that
\[
\left|\frac{\del}{\del x}\right|_{\mathrm{sp}, M \otimes F_{\alpha,s}} 
 = e^{f_1(M,\alpha,s)}.
\]
Put $F_i(M,\alpha,s) = f_1(M,\alpha,s) + \cdots + f_i(M,\alpha,s)$.
\end{defn}

\begin{prop} \label{P:convex stuff}
The function $F_i(M,\alpha,s)$ is continuous, convex, and 
piecewise affine in $s$,
with slopes in $\frac{1}{d!}\ZZ$. Furthermore, the slopes of $F_i(M,\alpha,s)$
are nonpositive in a neighborhood of any $s$ for which $f_i(M,\alpha,s) > s$.
\end{prop}
\begin{proof}
As in \cite[Proposition~4.6.4]{kedlaya-part4}, this reduces to
\cite[Theorem~11.3.2]{kedlaya-course}.
\end{proof}

The following argument makes critical use of the hypothesis that $F$
has residual characteristic $0$.
The situation of positive residual characteristic is much subtler;
compare \cite[Proposition~4.7.5]{kedlaya-part4}.
\begin{prop} \label{P:terminal}
\begin{enumerate}
\item[(a)]
Suppose $\alpha \in \DD_0$ is of type (i).
Then in a neighborhood of $s = +\infty$,
for each $i \in \{1,\dots,d\}$, $f_i(M,\alpha,s) = s$ identically.
\item[(b)]
Suppose $\alpha \in \DD_0$ is of type (iv).
Then in a neighborhood of $s = -\log r(\alpha)$,
for each $i \in \{1,\dots,d\}$,
either $f_i(M,\alpha,s)$ is constant, or
$f_i(M,\alpha,s) = s$ identically.
\end{enumerate}
\end{prop}
\begin{proof}
Suppose first that $\alpha$ is of type (i).
By Proposition~\ref{P:convex stuff}, $f_1(M,\alpha,s) = F_1(M,\alpha,s)$
is convex, and it cannot have a positive slope except in a stretch
where $f_i(M,\alpha,s) = s$ identically for all $i$. 
In particular, we cannot have
$f_i(M,\alpha,s) - s > 0$ for all $s$: otherwise, the left side
would be a convex function on $(0, +\infty)$
with all slopes less than or equal to $-1$, and so could not be 
positive on an infinite interval. We thus have $f_i(M,\alpha,s_0)= s_0$ for 
some $s_0$; the right slope of $f_i(M,\alpha,s)$ at $s = s_0$
must be at least 1 because
$f_i(M,\alpha,s) \geq s$ for all $s$. By convexity, the right slope of
$f_i(M,\alpha,s)$ at any $s \geq s_0$ must be at least 1.
If we ever encounter a slope strictly greater than 1,
then that slope is achieved at some point with $f_i(M,\alpha,s) > s$,
contradicting Proposition~\ref{P:convex stuff}.
It follows that $f_i(M,\alpha,s) = s$ identically for $s \geq s_0$.

Suppose next that $\alpha$ is of type (iv).
Pick some $r \in (r(\alpha),1)$ such that $\alpha_{0,r} \geq \alpha$, and put
$E  = \Frac(F \langle x/r \rangle)$.
By the cyclic vector theorem \cite[Lemma~1.3.3]{kedlaya-goodformal1},
there exists an isomorphism $M \otimes E
\cong E\{T\}/E\{T\}P(T)$
for some twisted polynomial
$P(T) = \sum_i P_i T^i \in E\{T\}$ with respect to the derivation
$\frac{\del}{\del x}$.
By Corollary~\ref{C:stabilize}, for each $i$, $\alpha_s(P_i)$
is constant for $s$ in a neighborhood of $-\log r(\alpha)$.
Hence the Newton polygon of $P$ measured using $\alpha_s$
is also constant for $s$ in a neighborhood of $-\log r(\alpha)$.
By \cite[Proposition~1.6.3]{kedlaya-goodformal1},
we can read off $f_i(M,\alpha,s)$ as the greater of
$s$ and the negation of the $i$-th smallest
slope of the Newton polygon of $P$.
This implies that in a neighborhood of $-\log r(\alpha)$,
$f_i(M,\alpha,s)$ is either constant or identically equal to $s$,
as desired.
\end{proof}

\subsection{Extending horizontal sections}
\label{subsec:extend hor}

We need some additional arguments that allow us, in certain cases,
to extend horizontal sections of $\nabla$-modules over $\DD_0$.
Throughout \S~\ref{subsec:extend hor}, 
retain Hypothesis~\ref{H:Berkovich nabla}.

\begin{prop} \label{P:prepared}
For $\alpha \in \DD_0$, the following conditions are equivalent.
\begin{enumerate}
\item[(a)]
For $i=1,\dots,d$, $f_i(M, \alpha,s)$ is constant until it becomes equal to
$s$ and then stays equal to $s$ thereafter
(see Figure~\ref{figure:prepared}).
\item[(b)]
There exists $s_1 \in (0, -\log r(\alpha))$
such that for $i=1,\dots,d$, on the range $s \in (0, s_1]$,
$f_i(M, \alpha,s)$ is either constant or identically
equal to $s$.
\item[(c)]
There exists a direct sum decomposition $M = \oplus_t M_t$
of $\nabla$-modules over $\DD_0$,
such that $f_i(M_t,\alpha,s)$ is equal to a constant value depending only on $t$
(not on $i$) until it becomes equal to $s$ and then stays equal to $s$ thereafter.
\end{enumerate}
\end{prop}
\begin{proof}
It is clear that (a) implies (b) and that (c) implies (a). Given (b),
(c) holds by \cite[Theorem~12.4.1]{kedlaya-course}.
\end{proof}

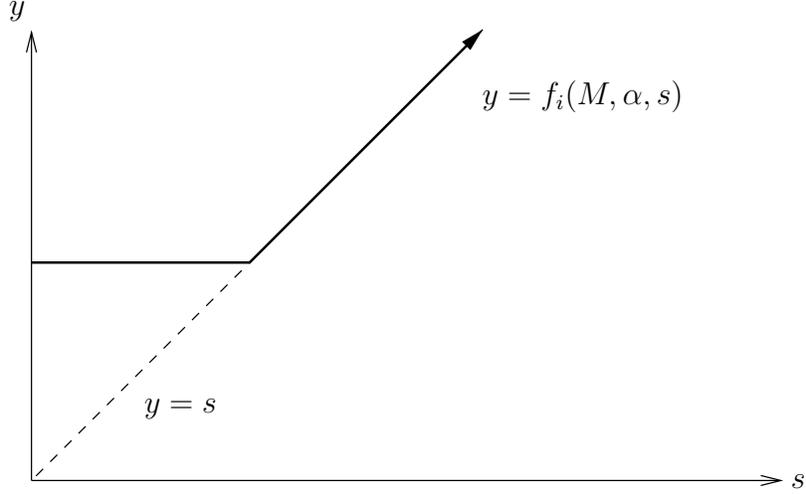
\begin{figure}[htbp]
\begin{center}
\input{gf2-fig1.pstex_t}
\caption{Graph of a function satisfying condition (a) 
of Proposition~\ref{P:prepared}.}
\label{figure:prepared}
\end{center}
\end{figure}

\begin{defn}
We say that $M$ is \emph{terminal} if the equivalent conditions of
Proposition~\ref{P:prepared} are satisfied. Note that condition (b)
does not depend on $\alpha$.
For $s_0 \in [0, -\log r(\alpha))$, if
condition (a) only holds for $s > s_0$, we say that $M$ \emph{becomes 
terminal at $s_0$}; this condition depends on $\alpha$, but only via the
point $\alpha_{s_0}$.

For $I$ a closed subinterval of $I_\alpha \cup \{0\}$,
we say $M$ \emph{becomes strongly terminal on $I$} if 
for $i=1,\dots,d$, over the interior of $I$,
$f_i(M, \alpha,s)$ is either everywhere constant, or everywhere equal to $s$.
By Proposition~\ref{P:prepared}, this implies that $M$ becomes terminal
at the left endpoint of $I$.

\end{defn}

We have the following criterion for becoming strongly terminal.
\begin{prop} \label{P:preparedness}
Suppose that for some $j \in \{0,\dots,d+1\}$ and some
values $0 < s_1 < s_2 < s_3 < -\log r(\alpha)$, we have
\begin{gather*}
f_i(M, \alpha,s_1) = 
f_i(M, \alpha,s_2) = 
f_i(M, \alpha,s_3) \qquad (i=1,\dots,j) \\
f_i(M, \alpha,s_1) - s_1 =
f_i(M, \alpha,s_2) - s_2 =
f_i(M, \alpha,s_3) - s_3 = 0 \qquad (i=j+1,\dots,d).
\end{gather*}
Then $M$ becomes strongly terminal on $[s_1, s_3]$.
In particular, $M$ becomes terminal at $s_1$.
\end{prop}
\begin{proof}
The given conditions imply that for $i=1,\dots,d$,
$F_i(M, \alpha,s)$ agrees with a certain affine function at 
$s = s_1, s_2, s_3$.
Since it is convex by Proposition~\ref{P:convex stuff}, it must be
affine over the range $s \in [s_1, s_3]$. This proves the claim.
\end{proof}

\begin{lemma} \label{L:regular}
Suppose that $f_1(M,\alpha,s) = s$ identically.
Then $M$ admits a basis of horizontal sections.
\end{lemma}
\begin{proof}
This follows from Dwork's transfer theorem 
\cite[Theorem~9.6.1]{kedlaya-course}.
\end{proof}

\begin{prop} \label{P:extend section}
Suppose that for some $s_0 \in (0, -\log r(\alpha))$,
$M$ becomes strongly terminal on $[0, s_0]$.
Then for any $s \in (0,s_0)$, $H^0(M) = H^0(M \otimes F_{\alpha,s})$.
\end{prop}
\begin{proof}
The hypothesis implies that $M$ is terminal.
By Proposition~\ref{P:prepared}, there is a direct sum decomposition
$M = M_0 \oplus M_1$
such that $f_i(M_0,\alpha,s) > s$ for $s$ near $0$ for $i =1,\dots,\rank(M_0)$,
and $f_i(M_1,\alpha,s) = s$ identically for $i=1,\dots,\rank(M_1)$.
Since $M$ becomes strongly terminal on $[0, s_0]$,
we must have $f_i(M_0,\alpha,s) > s$ for $s \in (0, s_0)$
and $i=1,\dots,\rank(M_0)$.

Pick $s \in (0, s_0)$ and $\bv \in H^0(M \otimes F_{\alpha,s})$.
Since $f_i(M_0,\alpha,s) > s$ for $i=1,\dots,\rank(M_0)$,
the projection of $\bv$ onto $M_0 \otimes F_{\alpha,s}$ must be zero;
that is, $\bv \in H^0(M_1 \otimes F_{\alpha,s})$.
However, by Lemma~\ref{L:regular}, $M_1$ admits a basis of horizontal 
sections. If we write $\bv$ in terms of this basis, the coefficients
must be horizontal elements of $F_{\alpha,s}$, but the only
such elements belong to $F$. Hence
$\bv \in H^0(M_1) \subseteq H^0(M)$, as desired.
\end{proof}

\begin{remark}
The restriction that $f_i(M,\alpha,s) \geq s$ is in a certain sense a bit
artificial. Recent work of Baldassarri (in progress, but see
\cite{baldassarri, baldassarri-divizio})
seems to provide a better definition of $f_i(M,\alpha,s)$ that eliminates
this restriction, which would lead to some simplification above.
(Rather more simplification could be expected in the
analogous but more complicated arguments in \cite[\S 5]{kedlaya-part4}.)
\end{remark}

\section{Valuation-local analysis}
\label{sec:local analysis}

We now make the core technical calculations of the paper.
We give a higher-dimensional analogue of the Hukuhara-Levelt-Turrittin
decomposition theorem, in terms of a valuation on a nondegenerate
differential scheme. It will be convenient to state both the result and all of the
intermediate calculations in terms of 
the following running hypothesis.

\setcounter{theorem}{0}
\begin{hypothesis} \label{H:geometric}
Throughout \S~\ref{sec:local analysis},
let $X$ be a nondegenerate differential integral scheme,
let $Z$ be a reduced closed proper subscheme of $X$, and let
$\calE$ be a $\nabla$-module over $X \setminus Z$,
identified with a finite differential module over $\calO_X(*Z)$.
(Note that we do not pass to the formal completion.)
Let $v$ be a centered valuation on $X$, with generic center $z$.
\end{hypothesis}

\subsection{Potential good formal structures}

We introduce the notion of a potential good formal structure 
associated to the valuation $v$, 
and state the theorem we will be proving over the course of this
section. In order to lighten notation, we phrase everything
in terms of ``replacing the input data''.

\begin{defn} \label{D:replacing}
Given an instance of Hypothesis~\ref{H:geometric},
the operation of \emph{modifying/altering the input data} 
will consist of the following.
\begin{itemize}
\item
Let $f: Y \to X$ be a modification/alteration of $X$.
Replace $v$ with an extension $v'$ of $v$ to a centered valuation on $Y$;
such an extension exists by Lemma~\ref{L:proper}.
Since the scheme 
$X$ is excellent by Lemma~\ref{L:nondegenerate}(a),
we have
$\height(v') = \height(v)$,
$\ratrank(v') = \ratrank(v)$, and
$\trdefect(v') = \trdefect(v)$ by Lemma~\ref{L:extension}.
\item
Replace $X$ with an open subscheme $X'$ of $Y$ on which $v'$ is centered.
Replace $Z$ with $Z' = X' \cap f^{-1}(Z)$ (viewed as a reduced closed subscheme
of $X'$).
\item
Replace $\calE$ with the restriction of $f^* \calE$ to $\calO_{X'}(*Z')$.
\end{itemize}
Note that composing operations of one of these forms gives another operation
of the same form.
\end{defn}

\begin{lemma} \label{L:set notation}
Given any instance of Hypothesis~\ref{H:geometric}, after modifying
the input data, we can enforce the following conditions.
\begin{enumerate}
\item[(a)]
The field $\kappa_v$ is algebraic over the residue field $k$ of $X$ at $z$.
\item[(b)]
The pair $(X,Z)$ is regular.
\item[(c)]
The scheme $X = \Spec R$ is affine.
\item[(d)]
There exists a regular system of parameters $x_1,\dots,x_n$ of $R$ at $z$,
such that $Z = V(x_1\cdots x_m)$
for some nonnegative integer $m$.
\item[(e)]
There exists an isomorphism $\widehat{\calO_{X,z}}(*Z)
\cong k \llbracket x_1,\dots,x_n \rrbracket[x_1^{-1},\dots,x_m^{-1}]$
for $k$ as in (a), and derivations $\del_1,\dots,\del_n \in \Delta_R$
acting as $\frac{\del}{\del x_1},\dots,\frac{\del}{\del x_n}$.
\end{enumerate}
\end{lemma}
\begin{proof}
To enforce (a), we may decrease $\trdeg(\kappa_v/k)$ by 
picking $g \in \gotho_v$ whose image in $\kappa_v$ is transcendental
over $k$, then blowing up to force one of $g$ or $g^{-1}$ into $\calO_{X,z}$.
This condition persists under all further modifications,
so we may additionally enforce
(b) using Theorem~\ref{T:desing1}. We can enforce the other
conditions by shrinking $X$, and in the case of (e) invoking
Corollary~\ref{C:complete nondegenerate}.
\end{proof}

\begin{prop} \label{P:potential good}
The following conditions are equivalent.
\begin{enumerate}
\item[(a)]
After modifying the input data,
$\calE$ admits a good formal structure at $z$
(or equivalently a ramified good decomposition, by
Proposition~\ref{P:good formal}).
\item[(b)]
After altering the input data,
$\calE$ admits a good decomposition at $z$.
\item[(c)]
After altering the input data,
$\calE$ admits an admissible decomposition at $z$.
\item[(d)]
After altering the input data,
the restriction of $f^* \calE$ to $\widehat{\calO_{X,z}}(*Z)$
admits a filtration with successive quotients of rank $1$.
\end{enumerate}
\end{prop}
\begin{proof}
Given (a), (b) is evident. Given (b), let $f: Y \to X$ be an alteration
such that $f^* \calE$ admits a good decomposition at the generic
center of some extension
of $v$. By Lemma~\ref{L:flatify}, we can find a modification $g: X' \to X$
such that the proper transform $h$ of $f$ under $g$ is finite flat.
By Theorem~\ref{T:desing1}, we can choose $X'$ so that $(X', Z')$ is a regular
pair, for $Z'$ the union of $g^{-1}(Z)$ with the branch locus of $h$.
Then $g^* \calE$ admits a ramified
good decomposition at the generic center of $v$ on
$X'$, yielding (a).

Given (b), (c) is evident. Given (b), (d)
holds by Proposition~\ref{P:filtration}.
It remains to show that each of (c) and (d) implies (b); we give the argument
for (d), as the argument for (c) is similar but simpler.

Given (d), set notation as in Lemma~\ref{L:set notation}.
By Proposition~\ref{P:twist-regular}, each quotient of the filtration
has the form $E(\phi_\alpha) \otimes \calR_\alpha$ for some $\phi_\alpha \in 
\widehat{\calO_{X,z}}(*Z)$ and some regular differential module $\calR_\alpha$
over $\widehat{\calO_{X,z}}(*Z)$.
By Proposition~\ref{P:descend good ring}, we can choose the $\phi_\alpha$
in $\calO_{X,z}(*Z)$.

After altering the input data, we can ensure that the $\phi_\alpha$
obey conditions (a) and (b) of Definition~\ref{D:local model1}.
This does not give a good decomposition directly, because
we do not have a splitting of the filtration.
On the other hand, if $M$ is the restriction of $\calE$ to
$\widehat{\calO_{X,z}}(*Z)$, and $M^{\semis}$ denotes the semisimplification
of $M$, then $F(M,r) = F(M^{\semis},r)$ 
and $F(\End(M),r) = F(\End(M^{\semis}),r)$ for all $r$,
and $M^{\semis}$ does admit a good decomposition.
Using Theorem~\ref{T:criterion}, we deduce that $M$ admits a good formal
structure, yielding (b).
\end{proof}

\begin{defn} \label{D:potential good}
Under Hypothesis~\ref{H:geometric},
we say that $\calE$ admits a \emph{potential good formal structure} at $v$
if any of the equivalent conditions of Proposition~\ref{P:potential good}
are satisfied. 
\end{defn}

\begin{remark} \label{R:reduce point}
Note that $\calE$ admits a potential good formal structure
(as a meromorphic differential module over $X$)
if and only if $\calE_z$ does so
(as a meromorphic differential module over $\Spec(\calO_{X,z})$). Thus for the
purposes of checking the existence of potential good formal structures,
we may always reduce to the case where $X$ is the spectrum of a local ring
and $z$ is the closed point.

On the other hand, we may not replace
$\Spec(\calO_{X,z})$ by its completion, because not every alteration of
the completion corresponds to an alteration of $\Spec(\calO_{X,z})$.
For instance, if $X = \Spec (k[x_1,x_2,x_3])$ and $z$ is the origin,
then blowing up the completion of $X$ at $z$ at the ideal
$(x_1 - f(x_2), x_3)$ fails to descend if $f \in k \llbracket x_2 \rrbracket$
is transcendental over $k(x_2)$.
\end{remark}

\begin{remark} \label{R:enlarge Z}
Note that if $Z'$ is another closed proper subscheme of $X$ containing $Z$,
and the restriction of $\calE$ to $\calO_X(*Z')$ admits a potential good
formal structure at $v$, then $\calE$ also admits a potential good
formal structure at $v$. This is true because the numerical
criterion for good formal structures (Theorem~\ref{T:criterion},
or Proposition~\ref{P:geom numerical})
is insensitive to adding extra singularities.
\end{remark}

The rest of this section is devoted to proving the following theorem.
\begin{theorem} \label{T:higher HLT}
For any instance of Hypothesis~\ref{H:geometric},
$\calE$ admits a potential good formal structure at $v$.
\end{theorem}
\begin{proof}[Outline of proof]
We prove the theorem by induction on the height and transcendence defect
of $v$, as follows.
\begin{itemize}
\item
We first note that Theorem~\ref{T:higher HLT} holds trivially 
for $v$ trivial. This is the only case for which $\height(v) = 0$.
\item
We next prove that Theorem~\ref{T:higher HLT} holds in all cases where
$\height(v) = 1$ and $\trdefect(v) = 0$. See Lemma~\ref{L:local decomp1}.
\item
We next prove that for any positive integer $e$,
if Theorem~\ref{T:higher HLT} holds in all cases where
$\height(v) = 1$ and $\trdefect(v) <  e$, then it also holds in all cases where
$\height(v) = 1$ and $\trdefect(v) =  e$.
See Lemma~\ref{L:local decomp2}.
\item
We finally prove that for any integer $h> 1$,
if Theorem~\ref{T:higher HLT} holds in all cases where
$\height(v) < h$, then it also holds in all cases where
$\height(v) = h$. See Lemma~\ref{L:local decomp3}.
\end{itemize}
\end{proof}

\subsection{Abyhankar valuations}

We begin the proof of Theorem~\ref{T:higher HLT} by analyzing valuations
of height 1 and transcendence defect 0, i.e.,
all real Abhyankar valuations.
This analysis relies on the simple description of such valuations
in local coordinates.

\begin{lemma} \label{L:abhyankar}
Suppose that $\height(v) = 1$ and $\trdefect(v) = 0$. 
After modifying the input data, 
in addition to the conditions of Lemma~\ref{L:set notation},
we can ensure that the following conditions hold.
\begin{enumerate}
\item[(f)]
The value group of $v$ is freely generated by
$v(x_1),\dots,v(x_n)$.
\item[(g)]
The valuation $v$ is induced by the $(v(x_1),\dots,v(x_n))$-Gauss valuation on
$\widehat{\calO_{X,z}}$.
\end{enumerate}
\end{lemma}
\begin{proof}
This is a consequence of the equality case of Theorem~\ref{T:Abhyankar}.
See \cite{knaf-kuhlmann} for details. (See Lemma~\ref{L:not abhyankar} for
a similar argument.)
\end{proof}

\begin{lemma} \label{L:local decomp1}
For any instance of Hypothesis~\ref{H:geometric}
in which $\height(v) = 1$ and $\trdefect(v) = 0$,
$\calE$ admits a potential good formal structure at $v$.
\end{lemma}
\begin{proof}
Set notation as in Lemma~\ref{L:abhyankar}.
Normalize the embedding of the value group of $v$ into $\RR$ so that
$v(x_1 \cdots x_n) = 1$, and put  $\alpha = (v(x_1),\dots,v(x_n))$,
so that the components of $\alpha$ are linearly independent over $\QQ$
(viewing $\RR$ as a vector space over $\QQ$).

By Theorem~\ref{T:convex}, $F(\calE,r)$ and $F(\End(\calE),r)$
are piecewise integral linear in  $r$. 
Since $\alpha$ lies on no rational hyperplane,
it lies in the interior of a simultaneous
domain of linearity for  $F(\calE,r)$ and $F(\End(\calE),r)$.
Write this domain as the intersection of finitely many closed rational
halfspaces. We modify the input data as follows: for each
of these halfspaces, choose a defining inequality $m_1 r_1 + \cdots m_n r_n
\geq 0$ with $m_1,\dots,m_n \in \ZZ$, then ensure that
$x_1^{m_1} \cdots x_n^{m_n}$ becomes regular.
(This amounts to making a toric blowup in $x_1,\dots,x_n$.)

After this modification of the input data,
$F(\calE,r)$ and $F(\End(\calE),r)$ 
become linear functions of $r$.
By Theorem~\ref{T:criterion}, $\calE$ has a good formal structure at $z$,
as desired.
\end{proof}

\subsection{Increasing the transcendence defect}
\label{subsec:increasing}

We now take the decisive step from real Abhyankar valuations to
real valuations of higher transcendence defect.
For this, we need to invoke the analysis of $\nabla$-modules on Berkovich discs
made in \S~\ref{sec:berkovich}.
The overall structure of the argument is inspired 
directly by \cite[\S 5]{kedlaya-part4},
and somewhat less directly by \cite{temkin-unif}; see
Remark~\ref{R:sketch} for a summary in terms of the relevant notations.
(Note that this is the only step where we make essential use of alterations rather than
modifications.)

\begin{lemma} \label{L:perron}
Let $r \leq s$ be positive integers.
Let $c_1, \dots, c_s$ be positive real numbers such that
$c_1, \dots, c_r$ form a basis for the $\QQ$-span of 
$c_1, \dots, c_s$.
Then there exists a matrix $A \in \mathrm{GL}_s(\ZZ)$ such that $A^{-1}$ has 
nonnegative entries, and 
\[
\sum_{j=1}^s A_{ij} c_j > 0 \qquad (i=1, \dots, r),
\qquad
\sum_{j=1}^s A_{ij} c_j = 0 \qquad (i=r+1, \dots, s).
\]
\end{lemma}
\begin{proof}
The general case follows from the case $r = s-1$,
which is due to Perron. See \cite[Theorem~1]{zariski-local}.
\end{proof}

The following statement and proof,
a weak analogue of Lemma~\ref{L:abhyankar}, are close to those of
\cite[Lemma~2.3.5]{kedlaya-part4}.
(Compare also \cite[Lemma~5.1.2]{kedlaya-part4}.)
\begin{lemma} \label{L:not abhyankar}
Suppose that $\height(v) = 1$.
After modifying the input data and enlarging $Z$,
in addition to the conditions of Lemma~\ref{L:set notation},
we can ensure that the following condition holds.
\begin{enumerate}
\item[(f)]
We have $m = \ratrank(v)$,
and $v(x_1),\dots,v(x_m)$ are linearly independent
over $\QQ$.
\end{enumerate}
\end{lemma}
\begin{proof}
Put $r = \ratrank(v)$.
We first choose $a_1,\dots,a_r \in \gotho_v$ whose
valuations are linearly independent over $\QQ$.
We shrink $X$ to ensure that $a_1,\dots,a_r \in \Gamma(X, \calO_X)$,
then enlarge $Z$ to ensure that $a_1,\dots,a_r \in \Gamma(X \setminus Z,
\calO_X^\times)$.

Modify the input data as in Lemma~\ref{L:set notation},
then change notation by replacing the labels
$x_1,\dots,x_n$ with $x'_1,\dots,x'_n$ and the label $m$ with $s$.
In this notation, each $a_i$ generates the same ideal as some monomial in
$x'_1,\dots,x'_s$.
This implies that $v(x'_1),\dots,v(x'_s)$ must also be linearly
independent over $\QQ$. Since it is harmless to reorder the indices
on $x'_1,\dots,x'_s$, we can ensure that in fact
$v(x'_1),\dots,v(x'_r)$ are linearly independent over $\QQ$.

Fix an embedding of $\Gamma_v$ into $\RR$.
Apply Lemma~\ref{L:perron} with $(c_1,\dots,c_s) = (v(x'_1),\dots,v(x'_s))$,
then put
\[
y_i = \prod_{j=1}^s x_j^{A_{ij}} \qquad (i=1,\dots,s).
\]
By modifying the input data (again with a toric blowup), we
end up with a new ring $R$ with local coordinates $y_1,\dots,y_s,x'_{s+1},
\dots,x'_n$ at the center of $v$.
Note that for $i=r+1,\dots,s$, we have $v(y_i) = 0$.
Since $\trdeg(\kappa_v/k) = 0$, $y_i$ must generate an element of $\kappa_v$
which is algebraic over $k$. Hence $y_i \in \calO_{X,z}^\times$,
so $Z$ must now be the zero locus of $y_1\cdots y_r$.
We may now achieve the desired result by taking $m = r$
and using any regular sequence of parameters starting with
$y_1,\dots,y_m$.
\end{proof}

We now state the desired result of this subsection, giving the
induction on transcendence defect.
\begin{lemma} \label{L:local decomp2}
Let $e > 0$ be an integer.
Suppose that for any instance of Hypothesis~\ref{H:geometric}
with $\height(v) = 1$ and $\trdefect(v) < e$,
$\calE$ admits a potential good formal structure at $v$.
Then for any instance of Hypothesis~\ref{H:geometric}
with $\height(v) = 1$ and $\trdefect(v) = e$,
$\calE$ admits a potential good formal structure at $v$.
\end{lemma}
We will break up the proof of Lemma~\ref{L:local decomp2}
into several individual lemmata. These will all be
stated in terms of the following running hypothesis.

\begin{hypothesis} \label{H:not abhyankar}
During the course of proving Lemma~\ref{L:local decomp2}, we will carry
hypotheses as follows.
Let $e$ be a positive integer such that for any instance of
Hypothesis~\ref{H:geometric} with $\height(v) = 1$ and $\trdefect(v) < e$,
$\calE$ admits a good formal structure at $v$.

Choose an instance of Hypothesis~\ref{H:geometric}
in which $\height(v) = 1$ and $\trdefect(v) = e$.
Fix an embedding of $\Gamma_v$ into $\RR$.
Put $d = \rank(\calE)$.
Set notation as in Lemma~\ref{L:not abhyankar} (after possibly enlarging $Z$,
which is harmless thanks to Remark~\ref{R:enlarge Z}).
\end{hypothesis}

In terms of this hypothesis, we may set some more notation.
\begin{defn} \label{D:not abhyankar}
Put $R_n = R/x_n R$
and let $\widehat{R}$ be the $x_n$-adic completion of $R$;
note that by Lemma~\ref{L:complete nondegenerate},
we may identify $\widehat{R}$ with $R_n \llbracket x_n \rrbracket$.
Extend $v$ by continuity to a real semivaluation on $\widehat{R}$
(see Remark~\ref{R:extend valuation}), then let
$v_n$ be the restriction to $R_n$.

For any real semivaluation $w$ on $R_n$, put $\gothp_w = w^{-1}(+\infty)$,
so that $w$ induces a true valuation on $R_n/\gothp_w$.
Let $\ell(w)$ denote the completion of $\Frac(R_n/\gothp_w)$ under $w$, 
carrying the norm $e^{-w(\cdot)}$.
By extending scalars to $\ell(w) \llbracket x_n \rrbracket$,
form the restriction $N_{w}$ of $\calE$ to the Berkovich open unit disc
$\DD_{0,w}$ over $\ell(w)$.

Let $z_n$ be the center of $v_n$ on $\Spec(R_n/\gothp_{v_n})$.
Let $\alpha_v \in \DD_{0,v_n}$ be the seminorm $e^{-v(\cdot)}$.
Define $\alpha_s$ for $s \in (0, -\log r(\alpha_v))$ as in
Definition~\ref{D:berkovich radii}.
\end{defn}

\begin{remark} \label{R:sketch}
In terms of the notation from Definition~\ref{D:not abhyankar}, we can now give
a possibly helpful summary of the rest of the proof of Lemma~\ref{L:local decomp2}.
The basic idea is to view the formal spectrum of $R_n\llbracket x_n \rrbracket$ as a family of
formal discs over $\Spec(R_n)$; however, one can give a more useful description of the situation
in Berkovich's language of nonarchimedean analytic spaces.

In Berkovich's theory, one associates to a commutative nonarchimedean Banach algebra
its \emph{Gel'fand transform}, which consists of all multiplicative seminorms
bounded above by the Banach norm. For instance, for $F$ a complete nonarchimedean field,
this construction applied to $F \langle x \rangle$ (the completion of $F[x]$ for the Gauss
norm) produces the closed unit disc $\DD_F$ as in Definition~\ref{D:closed disc}.

Equip $R_n \llbracket x_n \rrbracket$ with the $\rho$-Gauss norm for some $\rho \in (0,1)$.
The Gel'fand transform of $R_n \llbracket x_n \rrbracket$ then fibres over the Gel'fand
transform of $R_n$ for the trivial norm. The fibre over a semivaluation $w$ (or rather,
over the corresponding multiplicative seminorm $e^{-w(\cdot)}$) is a closed disc over
$\ell(w)$; taking the union over all $\rho$ gives the open unit disc over $\ell(w)$.

Imagine a two-dimensional picture in which the Gel'fand transform of $R_n$ is oriented horizontally,
while the fibres over it are oriented vertically.
Using the analysis of $\nabla$-modules on Berkovich discs from \S~\ref{sec:berkovich}, 
we can control the spectral behavior of $\del_n$ on a single fibre. In particular, within the fibre over $v_n$, we obtain good control in a neighborhood of $v$. To make more progress, however,
we must combine this vertical information with some horizontal information. We do this by picking another point in the fibre over $v_n$ at which we have access to the induction hypothesis.
This gives good horizontal control not just of the irregularity, but of the variation of the
individual components of the scale multiset of $\del_n$. (This is needed because we may not have
enough continuous derivations on $\ell(v_n)$ to control the irregularity along the fibre over $v_n$.) We ultimately combine the horizontal and vertical information to control the behavior
of $\calE$ over a neighborhood of $v_n$ in the Gel'fand transform of 
$R_n \llbracket x_n \rrbracket$. This control leads to a proof of Lemma~\ref{L:local decomp2}.
\end{remark}

We now set about the program outlined in Remark~\ref{R:sketch}. We first give 
a refinement of Definition~\ref{D:replacing} which
respects the notation of Definition~\ref{D:not abhyankar}.
\begin{defn} \label{D:replacing2}
Given an instance of Hypothesis~\ref{H:not abhyankar},
by \emph{modifying/altering the input data on $R_n$}, we will mean
performing a sequence of operations of the following form.
\begin{itemize}
\item
Let $R_n \to R'_n$ be a morphism of finite type  to another nondegenerate
differential domain such that $\Frac(R'_n)$ is finite over $\Frac(R_n)$,
$\gothp_{v_n}$ has a positive but finite number of preimages in $\Spec(R'_n)$,
and the conclusion of Lemma~\ref{L:not abhyankar} holds
for some extension $v'_n$ of $v_n$. That is, there must exist
a regular sequence of parameters $x'_1,\dots,x'_{n-1}$ in $R'_n$ at 
the center of $v'_n$, such that for $m = \ratrank(v)$,
the inverse image of $Z$ in $\Spec(R'_n)$ is the zero locus of
$x'_1 \cdots x'_m$, and $v(x'_1),\dots,v(x'_m)$ are linearly independent over $\QQ$.
\item
Choose a finite list of generators $y_1,\dots,y_h$ of $R'_n$ over $R_n$. For each generator $y_j$,
choose a lift $\tilde{y}_j$ of $y_j$ in $R'_n \llbracket x_n \rrbracket$
which is integral over $\Frac(R)$. Let $R'$ be the ring obtained by
adjoining $\tilde{y}_1,\dots,\tilde{y}_h$ to $R$; we may identify the $x_n$-adic completion
of $R'$ with $R'_n \llbracket x_n \rrbracket$. 
Apply Lemma~\ref{L:Berkovich surjective}
to obtain an extension $v'$ of the semivaluation $v$ to $R'_n \llbracket x_n \rrbracket$.
By restriction, we obtain a true valuation on $R'$.
\item
Replace $R$ with $R'$ and $v$ with $v'$. Replace $x_1,\dots,x_{n-1}$
with any lifts of $x'_1,\dots,x'_{n-1}$ to $R'$. 
\end{itemize}
We will distinguish this operation from \emph{modifying/altering
the input data on $R$},
as the latter 
must be done carefully in order to have any predictable effect on $R_n$
and the other structures introduced in Definition~\ref{D:not abhyankar}.
\end{defn}

We next collect some horizontal information,
by extracting consequences from the induction hypothesis
on transcendence defect.
\begin{lemma} \label{L:invoke hypothesis}
Assume Hypothesis~\ref{H:not abhyankar}.
Choose $s \in (0, +\infty) \setminus (\Gamma_v \otimes_\ZZ \QQ)$.
Let $v^s$ be the valuation on $R$ induced from the
$s$-Gauss semivaluation on $R_n \llbracket x_n \rrbracket$
(relative to $v_n$).
\begin{enumerate}
\item[(a)]
There exists a finitely generated integral $R$-algebra
$R^s$ with $\Frac(R^s)$ finite over $\Frac(R)$,
such that $v^s$ admits a centered extension to $R^s$ with
generic center $z^s$,
and $\calE$ admits a good decomposition at $z^s$.
\item[(b)]
For a suitable choice of $R^s$, there exist
$\psi_1,\dots,\psi_{d^2} \in \Frac(R^s)$ such that
for any real valuation $w$ on $R$ admitting a centered extension to $R^s$
with generic center $z^s$,
the scale multiset of $\del_n$ on $\End(\calE)$ computed with respect
to $e^{-w(\cdot)}$ equals $e^{-w(\psi_1)}, \dots, e^{-w(\psi_{d^2})}$.
\item[(c)]
For a suitable choice of $R^s$,
the completion $\widehat{\calE}_{z^s}$ of $\calE$ at $z^s$
admits a filtration whose successive quotients are of rank $1$,
such that for $i =1,\dots,\rank(\calE)-1$,
the step of the filtration having rank $i$ has top exterior power
equal to the image of some endomorphism of $\wedge^i \widehat{\calE}_{z^s}$.
\end{enumerate}
\end{lemma}
\begin{proof}
To deduce (a), 
note that $\height(v_n) = 1$ and $\ratrank(v_n) = \ratrank(v)$
by Lemma~\ref{L:not abhyankar}, so $\trdefect(v_n) = e-1$.
Then note that $\height(v^s) = 1$ and $\ratrank(v^s) = \ratrank(v_n) + 1$
by Lemma~\ref{L:same corank},
so $\trdefect(v^s) = e-1$. Hence the induction hypothesis
on transcendence defects may be invoked, yielding (a).
To deduce (b), use Proposition~\ref{P:descend good ring}.
To deduce (c), use Proposition~\ref{P:filtration}.
\end{proof}

We next collect some vertical information
from the analysis of $\nabla$-modules on Berkovich discs.
\begin{lemma} \label{L:make prepared1}
Assume Hypothesis~\ref{H:not abhyankar}.
After altering the input data, $\End(\wedge^j N_{v_n})$ is terminal
for $j=1,\dots,\rank(\calE)-1$.
\end{lemma}
\begin{proof}
Note that $\trdefect(v_n) < \trdefect(v)$, so
by Lemma~\ref{L:same corank}, $\alpha_v$ is a point of $\DD_{0,v_n}$ of type (i) or (iv).
By Proposition~\ref{P:terminal}, 
for some $s_0 \in [0, -\log r(\alpha_v))$,
$\End(\wedge^j N_{v_n})$ becomes terminal at $s_0$
for $j=1,\dots,\rank(\calE)-1$.

Choose $s_1 \in (s_0, -\log r(\alpha_v))$.
Write $\alpha_{s_1} = \alpha_{z,r}$ for $r = e^{-s_1}$ and 
some $z$ in a finite extension
of $\Frac(R_n)$ with $v_n(z) > 0$. By altering the input data on $R_n$, we 
can force $z \in R_n$.
Note that $R \cap R_n$ is the kernel of $\del_n$ on $R$,
which is dense in $R_n$ with respect to $v_n$ because
it contains $x_1,\dots,x_{n-1}$.
Thus we can in fact choose $z \in R \cap R_n$,
then modify 
the input data on $R$ by replacing $x_n$ with $x_n - z$. At this point,
we may now take $z = 0$.

After altering the input data on $R_n$, we can produce $h \in R \cap R_n$
with $v(h) = s_0$.
We may then modify the input data on $R$, by replacing $x_n$ by $x_n/h$,
to achieve the desired result.
\end{proof}

We now begin to mix the horizontal and vertical information.
We first use vertical information to refine our last horizontal statement
(Lemma~\ref{L:invoke hypothesis}), as follows.
\begin{lemma} \label{L:invoke hypothesis2}
Assume Hypothesis~\ref{H:not abhyankar},
and assume that $\End(N_{v_n})$ is terminal.
Choose $s \in (0, +\infty) \setminus (\Gamma_v \otimes_\ZZ \QQ)$
and an open neighborhood $I$ of $s$ in $(0, +\infty)$.
After altering the input data on $R_n$,
we may choose $R^s$ and $\psi_1,\dots,\psi_{d^2}$
as in the conclusions of Lemma~\ref{L:invoke hypothesis},
satisfying the following additional conditions.
\begin{enumerate}
\item[(a)]
We have $\psi_1,\dots,\psi_{d^2} \in (R \cap R_n) \cup \{x_n\}$.
\item[(b)]
For any centered real semivaluation $w_n$ on $R_n$
with generic center $z_n$, normalized so that $w_n(x_1 \cdots x_{n-1}) = v_n(x_1 \cdots x_{n-1})$,
there exists $s' \in I$ such that
\[
f_i(N_{w_n}, s') = \begin{cases} w_n(\psi_i) & (\psi_i \in R \cap R_n) \\
s' & (\psi_i = x_n)
\end{cases} \qquad (i=1,\dots,d^2).
\]
\end{enumerate}
\end{lemma}
\begin{proof}
Set notation as in Lemma~\ref{L:invoke hypothesis}.
After altering the input data on $R_n$,
for $i=1,\dots,d^2$, if $f_i(\End(N_{v_n}), \alpha_v,s)$ is constant,
we can find an element $\phi_i$ of $R \cap R_n$ such that $v_n(\phi_i)
= f_i(\End(N_{v_n}), \alpha_v, 0)$. 
For $i$ for which $f_i(\End(N_{v_n}),\alpha_v,s) = s$
identically, we instead put $\phi_i = x_n$.
By permuting the $\psi_i$ appropriately, we may ensure that
$v^s(\phi_i) = v^s(\psi_i)$ for $i=1,\dots,d^2$.
By replacing $R^s$ by a suitable modification,
we can ensure that $w(\phi_i) = w(\psi_i)$ for all $w \in \RZ(R^s)$.
We may thus replace the $\psi_i$ with the $\phi_i$ hereafter;
this yields (a).

By modifying $R^s$, we can ensure that for 
any centered real valuation $w_n$ on $R_n$
normalized such that $w_n(x_1 \cdots x_{n-1}) = 1$,
any centered real semivaluation $w$ on $R^s$ extending $w_n$
satisfies $w(x_n) \in I$.
On the other hand, 
by Proposition~\ref{P:open map}, the image of $\RZ(R^s)$ in $\RZ(R_n)$
is open. Hence by modifying the input data on $R_n$, we may thus ensure that
$\RZ(R^s)$ surjects onto $\RZ(R_n)$. 
These two assertions together yield (b).
\end{proof}

We now turn around and use this improved horizontal information to refine our
last vertical assertion (Lemma~\ref{L:make prepared1}), so that it applies not
just at the semivaluation $v_n$ but also in a neighborhood thereof.
\begin{lemma} \label{L:make prepared2}
Assume Hypothesis~\ref{H:not abhyankar}.
After altering the input data,
for any centered real valuation $w_n$ on $R_n$ with generic center $z_n$,
$\End(\wedge^j N_{w_n})$ is terminal for $j=1,\dots,\rank(\calE) - 1$.
\end{lemma}
\begin{proof}
By Lemma~\ref{L:make prepared1}, we may assume that $\End(\wedge^j
N_{v_n})$ is terminal for $j=1,\dots,\rank(\calE) - 1$.
Choose $s_0 \in (0, -\log r(\alpha_v)) \cap (\Gamma_v \otimes_\ZZ \QQ)$
such that $\alpha_{0,s_0} \geq \alpha_v$.
Choose an open subinterval $I$ of $[0, s_0]$ 
on which $f_i(\End(N_{v_n}), s)$ is affine for $i=1,\dots,d^2$.
Choose three nonempty open subintervals $I_1, I_2, I_3$ of $I$
such that for any $s_j \in I_j$, we have $s_1 < s_2 < s_3$.

For $j = 1, 2, 3$, choose $s_j \in I_j \setminus (\Gamma_v \otimes_\ZZ \QQ)$;
this is possible because $\Gamma_v$ has finite rational rank.
By Lemma~\ref{L:invoke hypothesis2},
after  altering the input data on $R_n$,
for any centered real valuation $w_n$ on $R_n$ with generic center $z_n$,
normalized such that $w_n(x_1 \cdots x_{n-1}) = v_n(x_1 \cdots x_{n-1})$,
there exists $s'_j \in I_j$
such that $f_i(N_{w_n}, s'_j) = f_i(N_{v_n}, s'_j)$
for $i=1,\dots,d^2$.
By Proposition~\ref{P:preparedness}, 
$N_{w_n}$ becomes terminal at $s'_1$, and hence
also at $s_0$. By a similar argument, $\wedge^j N_{w_n}$ also becomes
terminal at $s_0$ for $j=2,\dots,\rank(\calE)-1$.

After altering the input data on $R_n$, we can produce $h \in R \cap R_n$
with $v(h) = s_0$.
We may then modify the input data on $R$, by replacing $x_n$ by $x_n/h$,
to achieve the desired result.
\end{proof}

We now combine horizontal and vertical information once more to obtain potential
good formal structures.
\begin{lemma} \label{L:target}
Under Hypothesis~\ref{H:not abhyankar}, 
suppose that for any centered real valuation $w_n$ on $R_n$
with generic center $z_n$,
$\End(\wedge^j N_{w_n})$ is terminal for $j=1,\dots,\rank(\calE)-1$. Then
$\calE$ admits a potential good formal structure at $v$.
\end{lemma}
\begin{proof}
Pick any centered height 1 Abhyankar valuation $w$ on $R^s$ with generic
center $z^s$, such that the restriction $w_n$ of $w$ to $R_n$ has generic center $z_n$.
Normalize the embedding of $\Gamma_w$ into $\RR$ so that
$w_n(x_1 \cdots x_{n-1}) = v_n(x_1 \cdots x_{n-1})$,
then let $\alpha_w \in \DD_{0,w_n}$ be the point corresponding to $w$.
Note that by Lemma~\ref{L:invoke hypothesis2},
on some closed interval
containing $-\log r(\alpha_w)$ in its interior,
$\End(\wedge^j N_{w_n})$ becomes strongly terminal for $j=1,\dots,
\rank(\calE)-1$.
Let $\widehat{R_n}$ be the completion of $R_n$ at $z_n$,
and let $\widehat{R^s}$ be the completion of $R^s$ at $z^s$.
By Remark~\ref{R:extend valuation},
$w$ extends to a centered real valuation on $\widehat{R^s}$.

For $j = 1,\dots,\rank(\calE)-1$,
let $\bv_j \in \End(\wedge^j \widehat{\calE}_{z^s})$
be the horizontal element corresponding to the endomorphism
of $\wedge^j \widehat{\calE}_{z^s}$ 
described in Lemma~\ref{L:invoke hypothesis}(c).
By Proposition~\ref{P:extend section},
$\bv_j$ belongs to
$\End(\wedge^j \calE) \otimes \ell(w_n)\langle x_n/r \rangle$.
On the other hand, it also belongs to
$\End(\wedge^j \calE) \otimes \widehat{R^s}$; by Remark~\ref{R:locally free},
we thus find it in
$\End(\wedge^j \calE) \otimes S$ for $S$ a complete subring of
$\ell(w_n) \langle r/x_n, x_n/r \rangle$ which is topologically finitely
generated over $R_n \langle r/x_n, x_n/r \rangle$, such that
$\Frac(S)$ is finite over $\Frac(R_n \langle r/x_n, x_n/r \rangle)$
(which may be chosen independently of $j$).
By Proposition~\ref{P:irrational glueing},
$\bv_j$ belongs to
$\End(\wedge^j \calE) \otimes 
R_n' \langle r/x_n,x_n/r \rangle$ for some topologically finitely generated 
$\widehat{R_n}$-algebra $R'_n$ 
such that $\Frac (R'_n)$ is finite over $\Frac(\widehat{R_n})$
(again chosen independently of $j$).
By Remark~\ref{R:locally free},
$\bv_j$ belongs to
$\End(\wedge^j \calE) \otimes R'_n \langle x_n/r \rangle$,
and in particular to
$\End(\wedge^j \calE) \otimes R'_n \llbracket x_n \rrbracket$.

This last conclusion is stable under altering the input data on $R_n$.
By applying Proposition~\ref{P:descend1}, we may alter the input data on $R_n$
so that $\bv_j$ belongs to
$\End(\wedge^j \calE) \otimes 
\widehat{R_n}\llbracket x_n \rrbracket[x_1^{-1},\dots,x_{n-1}^{-1}]$
for $j=1,\dots,\rank(\calE)-1$.
We obtain a filtration of
$\calE \otimes 
\widehat{R_n}\llbracket x_n \rrbracket[x_1^{-1},\dots,x_{n-1}^{-1}]$
with successive quotients of rank 1 by taking the step of rank $j$
to contain elements which wedge to 0 with the image of $\bv_j$.
By Proposition~\ref{P:potential good}, $\calE$ admits a potential
good formal structure at $v$, as desired.
\end{proof}

By combining Lemma~\ref{L:make prepared2}
with Lemma~\ref{L:target}, we deduce Lemma~\ref{L:local decomp2}.

\begin{remark}
Note that in the proof of Lemma~\ref{L:local decomp2},
we arrive easily at the situation where $\End(N_{v_n})$ is terminal,
but it takes more work to reach the situation where
$\End(N_{w})$ is terminal for any centered real valuation $w$ on $R_n$
with generic center $z_n$.
This extra work is not needed in case $v_n$ itself extends to a real valuation
on the completion of $R_n$ at $z_n$, but this does not always occur;
see Remark~\ref{R:extend valuation}.
\end{remark}

\subsection{Increasing the height}

We finally construct good potential formal structures
for valuations of height greater than 1.
This argument is loosely modeled on \cite[Theorem~4.3.4]{kedlaya-part2}.

\begin{lemma} \label{L:local decomp3}
Let $h> 1$ be an integer.
Suppose that for any instance of Hypothesis~\ref{H:geometric}
with $\height(v) < h$,
$\calE$ admits a potential good formal structure at $v$.
Then for any instance of Hypothesis~\ref{H:geometric}
with $\height(v) = h$,
$\calE$ admits a potential good formal structure at $v$.
\end{lemma}
\begin{proof}
Choose a nonzero proper isolated subgroup of $\Gamma_v$,
then define $v', \overline{v}$ as in Definition~\ref{D:composite}.
Note that $\height(v')$ and $\height(\overline{v})$ are both
positive and their sum is $\height(v) = h$, so both are at most $h-1$.
By the induction hypothesis, $\calE$ admits a potential good
formal structure at $v'$; in particular,
after altering the input data, a good decomposition
exists at the generic center of $v'$.

After altering the input data and possibly enlarging $Z$
(which is harmless by Remark~\ref{R:enlarge Z}), 
we may set notation as in Lemma~\ref{L:set notation}
in such a way that for some $r$,
the center of $v'$ on $X$ is the zero locus of $x_{r+1},\dots,x_n$.
Let $R'$ be the completion of $R$ for the ideal $(x_1,\dots,x_r)$.
By Corollary~\ref{C:complete nondegenerate}, we may write
$R' \cong R_1 \llbracket x_1,\dots,x_r \rrbracket$, where $R_1$
is the joint kernel of $\del_{1},\dots,\del_r$ on $R'$.
Put $K = \Frac(R_1)$, so that by construction
\[
\calE \otimes K\llbracket x_{1},\dots,x_r \rrbracket
[x_{1}^{-1},\dots,x_r^{-1}]
\]
admits a minimal good decomposition.
By Theorem~\ref{T:descend}, this decomposition descends to
a minimal good decomposition of
\[
\calE \otimes R_{r,r}^\dagger(R_1).
\]
With notation as in \eqref{eq:local model1}, 
by the last assertion of Proposition~\ref{P:descend2},
each $\calR_\alpha$
admits a $K$-lattice stable under
the action of $x_1 \del_1,\dots,x_r \del_r$.
This gives a collection of instances of Hypothesis~\ref{H:geometric}
with $R$ replaced by $R_1$ and $v$ replaced by $\overline{v}$.
Again by the induction hypothesis, after altering the input data
(and lifting from $R_1$ to $R$, as in Definition~\ref{D:replacing2}),
we obtain good decompositions of each of these lattices.

Putting this all together, we obtain an admissible but possibly not good 
decomposition of $\calE$ at $z$. By Proposition~\ref{P:potential good},
this suffices to imply that $\calE$ admits a potential good formal structure
at $v$.
\end{proof}

\section{Good formal structures after modification}
\label{sec:good formal}

To conclude, we extract from Theorem~\ref{T:higher HLT} a global
theorem on the existence of good formal structures for formal flat
meromorphic connections on nondegenerate differential schemes,
after suitable blowing up.
We also give partial results in the cases of formal completions
of nondegenerate differential schemes and complex analytic varieties.

\subsection{Local-to-global construction of good formal structures}
\label{subsec:good}

Using the compactness of Riemann-Zariski spaces, we are able to
pass from the valuation-local Theorem~\ref{T:higher HLT}
to a more global theorem on construction of good formal
structures after a blowup, in the case of an algebraic connection.

\begin{hypothesis} \label{H:good formal}
Throughout \S~\ref{subsec:good}, let
$X$ be a nondegenerate integral differential scheme,
and let $Z$ be a closed proper subscheme of $X$.
Let $\calE$ be a $\nabla$-module over $\calO_X(*Z)$.
\end{hypothesis}

\begin{lemma} \label{L:local}
Let $v$ be a centered valuation on $X$.
Then there exist a modification $f_v: X_v \to X$, a centered extension $w$ of $v$
to $X_v$,  and an open subset $U_v$ of $X_v$ on which $w$ is centered,
such that $f_v^* \calE$ admits a good formal structure at each point of 
$U_v$.
\end{lemma}
\begin{proof}
By Theorem~\ref{T:higher HLT},
we can choose data as in the statement of the lemma so that
$f_v^* \calE$ admits a good formal structure at the generic center 
of $v$ on $U_v$. 
(Note that Proposition~\ref{P:potential good} ensures that we can
choose $f$ to be a modification, not just an alteration.)
By Proposition~\ref{P:open locus}, this implies that 
we can rechoose $U_v$ so that $\calE$
admits a good formal structure at each point of $U_v$.
\end{proof}

\begin{theorem} \label{T:global}
There exists a modification $f: Y \to X$ such that
$(Y,W)$ is a regular pair for $W = f^{-1}(Z)$, and
$f^* \calE$ admits a good formal structure at each point of $Y$.
\end{theorem}
\begin{proof}
For each valuation $v \in \RZ(X)$,
set notation as in Lemma~\ref{L:local}.
Since $v \in \RZ(U_v)$ by construction,
the sets $\RZ(U_v)$ cover $\RZ(X)$.
Since $\RZ(X)$ is quasicompact by Theorem~\ref{T:Zariski},
we can choose finitely many valuations
$v_1,\dots,v_n \in \RZ(X)$ 
such that the sets $T_i = \RZ(U_{v_i})$ for $i=1,\dots,n$
cover $\RZ(X)$. Put $f_i = f_{v_i}$ and $X_i = X_{v_i}$.
By applying
Theorem~\ref{T:desing1} to the unique component of $X_1 \times_X \cdots
\times_X X_n$ which dominates $X$,
we construct a modification $f: Y \to X$ factoring through each $X_i$,
such that $(Y, f^{-1}(Z))$ is a regular pair.

We now check that this choice of $f$ has the desired property.
For any $z \in Y$, we may choose a valuation $v \in \RZ(X)$ with generic
center $z$ on $Y$. For some $i$, we have $v \in T_i$,
so $f_i^* \calE$ admits a good formal structure at the generic center of $v$
on $X_i$. That point is the image of $z$ in $X_i$, so
$f^* \calE$ admits a good formal structure at $z$, as desired.
\end{proof}

\begin{remark}
For $X$ an algebraic variety over an algebraically
closed field of characteristic $0$,
Theorem~\ref{T:global} reproduces a result of Mochizuki
\cite[Theorem~19.5]{mochizuki2}.
(More precisely, one must apply Theorem~\ref{T:global} to both
$\calE$ and $\End(\calE)$, due to the discrepancy between our notion
of good formal structures and Mochizuki's definition. See
again \cite[Remark~4.3.3, Remark~6.4.3]{kedlaya-goodformal1}.)
Mochizuki's argument is completely different from ours:
he uses analytic methods to reduce to the case of meromorphic
connections on surfaces, which he had previously treated 
\cite[Theorem~1.1]{mochizuki} using positive-characteristic arguments.
\end{remark}

\subsection{Formal schemes and analytic spaces}
\label{subsec:formal}

From Theorem~\ref{T:global}, we obtain a corresponding result for
formal completions of nondegenerate schemes. We also obtain a somewhat weaker result
for formal completions of complex analytic spaces. We do not obtain the best possible
result in the analytic case; see Remark~\ref{R:functorial}.

\begin{theorem} \label{T:global formal1}
Let $X$ be a nondegenerate integral differential scheme,
let $Z$ be a closed proper subscheme of $X$,
and let $\widehat{X|Z}$ be the formal completion of $X$
along $Z$.
Let $\calE$ be a $\nabla$-module over $\calO_{\widehat{X|Z}}(*Z)$.
Then there exists a modification $f: Y \to X$ such that
$(Y,W)$ is a regular pair for $W = f^{-1}(Z)$, and
$f^* \calE$ admits a good formal structure at each point of $W$.
\end{theorem}
\begin{proof}
We first consider the case where $X$ is affine.
Put $X = \Spec(R)$ and $Z = \Spec (R/I)$.
Let $\widehat{R}$ be the $I$-adic completion of $R$, and
put $\widehat{I} = I \widehat{R}$.
Put $\widehat{X} = \Spec(\widehat{R})$ and $\widehat{Z} = \Spec(\widehat{R}/\widehat{I})$.
We can then view $\calE$ as a $\nabla$-module on $\calO_{\widehat{X}}(*\widehat{Z})$,
and apply Theorem~\ref{T:global} to deduce the claim.

We now turn to the general case.
Since $X$ is integral and  noetherian, it is
covered by finitely many dense open affine subschemes $U_1,\dots,U_n$.
For $i=1,\dots,n$, we may apply the previous paragraph to construct
a modification $f_i: Y_i \to U_i$ such that $(Y_i, W_i)$
is a regular pair for $W_i = f_i^{-1}(U_i \cap Z)$,
and $f_i^* \calE$ admits a good formal structure at each point of $W_i$.
By taking the Zariski closure of the graph of $f_i$ within $Y_i \times_{\Spec \ZZ} X$, we may extend $f_i$ to a modification $f'_i: Y'_i \to X$.
By Theorem~\ref{T:desing1},
we may construct a modification $f: Y \to X$ factoring through
the fibred product of the $f'_i$,
such that $(Y,W)$ is a regular pair for $W = f^{-1}(Z)$.
This modification has the desired effect.
\end{proof}

\begin{theorem} \label{T:global formal2}
Let $X$ be a smooth (separated) complex analytic space.
Let $Z$ be a closed subspace of $X$ containing no irreducible component of $X$.
Let $\widehat{X|Z}$ be the formal completion of $X$ along $Z$.
Let $\calE$ be a $\nabla$-module over $\calO_{\widehat{X|Z}}(*Z)$.
For each $z \in Z$, there exist an open neighborhood $U$ of $z$
in $X$ and a modification $f: Y \to U$ such that
$(Y,W)$ is a regular pair for $W = f^{-1}(U \cap Z)$, and
$f^* \calE$ admits a good formal structure at each point of $W$.
\end{theorem}
\begin{proof}
We may reduce to Theorem~\ref{T:global} as in the proof of
Proposition~\ref{P:open locus}.
\end{proof}

In both the algebraic and analytic cases, we recover Malgrange's construction
of canonical lattices 
\cite[Th\'eor\`eme~3.2.2]{malgrange-reseau}.
\begin{theorem} \label{T:deligne-malgrange}
Let $X$ be either a nondegenerate integral differential scheme or a smooth
irreducible complex analytic space.
Let $Z$ be a closed proper subspace of $X$,
and let $\widehat{X|Z}$ be the formal completion of $X$
along $Z$.
Let $\calE$ be a $\nabla$-module over $\calO_{\widehat{X|Z}}(*Z)$.
Let $U$ be the open (by Proposition~\ref{P:open locus})
subspace of $X$ over which $(X,Z)$ is a regular pair
and $\calE$ has empty turning locus.
Then the Deligne-Malgrange lattice of $\calE|_U$ extends uniquely to a lattice of
$\calE$.
\end{theorem}
\begin{proof}
Let $j: \widehat{U|Z} \to \widehat{X|Z}$ be the inclusion. Let $\calE_0$ be the Deligne-Malgrange
lattice of $\calE_U$; it is sufficient to check
that $j_* \calE_0$ is coherent over $\calO_{\widehat{X|Z}}$. This may be checked locally
around a point $z \in Z$. After replacing $X$ by a suitable neighborhood of $z$,
we may apply Theorem~\ref{T:global formal1} or Theorem~\ref{T:global formal2} to 
construct a modification $f: Y \to X$ such that $(Y, f^{-1}(Z))$ is a regular pair
and $f^* \calE$ has empty turning locus. By Theorem~\ref{T:DM1}, $f^* \calE$ admits
a Deligne-Malgrange lattice $\calE_1$.
We then have $j_* \calE_0 = f_* \calE_1$, which is coherent because 
$f$ is proper (by \cite[Th\'eor\`eme~3.2.1]{ega3-1} in the algebraic case,
and \cite[\S 6, Hauptsatz I]{grauert} in the analytic case),
This proves the claim.
\end{proof}

\begin{remark}
As noted in \cite[Remarque~3.3.2]{malgrange-reseau},
the lattice constructed in Theorem~\ref{T:deligne-malgrange} is reflexive, and hence
locally free in codimension 2.
\end{remark}

\begin{remark} \label{R:functorial}
In Theorem~\ref{T:global formal2},
we would prefer to give a \emph{global} modification rather than a local modification
around each point of $Z$.
The obstruction to doing so is that while Theorem~\ref{T:global} gives
a procedure for constructing a suitable modification, the procedure is not functorial
for open immersions. Such functoriality is necessary to glue the local modifications;
we cannot instead imitate the proof of Theorem~\ref{T:global formal1}
because the complex analytic topology
is too fine to admit Zariski closures.

In a subsequent paper, we plan to describe a modification procedure which is
functorial for all \emph{regular} morphisms of nondegenerate differential schemes.
For this, one needs a form of embedded resolution of singularities for excellent schemes
which is functorial for regular morphisms. Fortunately, such a result has recently been
given by Temkin \cite{temkin2}, based on earlier work of 
Bierstone, Milman, and Temkin \cite{bmt, temkin1}.
\end{remark}

\end{document}

%% file: gf2-fig1.pstex_t
\begin{picture}(0,0)%
\includegraphics{gf2-fig1.pstex}%
\end{picture}%
\setlength{\unitlength}{4144sp}%
\begingroup\makeatletter\ifx\SetFigFont\undefined%
\gdef\SetFigFont#1#2#3#4#5{%
  \reset@font\fontsize{#1}{#2pt}%
  \fontfamily{#3}\fontseries{#4}\fontshape{#5}%
  \selectfont}%
\fi\endgroup%
\begin{picture}(4710,3070)(5251,-5120)
\put(5266,-2221){\makebox(0,0)[lb]{\smash{{\SetFigFont{12}{14.4}{\rmdefault}{\mddefault}{\updefault}{\color[rgb]{0,0,0}$y$}%
}}}}
\put(9946,-5056){\makebox(0,0)[lb]{\smash{{\SetFigFont{12}{14.4}{\rmdefault}{\mddefault}{\updefault}{\color[rgb]{0,0,0}$s$}%
}}}}
\put(6076,-4606){\makebox(0,0)[lb]{\smash{{\SetFigFont{12}{14.4}{\rmdefault}{\mddefault}{\updefault}{\color[rgb]{0,0,0}$y=s$}%
}}}}
\put(8731,-2761){\makebox(0,0)[b]{\smash{{\SetFigFont{12}{14.4}{\rmdefault}{\mddefault}{\updefault}{\color[rgb]{0,0,0}$y = f_i(M, \alpha,s)$ }%
}}}}
\end{picture}%

%% file: goodformal2.bbl
\begin{thebibliography}{99}

\bibitem{ahv1}
J.-M. Aroca, H. Hironaka, and J.-L. Vicente,
\textit{The theory of the maximal contact},
Memorias Mat. Inst. Jorge Juan 29, 
Consejo Superior de Investigaciones Cient\'ificas,
Madrid, 1975.

\bibitem{ahv2}
J.-M. Aroca, H. Hironaka, and J.-L. Vicente,
\textit{Desingularization theorems}. 
Memorias Mat. Inst. Jorge Juan 30, 
Consejo Superior de Investigaciones Cient\'ificas,
Madrid, 1975.

\bibitem{baldassarri}
F. Baldassarri, Continuity of the radius of convergence
of differential equations on $p$-adic analytic curves,
arXiv preprint \texttt{0809.2479v6} (2010).

\bibitem{baldassarri-divizio}
F. Baldassarri and L. Di Vizio, Continuity of the radius of convergence
of $p$-adic differential equations on Berkovich analytic spaces, 
arXiv preprint \texttt{0709.2008v3} (2008).

\bibitem{berkovich}
V.G. Berkovich, \textit{Spectral theory and analytic geometry
over non-archimedean fields} (translated by N.I. Koblitz), 
Math. Surveys and Monographs 33, Amer. Math. Soc., Providence, 1990.

\bibitem{bierstone-milman}
E. Bierstone and P. D. Milman, 
Canonical desingularization in characteristic zero by blowing 
up the maximum strata of a local invariant,
\textit{Invent. Math.} \textbf{128} (1997), 207--302.

\bibitem{bmt}
E. Bierstone, P. Milman, and M. Temkin,
$\mathbb{Q}$-universal desingularization,
arXiv preprint \texttt{0905.3580v1} (2009).

\bibitem{bgr}
S. Bosch, U. G\"untzer, and R. Remmert, \textit{Non-Archimedean analysis},
Grundlehren der Math.\ Wiss.\ 261, Springer-Verlag, Berlin, 1984.

\bibitem{boucksom-favre-jonsson}
S. Boucksom, C. Favre, and M. Jonsson, Valuations and plurisubharmonic
singularities, \textit{Publ. RIMS, Kyoto Univ.} \textbf{44} (2008), 449--494.

\bibitem{christol}
G. Christol, \textit{Modules Diff\'erentiels et \'Equations Diff\'erentielles
$p$-adiques}, Queen's Papers in Pure and Applied Math. 66, Queen's Univ.,
Kingston, 1983.

\bibitem{dejong}
A.J. de Jong, Smoothness, semi-stability and alterations,
\textit{Publ. Math. IH\'ES} \textbf{83} (1996), 51--93.

\bibitem{deligne}
P. Deligne, \textit{\'Equations Diff\'erentielles \`a Points Singuliers
R\'eguliers}, Lecture Notes in Math.\ 163, Springer-Verlag, 1970.

\bibitem{eisenbud}
D. Eisenbud, \textit{Commutative Algebra with a View Toward Algebraic
Geometry}, Springer-Verlag, New York, 1995.

\bibitem{fulton-harris}
W. Fulton and J. Harris, \textit{Representation Theory: A First Course},
Graduate Texts in Math. 129, Springer, New York, 1996.

\bibitem{grauert}
H. Grauert, Ein Theorem der analytischen Garbentheorie und die
Modulr\"aume complexer Structuren,
\textit{Publ. Math. IH\'ES} \textbf{5} (1960).

\bibitem{ega2}
A. Grothendieck, \'El\'ements de g\'eom\'etrie alg\'ebrique. II.
\'Etude globale \'el\'ementaire de quelques classes de morphismes,
\textit{Publ. Math. IH\'ES} \textbf{8} (1961), 5--222.

\bibitem{ega3-1}
A. Grothendieck, \'El\'ements de g\'eom\'etrie alg\'ebrique. III.
\'Etude cohomologique des faisceaux coh\'erents, premi\`ere partie,
\textit{Publ. Math. IH\'ES} \textbf{11} (1961), 5--167.

\bibitem{ega4-2}
A. Grothendieck, \'El\'ements de g\'eom\'etrie alg\'ebrique. IV. 
\'Etude locale des sch\'emas et des morphismes de sch\'emas,
deuxi\`eme partie, 
\textit{Publ. Math. IH\'ES} \textbf{24} (1965), 5--231.

\bibitem{ega4-3}
A. Grothendieck, \'El\'ements de g\'eom\'etrie alg\'ebrique: IV.
\'Etude local des sch\'emas et des morphismes de sch\'emas,
troisi\`eme partie, \textit{Publ. Math. IH\'ES} \textbf{28} (1966),
5--255.

\bibitem{sga1}
A. Grothendieck et al., \textit{SGA 1: Rev\^etements \'Etales et Groupe
Fondamental}, corrected edition, Documents Math\'ematiques 3, Soc. Math. France, Paris, 2003.

\bibitem{hironaka-flattening}
H. Hironaka, Flattening theorem in complex-analytic geometry,
\textit{Amer. J. Math.} \textbf{97} (1975), 503--547.

\bibitem{kedlaya-part1}
K.S. Kedlaya, Semistable reduction for overconvergent $F$-isocrystals, I:
    Unipotence and logarithmic extensions,
\textit{Compos. Math.} \textbf{143} (2007), 1164--1212.

\bibitem{kedlaya-part2}
K.S. Kedlaya, Semistable reduction for overconvergent $F$-isocrystals, II:
A valuation-theoretic approach, 
\textit{Compos. Math.} \textbf{144} (2008), 657--672. 

\bibitem{kedlaya-part3}
K.S. Kedlaya, Semistable reduction for overconvergent $F$-isocrystals, III:
Local semistable reduction at monomial valuations, 
\textit{Compos. Math.} \textbf{145} (2009), 143--172. 

\bibitem{kedlaya-goodformal1}
K.S. Kedlaya, Good formal structures for flat meromorphic connections, I:
Surfaces, arXiv:0811.0190v4 (2009); to appear in \textit{Duke Math. J.}

\bibitem{kedlaya-part4}
K.S. Kedlaya, Semistable reduction for overconvergent $F$-isocrystals, IV:
Local semistable reduction at nonmonomial valuations, arXiv preprint 
\texttt{0712.3400v3} (2009).

\bibitem{kedlaya-course}
K.S. Kedlaya, \textit{$p$-adic Differential Equations},
Cambridge Studies in Advanced Math. 125, Cambridge Univ. Press,
2010.

\bibitem{kedlaya-xiao}
K.S. Kedlaya and L. Xiao, Differential modules on $p$-adic polyannuli,
\textit{J. Institut Math. Jussieu} \textbf{9} (2010), 155--201;
errata, \textit{ibid.} \textbf{9} (2010), 669--671.

\bibitem{knaf-kuhlmann}
H. Knaf and F.-V. Kuhlmann, Abhyankar places admit local uniformization
in any characteristic, \textit{Ann. Scient. \'Ec. Norm. Sup.}
\textbf{38} (2005), 833--846.

\bibitem{majima}
H. Majima, \textit{Asymptotic Analysis for Integrable Connections with Irregular
Singular Points}, Lecture Notes in Math.\ 1075, Springer-Verlag, Berlin, 1984.

\bibitem{malgrange-reseau}
B. Malgrange, Connexions m\'eromorphes 2: Le r\'eseau canonique,
\textit{Invent. Math.} \textbf{124} (1996), 367--387.

\bibitem{matsumura-alg}
H. Matsumura, \textit{Commutative Algebra},
second edition, Math. Lect. Series 56, Benjamin/Cummings, Reading, 1980.

\bibitem{matsumura}
H. Matsumura, \textit{Commutative Ring Theory} (translated from the Japanese
by M. Reid), second edition, Cambridge Studies in Advanced Math. 8, Cambridge Univ.
Press, Cambridge, 1989.

\bibitem{mochizuki}
T. Mochizuki, Good formal structure for meromorphic flat connections
on smooth projective surfaces, in
T. Miwa et al. (eds.), \textit{Algebraic Analysis and Around},
Advanced Studies in Pure Math. 54, Math. Soc. Japan, 2009, 223--253.

\bibitem{mochizuki2}
T. Mochizuki, Wild harmonic bundles and wild pure twistor
$D$-modules, arXiv preprint \texttt{0803.1344v3} (2009).

\bibitem{raynaud-gruson}
M. Raynaud and L. Gruson,
Crit\`eres de platitude et de projectivit\'e. Techniques de
``platification'' d'un module,
\textit{Invent. Math.} \textbf{13} (1971), 1--89.

\bibitem{rotthaus}
C. Rotthaus,
Zur Komplettierung ausgezeichneter Ringe,
\textit{Math. Ann.} \textbf{253} (1980), 213--226. 

\bibitem{sabbah}
C. Sabbah, \'Equations diff\'erentielles \`a points singuliers
irr\'eguliers et phenomene de Stokes en dimension 2,
\textit{Ast\'erisque} \textbf{263} (2000).

\bibitem{temkin-excellent}
M. Temkin, Desingularization of quasi-excellent schemes in characteristic zero,
\textit{Advances in Math.} \textbf{219} (2008), 488--522.

\bibitem{temkin-rz}
M. Temkin, Relative Riemann-Zariski spaces,
\textit{Israel J. Math.}, to appear;
arxiv preprint \texttt{0804.2843v1} (2008).

\bibitem{temkin-unif}
M. Temkin, Inseparable local uniformization, 
arXiv preprint \texttt{0804.1554v2} (2010).

\bibitem{temkin1}
M. Temkin, Functorial desingularization of quasi-excellent schemes in 
characteristic zero: the non-embedded case,
arXiv preprint \texttt{0904.1592v1} (2009).

\bibitem{temkin2}
M. Temkin, Functorial desingularization over $\mathbb{Q}$: boundaries and 
the embedded case, arXiv preprint \texttt{0912.2570v1} (2009).


\bibitem{vaquie}
M. Vaqui\'e, Valuations, in \textit{Resolution of singularities (Obergurgl, 1997)},
Progr. Math. 181, Birkh\"auser, Basel, 2000, 539--590.


\bibitem{varadarajan}
V.S. Varadarajan, Linear meromorphic differential equations: a modern
point of view, \textit{Bull. Amer. Math. Soc.} \textbf{33} (1996), 1--42.

\bibitem{zariski-local}
O. Zariski, Local uniformization on algebraic varieties,
\textit{Ann. of Math.} \textbf{41} (1940), 852--896.

\bibitem{zariski-samuel}
O. Zariski and P. Samuel, \textit{Commutative Algebra, Volume II},
reprint of the 1960 original, Graduate Texts in Math. 29, Springer-Verlag,
New York, 1976.

\end{thebibliography}
